\documentclass[11pt,reqno,twoside]{article}
\usepackage{amsmath,amsfonts,amsthm,amssymb}
\usepackage{graphics}
\usepackage[all]{xy}
\usepackage{color}
\theoremstyle{definition}
\newtheorem{theorem}{Theorem}[section]

\newtheorem{corollary}[theorem]{Corollary}
\newtheorem{lemma}[theorem]{Lemma}
\newtheorem{remark}[theorem]{Remark}

\newtheorem{definition}[theorem]{Definition}
\newtheorem{example}[theorem]{Example}

\xyoption{knot} \UseComputerModernTips \knottips{FF}
\def\Bracket#1{\mathord{\Bigl<\,~ \raise8pt\xybox{0;/r1.3pc/:#1}\,
\Bigr>}}

\def\rBracket#1{\mathord{\Bigl<\,~ \raise1pt\xybox{0;/r1.3pc/:#1}\,
~\Bigr>}}

\def\newBracket#1{\mathord{\Big\langle\, \raise8pt\xybox{0;/r1.3pc/:#1}\,
\Big\rangle}}

\def\doubleBracket#1{\mathord{\Bigl[\Bigl[\, \raise8pt\xybox{0;/r1.3pc/:#1}\,\Bigr]\Bigr]}}

\def\Fig#1{\mathord{~ \raise-10pt\xybox{0;/r0.15pc/:#1}\,~}}

\def\mBracket#1{\mathord{\,~ \raise10pt\xybox{0;/r0.2pc/:#1}\,~}}

\begin{document}

\title{\Large\bf 
Constructions of invariants for surface-links via link invariants and applications to the Kauffman bracket
}

\author{\small 
{Sang Youl Lee}
\smallskip\\
{\small\it 
Department of Mathematics, Pusan National University,
}\\ 
{\small\it Busan 46241, Republic of Korea}\\
{\small\it sangyoul@pusan.ac.kr}}

\renewcommand\leftmark{\centerline{\footnotesize 
S. Y. Lee}}
\renewcommand\rightmark{\centerline{\footnotesize 
Constructions of invariants for surface-links via link invariants
}}

\date{}

\maketitle

\begin{abstract}
 In this paper, we formulate a construction of ideal coset invariants for surface-links in $4$-space using invariants for knots and links in $3$-space. We apply the construction to the Kauffman bracket polynomial invariant and obtain an invariant for surface-links called the Kauffman bracket ideal coset invariant of surface-links. We also define a series of new invariants $\{{\mathbf K}_{2n-1}(\mathcal L) | n=2, 3, 4, \ldots\}$ for surface-links $\mathcal L$ defined by skein relations, which are more effective than the Kauffman bracket ideal coset invariant to distinguish given surface-links.
\end{abstract}

\noindent{\it Mathematics Subject Classification 2000}: 57Q45; 57M25.

\noindent{\it Key words and phrases}: ideal coset invariant; knotted surface; marked graph diagram; surface-link; invariant for surface-links; Kauffman bracket polynomial, skein relation.



\section{Introduction}
\label{intro}

By a {\it surface-link} (or {\it knotted surface}) we mean a closed 2-manifold smoothly (or piecewise linearly and locally flatly) embedded in the $4$-space $\mathbb R^4$ or $S^4$. Two surface-links are said to be {\it equivalent} if they are ambient isotopic. A {\it marked graph diagram} (or {\it ch-diagram}) is a link diagram possibly with some $4$-valent vertices equipped with markers: \xy (-4,4);(4,-4) **@{-}, 
(4,4);(-4,-4) **@{-},  
(3,-0.2);(-3,-0.2) **@{-},
(3,0);(-3,0) **@{-}, 
(3,0.2);(-3,0.2) **@{-}, 
\endxy.
S. J. Lomonaco, Jr. \cite{Lo} and K. Yoshikawa \cite{Yo} introduced a method of presenting surface-links by marked graph diagrams. Indeed, every surface-link $\mathcal L$ is presented by an admissible marked graph diagram $D$. Moreover, if $D$ is an admissible marked graph diagram presenting a surface-link $\mathcal L$, then we can construct a surface-link $\mathcal S(D)$ from $D$ such that $\mathcal S(D)$ is equivalent to $\mathcal L$. If $\mathcal L$ is an oriented surface-link, then it is presented by an admissible oriented marked graph diagram (see Section \ref{sect-omgr-osl}). Using marked graph diagram presentations of surface-links, some properties and invariants for surface-links have been studied by several researchers up to now. For example, see \cite{As,JKaL,JKL,KKL,KJL1,KJL2,Le1,Le2,Le3,So,Yo} and therein.

In \cite{Le3}, the author defined a polynomial, denoted by  $\ll D\gg$, for marked graph diagrams $D$ by a state-sum model associated with an arbitrary given link invariant for knots and links in $3$-space as its state evaluation, which is an invariant for marked graphs in the 3-space $\mathbb R^3$ presented by the diagram $D$ and satisfies a skein relation (see Section \ref{sect-poly-mgd}). In \cite{JKL}, Y. Joung, J. Kim and the author constructed {\it ideal coset invariants} for surface-links in 4-space by means of the polynomial $\ll D\gg$ for marked graph diagrams $D$, and applied the construction to the elementary classical link invariant $[K]:=A^{\|K\|-1}$, where $A$ is a variable and $\|K\|$ is the number of components of a classical link $K$ and obtained an ideal coset invariant for unoriented (nonorientable or orientable but not oriented) surface-links. 

In this paper, we formulate a construction of ideal coset invariants for oriented surface-links in $4$-space by means of the polynomial $\ll D\gg$ for oriented marked graph diagrams. When we forget orientation from this formulation, we get a refined construction of ideal coset invariants for unoriented surface-links given in \cite{JKL} with a simplification that is more applicable in practice. We present a way how to find a unique representative, namely, the {\it normal form} of a given ideal coset from the polynomial $\ll D\gg$ of an (resp.~oriented) marked graph diagram $D$ by using Groebner basis calculation on computer, which is an invariant for the (resp.~oriented) surface-link in the 4-space $\mathbb R^4$ or $S^4$ presented by the (resp.~oriented) marked graph diagram $D$. 
We apply this construction to the Kauffman bracket polynomial for unoriented knots and links and the normalized Kauffman bracket polynomial for oriented knots and links in $3$-space, which lead the {\it Kauffman bracket ideal coset invariant} for unoriented surface-links and the {\it normalized Kauffman bracket ideal coset invariant} for oriented surface-links, respectively. Further, by specializing variables of the polynomial $\ll\cdot\gg$ associated with the (resp.~normalized) Kauffman bracket polynomial, we define a series of new invariants $\{{\mathbf K}_{2n-1}(\mathcal L) | n=2, 3, 4, \ldots\}$ for (resp.~oriented) surface-links $\mathcal L$, which are more powerful than the (resp.~normalized) Kauffman bracket ideal coset invariant to distinguish given (resp.~oriented) surface-links. We also discuss various examples.

This paper is organized as follows. In Section \ref{sect-omgr-osl}, we review (oriented) marked graph diagram presentation of (oriented) surface-links. In Section \ref{sect-poly-mgd}, we deal with the polynomial invariants $\ll\cdot\gg$ for (resp.~oriented) marked graphs in the $3$-space $\mathbb R^3$ associated with classical (resp.~oriented) link invariants. In Section \ref{sect-ici-sl}, we construct ideal coset invariants for (oriented) surface-links and present how to find the normal form of a given ideal coset in terms of the polynomial $\ll\cdot\gg$ by using the Groebner basis theory. In Section \ref{sect-b-poly-MGD}, we apply the construction in Section \ref{sect-ici-sl} to the (resp.~normalized) Kauffman bracket polynomial and derive the (resp.~normalized) Kauffman bracket ideal coset invariant for (resp.~oriented) surface-links. In Section \ref{sect-inv-kbp-spl}, we define a series of new invariants $\{{\mathbf K}_{2n-1}(\mathcal L) | n=2, 3, 4, \ldots\}$ for surface-links $\mathcal L$ using skein relations and calculate the invariants for various surface-links in $4$-space.


\section{Marked graph diagrams of surface-links}\label{sect-omgr-osl}

In this section, we review marked graph diagrams presenting surface-links. A {\it marked vertex graph} or simply a {\it marked graph} is a spatial graph $G$ in $\mathbb R^3$ which satisfies that $G$ is a finite regular graph possibly with $4$-valent vertices, say $v_1, v_2, . . . , v_n$; each vertex $v_i$ is a rigid vertex (that is, we fix a rectangular neighborhood $N_i$ homeomorphic to $\{(x, y)|-1 \leq x, y \leq 1\},$ where $v_i$ corresponds to the origin and the edges incident to $v_i$ are represented by $x^2 = y^2$); each vertex $v_i$ has a {\it marker} which is the interval on $N_i$ given by  $\{(x, 0)|-\frac{1}{2} \leq x \leq \frac{1}{2}\}$.

 An {\it orientation} of a marked graph $G$ is a choice of an orientation for each edge of $G$ in such a way that every vertex in $G$ looks like 
\xy (-5,5);(5,-5) **@{-}, 
(5,5);(-5,-5) **@{-}, 
(3,3.2)*{\llcorner}, 
(-3,-3.4)*{\urcorner}, 
(-2.5,2)*{\ulcorner},
(2.5,-2.4)*{\lrcorner}, 
(3,-0.2);(-3,-0.2) **@{-},
(3,0);(-3,0) **@{-}, 
(3,0.2);(-3,0.2) **@{-}, 
\endxy~or
\xy (-5,5);(5,-5) **@{-},
(5,5);(-5,-5) **@{-},
(2.5,2)*{\urcorner},
(-2.5,-2.2)*{\llcorner},
(-3.2,3)*{\lrcorner},
(3,-3.4)*{\ulcorner},
(3,-0.2);(-3,-0.2) **@{-},
(3,0);(-3,0) **@{-},
(3,0.2);(-3,0.2) **@{-}
\endxy. 
A marked graph is said to be {\it orientable} if it admits an orientation. Otherwise, it is said to be {\it nonorientable}. By an {\it oriented marked graph} we mean an orientable marked graph with a fixed orientation. Two oriented marked graphs are {\it equivalent} if they are ambient isotopic in $\mathbb R^3$ with keeping rectangular neighborhoods, orientation and markers. An oriented marked graph $G$ in $\mathbb R^3$ can be described as usual by a diagram $D$ in $\mathbb R^2$, which is an oriented link diagram in $\mathbb R^2$ possibly with some marked $4$-valent vertices whose incident four edges have orientations illustrated as above, and is called an {\it oriented marked graph diagram} of $G$ (cf. Figure~\ref{fig-oriunori-wmp}). 

\begin{figure}[ht]
\begin{center}
\resizebox{0.55\textwidth}{!}{%
  \includegraphics{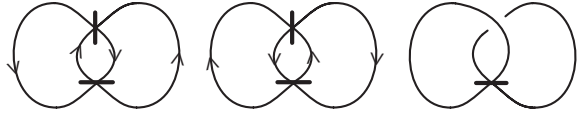}}
\caption{Oriented marked graph diagrams and a nonorientable marked graph diagram}\label{fig-oriunori-wmp}
\end{center}
\end{figure}

Two oriented marked graph diagrams in $\mathbb R^2$ represent equivalent oriented marked graphs in $\mathbb R^3$ if and only if they are transformed into each other by a finite sequence of oriented mark preserving rigid vertex 4-regular spatial graph moves (simply, {\it oriented mark preserving RV4 moves}) $\Gamma_1, \Gamma'_1, \Gamma_2, \Gamma_3, \Gamma_4, \Gamma'_4$ and $\Gamma_5$ shown in Figure~\ref{fig-Y-moves-type-I-o} (\cite{Kau0,KJL2}), which consist of Yoshikawa moves of type I (see Theorem \ref{thm-equiv-mgds-ym}).

\begin{figure}[ht]
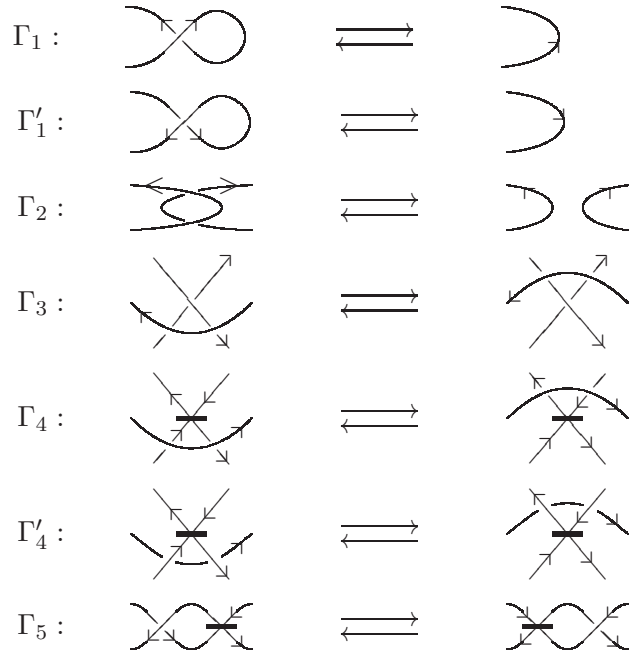

\centerline{\xy 
(12,2);(16,6) **@{-}, 
(12,6);(13.5,4.5) **@{-},
(14.5,3.5);(16,2) **@{-}, 
(16,6);(22,6) **\crv{(18,8)&(20,8)},
(16,2);(22,2) **\crv{(18,0)&(20,0)}, (22,6);(22,2) **\crv{(23.5,4)},
(7,8);(12,6) **\crv{(10,8)}, (7,0);(12,2) **\crv{(10,0)}, (12.4,5) *{\ulcorner}, (15.6,5) *{\urcorner},
(35,5);(45,5) **@{-} ?>*\dir{>}, (35,3);(45,3) **@{-} ?<*\dir{<},
(63.8,2) *{\urcorner},
(57,8);(57,0) **\crv{(67,7)&(67,1)}, (-5,4)*{\Gamma_1 :}, (73,4)*{},
\endxy}
\vskip.3cm
\centerline{ 
\xy (12,2);(16,6) **@{-}, 
(12,6);(13.5,4.5) **@{-},
(14.5,3.5);(16,2) **@{-}, 
(16,6);(22,6) **\crv{(18,8)&(20,8)},
(16,2);(22,2) **\crv{(18,0)&(20,0)}, (22,6);(22,2) **\crv{(23.5,4)},
(7,8);(12,6) **\crv{(10,8)}, (7,0);(12,2) **\crv{(10,0)}, (12.4,2.5) *{\llcorner}, (15.6,2.5) *{\lrcorner},
(35,5);(45,5) **@{-} ?>*\dir{>}, (35,3);(45,3) **@{-} ?<*\dir{<},(63.85,5.2) *{\lrcorner},
(57,8);(57,0) **\crv{(67,7)&(67,1)}, (-5,4)*{\Gamma'_1 :}, (73,4)*{},
\endxy}
\vskip.3cm
\centerline{ \xy (7,7);(7,1)  **\crv{(23,6)&(23,2)}, (16,6.3);(23,7)
**\crv{(19,6.9)}, (16,1.7);(23,1) **\crv{(19,1.1)},
(14,5.7);(14,2.3) **\crv{(8,4)}, (10,6.9) *{<}, (20,6.9) *{>},
(35,5);(45,5) **@{-} ?>*\dir{>}, (35,3);(45,3) **@{-} ?<*\dir{<},
(57,7);(57,1) **\crv{(65,6)&(65,2)}, (73,7);(73,1)
**\crv{(65,6)&(65,2)}, (60,5.3) *{\ulcorner}, (70,5.3) *{\urcorner}, (-5,4)*{\Gamma_2 :},
\endxy}
\vskip.3cm
\centerline{ 
\xy (7,6);(23,6) **\crv{(15,-2)}, 
(10,0);(11.5,1.8) **@{-}, 
(17.5,3);(14.5,6.6) **@{-},
(14.5,6.6);(10,12) **@{-}, 
(20,12);(15.5,6.6) **@{-},
(14.5,5.5);(12.5,3) **@{-},
(18.5,1.8);(20,0) **@{-},
(19.5,11) *{\urcorner}, 
(19,1.1) *{\lrcorner},  
(9,3.5) *{\ulcorner},
(35,7);(45,7) **@{-} ?>*\dir{>}, 
(35,5);(45,5) **@{-} ?<*\dir{<},
(57,6);(73,6) **\crv{(65,14)}, 
(70,12);(68.5,10.2) **@{-}, 
(67.5,9);(65.5,6.5) **@{-}, 
(64.6,5.5);(60,0) **@{-}, 
(62.5,9);(64.4,6.6) **@{-}, 
(64.4,6.6);(70,0) **@{-}, 
(61.5,10.2);(60,12) **@{-},
(69.5,11) *{\urcorner}, 
(69,1.1) *{\lrcorner},  
(58,7) *{\llcorner},
(-5,6)*{\Gamma_3:},
\endxy}
\vskip.3cm
 \centerline{ \xy 
 (7,6);(23,6)  **\crv{(15,-2)}, 
 (10,0);(11.5,1.8) **@{-},
(12.5,3);(20,12) **@{-}, 
(10,12);(17.5,3) **@{-}, 
(18.5,1.8);(20,0) **@{-}, 
(13,6);(17,6) **@{-}, (13,6.1);(17,6.1) **@{-}, (13,5.9);(17,5.9)
**@{-}, (13,6.2);(17,6.2) **@{-}, (13,5.8);(17,5.8) **@{-},
  (13,3.1) *{\urcorner}, (17.5,8.9) *{\llcorner}, (19,1.1) *{\lrcorner},  (21,3.5) *{\urcorner},(13,8) *{\ulcorner},
(35,7);(45,7) **@{-} ?>*\dir{>}, 
(35,5);(45,5) **@{-} ?<*\dir{<},
(57,6);(73,6)  **\crv{(65,14)}, 
(70,12);(68.5,10.2) **@{-},
(67.5,9);(60,0) **@{-}, 
(70,0);(62.5,9) **@{-}, 
(61.5,10.2);(60,12) **@{-}, 
(63,6);(67,6) **@{-}, (63,6.1);(67,6.1) **@{-}, (63,5.9);(67,5.9)
**@{-}, (63,6.2);(67,6.2) **@{-}, (63,5.8);(67,5.8) **@{-},
(62,2) *{\urcorner}, (67,8.3) *{\llcorner}, (67.5,3) *{\lrcorner},  (70.9,8) *{\lrcorner},(61.3,10) *{\ulcorner}, 
(-5,6)*{\Gamma_4:},
\endxy}
\vskip.3cm
 \centerline{ \xy 
  (13,2.2);(17,2.2)  **\crv{(15,1.7)}, 
  (7,6);(11,3)  **\crv{(10,3.5)}, 
  (23,6);(19,3)  **\crv{(20,3.5)}, 
 (10,0);(20,12) **@{-}, 
(10,12);(20,0) **@{-}, 
(13,6);(17,6) **@{-}, (13,6.1);(17,6.1) **@{-}, (13,5.9);(17,5.9)
**@{-}, (13,6.2);(17,6.2) **@{-}, (13,5.8);(17,5.8) **@{-}, 
(13,3.1) *{\urcorner}, (17.5,8.9) *{\llcorner}, (19,1.1) *{\lrcorner},  (21,3.5) *{\urcorner},(13,8) *{\ulcorner},
(35,7);(45,7) **@{-} ?>*\dir{>}, 
(35,5);(45,5) **@{-} ?<*\dir{<},
   (63,9.8);(67,9.8)  **\crv{(65,10.3)}, 
  (57,6);(61,9)  **\crv{(60,8.5)}, 
  (73,6);(69,9)  **\crv{(70,8.5)}, 
(70,12);(60,0) **@{-}, 
(70,0);(60,12) **@{-}, 
(63,6);(67,6) **@{-}, (63,6.1);(67,6.1) **@{-}, (63,5.9);(67,5.9)
**@{-},(63,6.2);(67,6.2) **@{-}, (63,5.8);(67,5.8) **@{-}, 
(62,2) *{\urcorner}, (67,8.3) *{\llcorner}, (67.5,3) *{\lrcorner},  (70.9,8) *{\lrcorner},(61.3,10) *{\ulcorner}, 
(-5,6)*{\Gamma'_4:},
\endxy}
\vskip.3cm
\centerline{ \xy (9,2);(13,6) **@{-}, (9,6);(10.5,4.5) **@{-},
(11.5,3.5);(13,2) **@{-}, (17,2);(21,6) **@{-}, (17,6);(21,2)
**@{-}, (13,6);(17,6) **\crv{(15,8)}, (13,2);(17,2) **\crv{(15,0)},
(7,7);(9,6) **\crv{(8,7)}, (7,1);(9,2) **\crv{(8,1)}, (23,7);(21,6)
**\crv{(22,7)}, (23,1);(21,2) **\crv{(22,1)}, 
(17,4);(21,4) **@{-}, (17,4.1);(21,4.1) **@{-}, (17,3.9);(21,3.9)
**@{-}, (17,4.2);(21,4.2) **@{-}, (17,3.8);(21,3.8) **@{-},
(10,3) *{\llcorner},  (12,3) *{\lrcorner}, (21,6) *{\llcorner},(21,2.2) *{\lrcorner},
(35,5);(45,5) **@{-} ?>*\dir{>}, (35,3);(45,3) **@{-} ?<*\dir{<},
(59,2);(63,6) **@{-}, (59,6);(63,2) **@{-}, (67,2);(71,6) **@{-},
(67,6);(68.5,4.5) **@{-}, (69.5,3.5);(71,2) **@{-}, (63,6);(67,6)
**\crv{(65,8)}, (63,2);(67,2) **\crv{(65,0)}, (57,7);(59,6)
**\crv{(58,7)}, (57,1);(59,2) **\crv{(58,1)}, (73,7);(71,6)
**\crv{(72,7)}, (73,1);(71,2) **\crv{(72,1)}, 
(63,4);(59,4) **@{-}, (63,4.1);(59,4.1) **@{-}, (63,3.9);(59,3.9)
**@{-}, (63,4.2);(59,4.2) **@{-}, (63,3.8);(59,3.8) **@{-},
(59.5,2.5) *{\llcorner},  (59,6) *{\lrcorner}, (71,6) *{\llcorner},(71,2.2) *{\lrcorner},
 (-5,4)*{\Gamma_5:},
\endxy}
\caption{Oriented mark preserving RV4 moves}\label{fig-Y-moves-type-I-o}
\end{figure}

 An {\it unoriented} marked graph diagram or, simply, a marked graph diagram is a nonorientable or an orientable but not oriented marked graph diagram in $\mathbb R^2$. Two marked graph diagrams in $\mathbb R^2$ represent equivalent marked graphs in $\mathbb R^3$ if and only if they are transformed into each other by a finite sequence of the moves $\Omega_1, \Omega_2, \Omega_3,\Omega_4, \Omega'_4$ and $\Omega_5$, where $\Omega_i$ stands for the move $\Gamma_i$ forgetting orientation.

For an (oriented) marked graph diagram $D$, let $L_-(D)$ and $L_+(D)$ be the (oriented) link diagrams obtained from $D$ by replacing each marked vertex \xy (-4,4);(4,-4) **@{-}, 
(4,4);(-4,-4) **@{-},  
(3,-0.2);(-3,-0.2) **@{-},
(3,0);(-3,0) **@{-}, 
(3,0.2);(-3,0.2) **@{-}, 
\endxy with \xy (-4,4);(-4,-4) **\crv{(1,0)},  
(4,4);(4,-4) **\crv{(-1,0)}, 
\endxy and \xy (-4,4);(4,4) **\crv{(0,-1)}, 
(4,-4);(-4,-4) **\crv{(0,1)},   
\endxy, respectively, as illustrated in Figure~\ref{fig-nori-mg}.
We call $L_-(D)$ and $L_+(D)$ the {\it negative resolution} and the {\it positive resolution} of D, respectively. An (oriented) marked graph diagram $D$ is {\it admissible} if both resolutions $L_-(D)$ and $L_+(D)$ are trivial link diagrams.

\begin{figure}[ht]
\begin{center}
\resizebox{0.45\textwidth}{!}{%
  \includegraphics{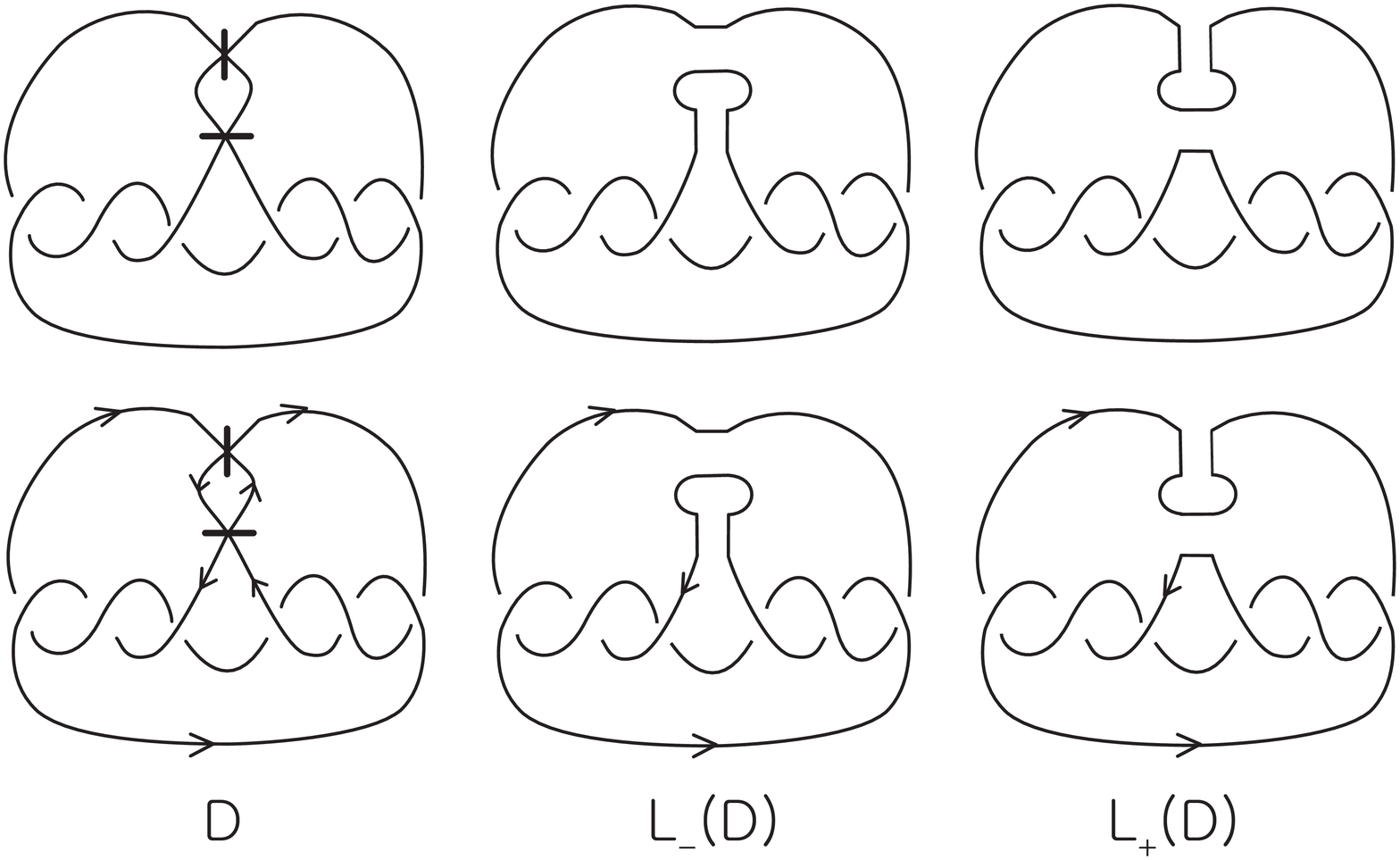}}
\caption{Marked graph diagrams and their resolutions}
\label{fig-nori-mg}
\end{center}
\end{figure}	

For $t \in \mathbb R,$ we denote by $\mathbb R^3_t$ the hyperplane of $\mathbb R^4$ whose fourth coordinate
is equal to $t \in \mathbb R$, i.e., $\mathbb R^3_t := \{(x_1, x_2, x_3, x_4) \in
\mathbb R^4~|~ x_4 = t \}$. A surface-link $\mathcal L \subset \mathbb R^4=\mathbb R^3 \times \mathbb R$ can be described in terms of its {\it cross-sections} $\mathcal L_t=\mathcal L \cap \mathbb R^3_t, ~ t \in \mathbb R$ (cf. \cite{Fox}). Let $p:\mathbb R^4 \to \mathbb R$ be the projection given by $p(x_1, x_2, x_3, x_4)=x_4$. Note that any surface-link can be perturbed to a surface-link $\mathcal L$ such that the projection $p : \mathcal L \to \mathbb R$ is a Morse function with finitely many distinct non-degenerate critical values. More especially, it is well known (cf. \cite{Ka2,Kaw,KSS,Lo}) that any surface-link $\mathcal L$ can be deformed into a surface-link $\mathcal L'$, called a {\it hyperbolic splitting} of $\mathcal L$,
by an ambient isotopy of $\mathbb R^4$ in such a way that
the projection $p: \mathcal L' \to \mathbb R$ satisfies that
all critical points are non-degenerate,
all the index 0 critical points (minimal points) are in $\mathbb R^3_{-1}$,
all the index 1 critical points (saddle points) are in $\mathbb R^3_0$, and
all the index 2 critical points (maximal points) are in $\mathbb R^3_1$.

Let $\mathcal L$ be a surface-link and let ${\mathcal L'}$ be a hyperbolic splitting of $\mathcal L.$ Then the cross-section $\mathcal L'_0=\mathcal L'\cap \mathbb R^3_0$ at $t=0$ is a spatial $4$-valent regular graph in $\mathbb R^3_0$. We give a marker at each $4$-valent vertex (saddle point) that indicates how the saddle point opens up above as illustrated in Figure~\ref{sleesan2:fig1}. The resulting marked graph $G$ is called a {\it marked graph presenting $\mathcal L$}.  A diagram of a marked graph presenting $\mathcal L$ is clearly admissible, and is called a {\it marked graph diagram} (or {\it ch-diagram} (cf. \cite{So})) {\it presenting} $\mathcal L$.  

\begin{figure}[h]
\begin{center}
\resizebox{0.55\textwidth}{!}{%
  \includegraphics{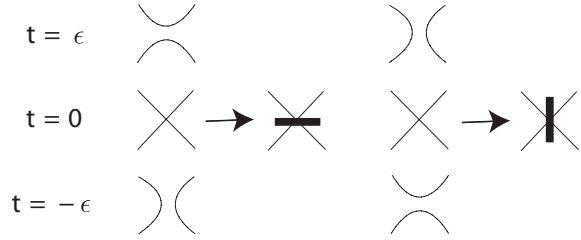} }
\caption{Marking of a vertex} \label{sleesan2:fig1}
\end{center}
\end{figure} 

When $\mathcal L$ is an oriented surface-link, we choose an orientation for each edge of $\mathcal L'_0$ that coincides with the induced orientation on the boundary of $\mathcal L' \cap \mathbb R^3 \times (-\infty, 0]$ by the orientation of $\mathcal L'$ inherited from the orientation of $\mathcal L$. The resulting oriented marked graph diagram $D$ is called an {\it oriented marked graph diagram presenting $\mathcal L$}. 

\begin{figure}[ht]
\begin{center}
\resizebox{0.50\textwidth}{!}{%
\includegraphics{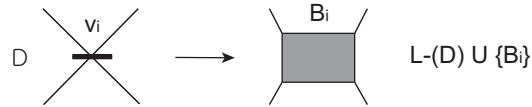} }
\caption{A band attached to $L_-(D)$ at $v_i$}\label{fig-orbd}
\end{center}
\end{figure}

Let $D$ be an admissible marked graph diagram with marked vertices $v_1, \ldots, v_n$. We define a surface $F(D) \subset \mathbb R^3 \times [-1,1]$ by
\begin{equation*}
(\mathbb R^3_t, F(D) \cap \mathbb R^3_t)=\left\{%
\begin{array}{ll}
    (\mathbb R^3, L_+(D)) & \hbox{for $0 < t \leq 1$,} \\
    \bigg( \mathbb R^3, L_-(D) \cup \bigg( \overset{n}{\underset{i=1}{\cup}} B_i \bigg) \bigg) & \hbox{for $t = 0$,} \\
    (\mathbb R^3, L_-(D)) & \hbox{for $-1 \leq t < 0$,} \\
\end{array}%
\right.
\end{equation*}
where $B_i (i=1, \ldots,n)$ is a band attached to $L_-(D)$ at the marked vertex $v_i$ as shown in Figure~\ref{fig-orbd}.
We call $F(D)$ the {\it proper surface associated with} $D$.
Since $D$ is admissible, we can obtain a surface-link from $F(D)$ by attaching trivial disks in $\mathbb R^3 \times [1, \infty)$ and another trivial disks in $\mathbb R^3 \times (-\infty, 1]$.  We denote the resulting surface-link by $\mathcal S(D)$, and call it the {\it surface-link associated with $D$}. It is known that the isotopy type of $\mathcal S(D)$ does not depend on choices of trivial disks (cf. \cite{Ka2,KSS}). It is straightforward from the construction of $\mathcal S(D)$ that $D$ is a marked graph diagram presenting $\mathcal S(D)$. 
It is known that $D$ is orientable if and only if $F(D)$ is an orientable surface.  When $D$ is oriented, the resolutions $L_-(D)$ and $L_+(D)$ have orientations induced from the orientation of $D$ (see Figure~\ref{fig-nori-mg}), and we assume $F(D)$ is oriented so that the induced orientation on $L_+(D) = \partial F(D) \cap \mathbb R^3_1$ matches the orientation of $L_+(D)$. Let $\mathcal L$ be an oriented surface-link in $\mathbb R^4$. We say that $\mathcal L$ is {\it presented} by an oriented marked graph diagram $D$ if $\mathcal L$ is ambient isotopic to the oriented surface-link $\mathcal S(D)$ in $\mathbb R^4$. Note that any oriented surface-link is presented by an oriented marked graph diagram. 

Throughout this paper, a {\it surface-link} means a nonorientable surface-link or an orientable surface-link without orientation, and an {\it oriented surface-link} means an orientable surface-link with a fixed orientation. Now we conclude this section by recalling the following:

\begin{theorem}[\cite{KJL2,KK,Sw}]\label{thm-equiv-mgds-ym}
(1) Two oriented marked graph diagrams present the same oriented surface-link if and only if they are transformed into each other by a finite sequence of the oriented mark preserving RV4 moves in Figure~\ref{fig-Y-moves-type-I-o}, called the {\it oriented Yoshikawa moves of type I}, and the {\it oriented Yoshikawa moves of type II} in Figure~\ref{fig-moves-type-II-o}.		
		
\begin{figure}[h]
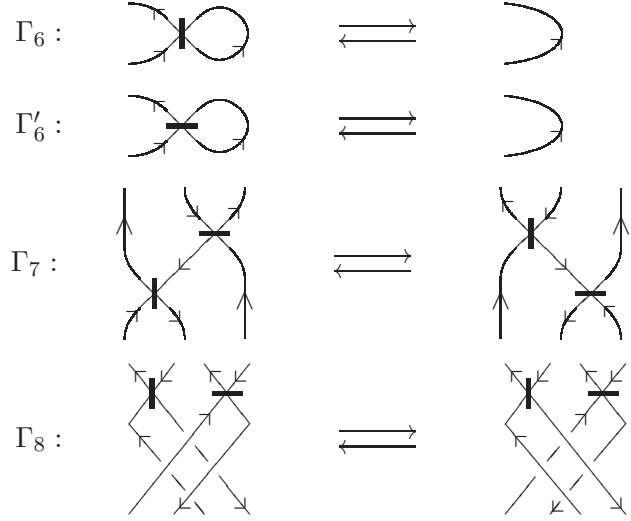

\centerline{ 
\xy (12,6);(16,2) **@{-}, (12,2);(16,6) **@{-},
(16,6);(22,6) **\crv{(18,8)&(20,8)}, (16,2);(22,2)
**\crv{(18,0)&(20,0)}, (22,6);(22,2) **\crv{(23.5,4)}, (7,8);(12,6)
**\crv{(10,8)}, (7,0);(12,2) **\crv{(10,0)}, (11,0.4) *{\urcorner}, (11,6) *{\ulcorner},(21.5,1)*{\urcorner},
(35,5);(45,5) **@{-} ?>*\dir{>}, (35,3);(45,3) **@{-} ?<*\dir{<},
(57,8);(57,0) **\crv{(67,7)&(67,1)}, (-5,4)*{\Gamma_6 :}, (73,4)*{},
(14,6);(14,2) **@{-}, (14.1,6);(14.1,2) **@{-}, (13.9,6);(13.9,2)
**@{-}, (14.2,6);(14.2,2) **@{-}, (13.8,6);(13.8,2) **@{-}, 
(63.8,2) *{\urcorner},
\endxy}
\vskip.3cm
\centerline{ \xy (12,6);(16,2) **@{-}, (12,2);(16,6) **@{-},
(16,6);(22,6) **\crv{(18,8)&(20,8)}, (16,2);(22,2)
**\crv{(18,0)&(20,0)}, (22,6);(22,2) **\crv{(23.5,4)}, (7,8);(12,6)
**\crv{(10,8)}, (7,0);(12,2) **\crv{(10,0)}, (11,0.4) *{\urcorner}, (11,6) *{\ulcorner},(21.5,1)*{\urcorner},
(35,5);(45,5) **@{-} ?>*\dir{>}, (35,3);(45,3) **@{-} ?<*\dir{<},
(57,8);(57,0) **\crv{(67,7)&(67,1)}, (-5,4)*{\Gamma'_6 :},
(73,4)*{}, (12,4);(16,4) **@{-}, (12,4.1);(16,4.1) **@{-},
(12,4.2);(16,4.2) **@{-}, (12,3.9);(16,3.9) **@{-},
(12,3.8);(16,3.8) **@{-}, (63.8,2) *{\urcorner},
\endxy}
\vskip.3cm
\centerline{ \xy (9,4);(17,12) **@{-}, (9,8);(13,4) **@{-},
(17,12);(21,16) **@{-}, (17,16);(21,12) **@{-}, (7,0);(9,4)
**\crv{(7,2)}, (7,12);(9,8) **\crv{(7,10)}, (15,0);(13,4)
**\crv{(15,2)}, (17,16);(15,20) **\crv{(15,18)}, (21,16);(23,20)
**\crv{(23,18)}, (21,12);(23,8) **\crv{(23,10)}, (7,12);(7,20)
**@{-}, (23,8);(23,0) **@{-},
(11,4);(11,8) **@{-}, 
(10.9,4);(10.9,8) **@{-}, 
(11.1,4);(11.1,8) **@{-}, 
(10.8,4);(10.8,8) **@{-}, 
(11.2,4);(11.2,8) **@{-},
(17,14);(21,14) **@{-}, 
(17,14.1);(21,14.1) **@{-},
(17,13.9);(21,13.9) **@{-}, 
(17,14.2);(21,14.2) **@{-},
(17,13.8);(21,13.8) **@{-},
(7,15) *{\wedge},(23,5) *{\wedge},(15,10) *{\llcorner},(8,2.3) *{\urcorner}, (21.5,16) *{\urcorner},(16,17) *{\lrcorner},(13.7,3) *{\lrcorner},
(35,11);(45,11) **@{-} ?>*\dir{>}, (35,9);(45,9) **@{-} ?<*\dir{<},
(71,4);(63,12) **@{-}, (71,8);(67,4) **@{-}, (63,12);(59,16) **@{-},
(63,16);(59,12) **@{-}, (73,0);(71,4) **\crv{(73,2)}, (73,12);(71,8)
**\crv{(73,10)}, (65,0);(67,4) **\crv{(65,2)}, (63,16);(65,20)
**\crv{(65,18)}, (59,16);(57,20) **\crv{(57,18)}, (59,12);(57,8)
**\crv{(57,10)}, (73,12);(73,20) **@{-}, (57,8);(57,0) **@{-},
(61,12);(61,16) **@{-}, 
(61.1,12);(61.1,16) **@{-},
(60.9,12);(60.9,16) **@{-}, 
(61.2,12);(61.2,16) **@{-},
(60.8,12);(60.8,16) **@{-},
(57,5) *{\wedge},(73,15) *{\wedge},(65,10) *{\lrcorner},(58,17) *{\ulcorner}, (71.5,3) *{\ulcorner},(66.3,3) *{\llcorner},(63.7,16.8) *{\llcorner},
(67,6);(71,6) **@{-}, 
(67,6.1);(71,6.1) **@{-}, 
(67,5.9);(71,5.9) **@{-}, 
(67,6.2);(71,6.2) **@{-},
(67,5.8);(71,5.8) **@{-},  
(-5,10)*{\Gamma_7:}, 
 \endxy}
\vskip.3cm
\centerline{ 
\xy (7,20);(14.2,11) **@{-}, (15.8,9);(17.4,7) **@{-},
(19,5);(23,0) **@{-}, (13,20);(7,12) **@{-}, (7,12);(11.2,7) **@{-},
(12.7,5.2);(14.4,3.2) **@{-}, (15.7,1.6);(17,0) **@{-},
(17,20);(23,12) **@{-}, (13,0);(23,12) **@{-}, (7,0);(23,20) **@{-},
(10,18);(10,14) **@{-}, (10.1,18);(10.1,14) **@{-},
(9.9,18);(9.9,14) **@{-}, (10.2,18);(10.2,14) **@{-},
(9.8,18);(9.8,14) **@{-}, (18,16);(22,16) **@{-},
(18,16.1);(22,16.1) **@{-}, (18,15.9);(22,15.9) **@{-},
(18,16.2);(22,16.2) **@{-}, (18,15.8);(22,15.8) **@{-},
 (8.5,17.7) *{\ulcorner}, (18.5,17.7) *{\ulcorner}, (9.1,9) *{\ulcorner}, (12,18.4) *{\llcorner}, (14.5,1.6) *{\llcorner}, (21.8,18.4) *{\llcorner}, (17,12) *{\urcorner}, (21.7,1.6) *{\lrcorner},
(35,11);(45,11) **@{-} ?>*\dir{>}, (35,9);(45,9) **@{-} ?<*\dir{<},
(73,20);(65.8,11) **@{-}, (64.2,9);(62.6,7) **@{-}, (61,5);(57,0)
**@{-}, (67,20);(73,12) **@{-}, (73,12);(68.8,7) **@{-},
(67.3,5.2);(65.6,3.2) **@{-}, (64.3,1.6);(63,0) **@{-},
(63,20);(57,12) **@{-}, (67,0);(57,12) **@{-}, (73,0);(57,20)
**@{-},
 (58.5,17.7) *{\ulcorner}, (68.5,17.7) *{\ulcorner}, (59.1,9) *{\ulcorner}, (62,18.4) *{\llcorner}, (64,1.2) *{\llcorner}, (71.8,18.4) *{\llcorner}, (67,12) *{\urcorner}, (71.7,1.6) *{\lrcorner},
(60,18);(60,14) **@{-}, (60.1,18);(60.1,14) **@{-},
(59.9,18);(59.9,14) **@{-}, (60.2,18);(60.2,14) **@{-},
(59.8,18);(59.8,14) **@{-}, (68,16);(72,16) **@{-},
(68,16.1);(72,16.1) **@{-}, (68,15.9);(72,15.9) **@{-},
(68,16.2);(72,16.2) **@{-}, (68,15.8);(72,15.8) **@{-},
(-5,10)*{\Gamma_{8}:}, 
\endxy}
\caption{Oriented Yoshikawa moves of type II}
\label{fig-moves-type-II-o}
\end{figure} 	

(2) Two unoriented marked graph diagrams present the same surface-link if and only if they are transformed into each other by a finite sequence of the unoriented mark preserving RV4 moves $\Omega_1, \Omega_2, \Omega_3,\Omega_4, \Omega'_4, \Omega_5$, called the {\it unoriented Yoshikawa moves of type I}, and the moves $\Omega_6, \Omega'_6, \Omega_7$ and $\Omega_8$, called the {\it unoriented Yoshikawa moves of type II}, where $\Omega_i$ stands for the move $\Gamma_i$ without orientation.
\end{theorem}	


\section{Polynomial invariants for marked graphs in $\mathbb R^3$ associated with link invariants}\label{sect-poly-mgd}

In this section, we first review the polynomial invariants for unoriented marked graphs in $\mathbb R^3$ associated with invariants for unoriented knots and links in $\mathbb R^3$ (firstly introduced in \cite{Le3} and refined in \cite{JKL}) in a specialized fashion that is more applicable in practice, and then we define polynomial invariants for oriented marked graphs in $\mathbb R^3$ associated with invariants for oriented knots and  links in $\mathbb R^3$. 

Let $R$ be a commutative ring with the additive identity $0$ and the multiplicative identity $1$ and let $[~~] : \{\text{unoriented links in $\mathbb R^3$}\} \longrightarrow R$ be a regular or an ambient isotopy invariant such that for a unit $\alpha \in R$ and an element $\delta \in R$,
\begin{align}
&\Big[~\xy (1,-3);(6,2) **@{-},
(1,3);(3.5,0.5) **@{-},
(4.5,-0.5);(6,-2) **@{-},
(6,2);(10,2) **\crv{(8,4)},
(6,-2);(10,-2) **\crv{(8,-4)}, 
(10,2);(10,-2) **\crv{(11.5,0)}, 
\endxy~\Big] = \alpha \Big[~\xy (1.5,-3);(1.5,3) **\crv{(4,0)}, 
\endxy~\Big],~~~
\Big[~\xy (4.5,0.6);(6,2) **@{-},
(1,-3);(3.3,-0.7) **@{-},
(1,3);(3.5,0.5) **@{-},
(3.5,0.5);(6,-2) **@{-},
(6,2);(10,2) **\crv{(8,4)},
(6,-2);(10,-2) **\crv{(8,-4)}, 
(10,2);(10,-2) **\crv{(11.5,0)}, 
\endxy~\Big] 
= \alpha^{-1}\Big[~\xy 
(1.5,-3);(1.5,3) **\crv{(4,0)}, 
\endxy~\Big],\label{p-clinv-1}\\
&\Big[K \sqcup \bigcirc\Big] = \delta \Big[ K\Big], \label{p-clinv-2}
\end{align}
where $K \sqcup \bigcirc$ denotes a disjoint union of a circle
$\bigcirc$ and a link diagram $K.$

Let $R[x,y]$ be the ring of polynomials in variables $x$ and $y$ with coefficients in $R$.

\begin{definition}\label{defn-poly-skr}
For a given marked graph diagram $D$, let $[[D]](x,y)$ ([[D]] for short) be a polynomial in $R[x,y]$ defined by the following two rules:
\begin{itemize}
\item[{\rm ({\bf L1})}] $[[D]]= [D]$ if $D$ is a link diagram,
\item[{\rm ({\bf L2})}]  
$[[~\xy (-4,4);(4,-4) **@{-}, 
(4,4);(-4,-4) **@{-},  
(3,-0.2);(-3,-0.2) **@{-},
(3,0);(-3,0) **@{-}, 
(3,0.2);(-3,0.2) **@{-}, 
\endxy~]] =
[[~\xy (-4,4);(4,4) **\crv{(0,-1)}, 
(4,-4);(-4,-4) **\crv{(0,1)},  
\endxy~]] x + [[~\xy (-4,4);(-4,-4) **\crv{(1,0)},  
(4,4);(4,-4) **\crv{(-1,0)},  
\endxy~]] y.$
\end{itemize}
\end{definition}

Let $D= D_1 \cup \cdots \cup D_m$ be an oriented link diagram and let $w(D_i)$ be the usual writhe of the component $D_i$. The {\it self-writhe $sw(D)$} of $D$ is defined to be the sum $sw(D)=\sum_{i=1}^mw(D_i).$
Now let $D$ be a marked graph diagram. We first choose an arbitrary orientation for each component of $L_+(D)$ and $L_-(D)$. Then we define the {\it self-writhe $sw(D)$} of $D$ by 
\begin{equation*}
sw(D)=\frac{sw(L_+(D))+sw(L_-(D))}{2}.
\end{equation*}
It is noted that $sw(L_+(D))$ and $sw(L_-(D))$ are independent of the choice of orientations because the writhe of each component of $L_+(D)$ and $L_-(D)$ is independent of the choice of orientation for the component. 
	
\begin{remark}\label{rmk-self-wr}
The self-writhe $sw(D)$ of a marked graph diagram $D$ is invariant under all unoriented Yoshikawa moves except the move $\Omega_1$. Indeed, it follows from \cite[Lemma 4.1]{Le3} that $t(D)=2sw(D)$ is invariant under all unoriented Yoshikawa moves and so does $sw(D)$, except the move $\Omega_1$. For $\Omega_1$, we have
\begin{equation*}
sw\bigg(~\xy (1,-3);(6,2) **@{-},
(1,3);(3.5,0.5) **@{-},
(4.5,-0.5);(6,-2) **@{-},
(6,2);(10,2) **\crv{(8,4)},
(6,-2);(10,-2) **\crv{(8,-4)}, 
(10,2);(10,-2) **\crv{(11.5,0)}, 
\endxy~\bigg)
=sw\bigg(~~\xy (1.5,-3);(1.5,3) **\crv{(4,0)}, 
\endxy~~\bigg)+1,~
sw\bigg(~\xy (4.5,0.6);(6,2) **@{-},
(1,-3);(3.3,-0.7) **@{-},
(1,3);(3.5,0.5) **@{-},
(3.5,0.5);(6,-2) **@{-},
(6,2);(10,2) **\crv{(8,4)},
(6,-2);(10,-2) **\crv{(8,-4)}, 
(10,2);(10,-2) **\crv{(11.5,0)}, 
\endxy~\bigg) 
= sw\bigg(~\xy 
(1.5,-3);(1.5,3) **\crv{(4,0)},  
\endxy~\bigg)-1.
\end{equation*}
\end{remark}

\begin{definition}\label{defn-poly}
Let $D$ be a marked graph diagram. We define $\ll D\gg(x,y)$ ($\ll D\gg$ for short) to be the polynomial in $R[x, y]$ given by 
\begin{equation*}
\ll D\gg=\alpha^{-sw(D)}[[D]](x,y).
\end{equation*}
\end{definition}

Let $D$ be a marked graph diagram.
A {\it state} of $D$ is an assignment of $T_\infty$ or $T_0$ to each marked vertex in $D$. Let $\mathfrak S(D)$ be the set of all states of $D$. For each state 
$\sigma \in \mathfrak S(D),$ let $D_{\sigma}$ denote the link diagram obtained from $D$ by replacing marked vertices of $D$ with trivial $2$-tangles
according to the assignment $T_\infty$ or $T_0$ by the
state $\sigma$:
\begin{align*}
&~~~\underset{~T_\infty}{\xy (-4,4);(4,-4) **@{-}, 
(4,4);(-4,-4) **@{-},  
(3,-0.2);(-3,-0.2) **@{-},
(3,0);(-3,0) **@{-}, 
(3,0.2);(-3,0.2) **@{-}, 
\endxy} \longrightarrow \xy (-4,4);(4,4) **\crv{(0,-1)}, 
(4,-4);(-4,-4) **\crv{(0,1)},    
\endxy~~~~~~
\underset{~T_0}{\xy (-4,4);(4,-4) **@{-}, 
(4,4);(-4,-4) **@{-},  
(3,-0.2);(-3,-0.2) **@{-},
(3,0);(-3,0) **@{-}, 
(3,0.2);(-3,0.2) **@{-}, 
\endxy} \longrightarrow \xy (-4,4);(-4,-4) **\crv{(1,0)},  
(4,4);(4,-4) **\crv{(-1,0)}, 
\endxy
\end{align*}
Then $\ll D \gg$ has the following {\it state-sum formula}:
\begin{equation}\label{state formula}
\ll D \gg=\alpha^{-sw(D)}\sum_{\sigma \in \mathfrak S(D)}
[ D_\sigma ]~ x^{\sigma(\infty)}y^{\sigma(0)},
\end{equation}
where $\sigma(\infty)$ and $\sigma(0)$ denote
the numbers of the assignments $T_\infty$ and $T_0$ of the
state $\sigma,$ respectively.

\bigskip

Now we define the polynomial $\ll D\gg$ for an oriented marked graph diagam $D$ associated with a given regular or an ambient isotopy invariant $$[~~] :\{\text{oriented links in $\mathbb R^3$}\} \to R$$
satisfying the properties (\ref{p-clinv-1}) and (\ref{p-clinv-2}) with all possible orientations.

\begin{definition}\label{defn-poly-ori-skr}
For a given oriented marked graph diagram $D$, let $[[D]](x,y)$ ([[D]] for short) be a polynomial in $R[x,y]$ defined by the following two rules:
\begin{itemize}
\item[{\rm ({\bf L1})}] $[[D]]= [D]$ if $D$ is an oriented link diagram,
\item[{\rm ({\bf L2})}]  
$[[~\xy (-4,4);(4,-4) **@{-}, 
(4,4);(-4,-4) **@{-}, 
(3,3.2)*{\llcorner}, 
(-3,-3.4)*{\urcorner}, 
(-2.5,2)*{\ulcorner},
(2.5,-2.4)*{\lrcorner}, 
(3,-0.2);(-3,-0.2) **@{-},
(3,0);(-3,0) **@{-}, 
(3,0.2);(-3,0.2) **@{-}, 
\endxy~]] =
[[~\xy (-4,4);(4,4) **\crv{(0,-1)}, 
(4,-4);(-4,-4) **\crv{(0,1)}, 
(-2.5,1.9)*{\ulcorner}, (2.5,-2.4)*{\lrcorner}, 
\endxy~]]x + [[~\xy (-4,4);(-4,-4) **\crv{(1,0)},  
(4,4);(4,-4) **\crv{(-1,0)}, 
(-2.5,1.9)*{\ulcorner}, (2.5,-2.4)*{\lrcorner}, 
\endxy~]]y.$
\end{itemize}
\end{definition}

Let $D$ be an oriented marked graph diagram. The {\it writhe} $w(D)$ of $D$ is defined to be the sum of the signs of all crossings in $D$ given by
$\mathrm{sign} \left( \xy (5,2.5);(0,-2.5) **@{-} ?<*\dir{<},
(5,-2.5);(3,-0.5) **@{-}, (0,2.5);(2,0.5) **@{-} ?<*\dir{<},
\endxy \right) = 1$ and $\mathrm{sign} \left( \xy (0,2.5);(5,-2.5)
**@{-} ?<*\dir{<}, (0,-2.5);(2,-0.5) **@{-}, (5,2.5);(3,0.5)
**@{-} ?<*\dir{<}, \endxy \right) = -1$ analogue to the writhe of an oriented link diagram. 
	
\begin{remark}\label{rmk-self-ori-wr}
It is easy to see from Figures \ref{fig-Y-moves-type-I-o} and \ref{fig-moves-type-II-o} that the writhe $w(D)$ of an oriented marked graph diagram $D$ is invariant under all oriented Yoshikawa moves except the move $\Gamma_1$ and $\Gamma'_1$. For $\Gamma_1$, $\Gamma'_1$ and their mirror moves, we have
\begin{equation}\label{rel-w(D)}
\begin{split}
&w(~\xy (1,-3);(6,2) **@{-},
(1,3);(3.5,0.5) **@{-},
(4.5,-0.5);(6,-2) **@{-},
(6,2);(10,2) **\crv{(8,4)},
(6,-2);(10,-2) **\crv{(8,-4)}, 
(10,2);(10,-2) **\crv{(11.5,0)},
(2.4,1.1) *{\ulcorner},  
\endxy~)
=w(~~\xy (1.5,-3);(1.5,3) **\crv{(4,0)}, 
(2.4,1.5) *{\ulcorner}, 
\endxy~~)+1,~~
w(~\xy (4.5,0.6);(6,2) **@{-},
(1,-3);(3.3,-0.7) **@{-},
(1,3);(3.5,0.5) **@{-},
(3.5,0.5);(6,-2) **@{-},
(6,2);(10,2) **\crv{(8,4)},
(6,-2);(10,-2) **\crv{(8,-4)}, 
(10,2);(10,-2) **\crv{(11.5,0)},
(2.4,1.1) *{\ulcorner}, 
\endxy~) 
= w(~\xy 
(1.5,-3);(1.5,3) **\crv{(4,0)}, 
(2.4,1.5) *{\ulcorner}, 
\endxy~)-1,\\
&w(~\xy (1,-3);(6,2) **@{-},
(1,3);(3.5,0.5) **@{-},
(4.5,-0.5);(6,-2) **@{-},
(6,2);(10,2) **\crv{(8,4)},
(6,-2);(10,-2) **\crv{(8,-4)}, 
(10,2);(10,-2) **\crv{(11.5,0)},
(2.4,1.8) *{\lrcorner}, 
\endxy~)
=w(~~\xy (1.5,-3);(1.5,3) **\crv{(4,0)}, 
(2.4,-1.8) *{\llcorner}, 
\endxy~~)+1,~~
w(~\xy (4.5,0.6);(6,2) **@{-},
(1,-3);(3.3,-0.7) **@{-},
(1,3);(3.5,0.5) **@{-},
(3.5,0.5);(6,-2) **@{-},
(6,2);(10,2) **\crv{(8,4)},
(6,-2);(10,-2) **\crv{(8,-4)}, 
(10,2);(10,-2) **\crv{(11.5,0)},
(2.4,1.8) *{\lrcorner}, 
\endxy~) 
= w(~\xy 
(1.5,-3);(1.5,3) **\crv{(4,0)}, 
(2.4,-1.8) *{\llcorner}, 
\endxy~)-1.
\end{split}
\end{equation}
\end{remark}

\begin{definition}\label{defn-poly-ori}
Let $D$ be an oriented marked graph diagram. We define $\ll D\gg(x,y)$ ($\ll D\gg$ for short) to be the polynomial in $R[x,y]$ given by 
\begin{equation*}
\ll D\gg=\alpha^{-w(D)}[[D]](x,y).
\end{equation*}
\end{definition}

Let $D$ be an oriented marked graph diagram.
A {\it state} of $D$ is an assignment of $T_\infty$ or $T_0$ to each marked vertex in $D$. Let $\mathfrak S(D)$ be the set of all states of $D$. For each state $\sigma \in \mathfrak S(D),$ let $D_{\sigma}$ denote the oriented link diagram obtained from $D$ by replacing marked vertices of $D$ with trivial oriented $2$-tangles
according to the assignment $T_\infty$ or $T_0$ by the
state $\sigma$:
\begin{align*}
&~~~\underset{~T_\infty}{\xy (-4,4);(4,-4) **@{-}, 
(4,4);(-4,-4) **@{-}, 
(3,3.2)*{\llcorner}, 
(-3,-3.4)*{\urcorner}, 
(-2.5,2)*{\ulcorner},
(2.5,-2.4)*{\lrcorner}, 
(3,-0.2);(-3,-0.2) **@{-},
(3,0);(-3,0) **@{-}, 
(3,0.2);(-3,0.2) **@{-}, 
\endxy} \longrightarrow \xy (-4,4);(4,4) **\crv{(0,-1)}, 
(4,-4);(-4,-4) **\crv{(0,1)},  
(-2.5,1.9)*{\ulcorner}, (2.5,-2.4)*{\lrcorner},   
\endxy~~~~~~
\underset{~T_0}{\xy (-4,4);(4,-4) **@{-}, 
(4,4);(-4,-4) **@{-}, 
(3,3.2)*{\llcorner}, 
(-3,-3.4)*{\urcorner}, 
(-2.5,2)*{\ulcorner},
(2.5,-2.4)*{\lrcorner}, 
(3,-0.2);(-3,-0.2) **@{-},
(3,0);(-3,0) **@{-}, 
(3,0.2);(-3,0.2) **@{-}, 
\endxy} \longrightarrow \xy (-4,4);(-4,-4) **\crv{(1,0)},  
(4,4);(4,-4) **\crv{(-1,0)}, 
(-2.5,1.9)*{\ulcorner}, (2.5,-2.4)*{\lrcorner}, 
\endxy
\end{align*}
Then $\ll D \gg$ has the following {\it state-sum formula}:
\begin{equation}\label{state formula-ori}
\ll D \gg=\alpha^{-w(D)}\sum_{\sigma \in \mathfrak S(D)}
[ D_\sigma ]~ x^{\sigma(\infty)}y^{\sigma(0)}.
\end{equation}

\begin{theorem}\label{thm-inv-mgs}
Let $G$ be an (resp.~oriented) marked graph in $\mathbb R^3$ and let $D$ be an (resp.~oriented) marked graph diagram of $G$. For any given regular or ambient isotopy invariant $[~~] : \{\text{(resp.~oriented) links in $\mathbb R^3$}\} \longrightarrow R$ satisfying the properties (\ref{p-clinv-1}) and (\ref{p-clinv-2}), the polynomial $\ll D \gg$ is invariant under (resp.~oriented) Yoshikawa moves of type I in Figure~\ref{fig-Y-moves-type-I-o}. Therefore $\ll D \gg$ is an ambient isotopy invariant of the (resp.~oriented) marked graph $G$ in $\mathbb R^3$.
\end{theorem}

\begin{proof}
Let $D$ be a marked graph diagram of $G$. It is easy to see that $\ll D\gg(x,y)=\langle\langle D\rangle\rangle(x,0,0,y)$ is an invariant for all unoriented Yoshikawa moves of type I (see \cite[Lemma 3.3]{JKL}).

To verify the invariance of $\ll \cdot \gg$ for oriented Yoshikawa moves of type I, we first check the moves $\Gamma_1$ and $\Gamma_1'$. It follows from the identities (\ref{p-clinv-1}) with orientations, (\ref{rel-w(D)}) and (\ref{state formula-ori}) that  
\begin{equation*}
\begin{split}
\ll~\xy (1,-3);(6,2) **@{-},
(1,3);(3.5,0.5) **@{-},
(4.5,-0.5);(6,-2) **@{-},
(6,2);(10,2) **\crv{(8,4)},
(6,-2);(10,-2) **\crv{(8,-4)}, 
(10,2);(10,-2) **\crv{(11.5,0)},
(2.4,1.1) *{\ulcorner}, 
(5.6,1.1) *{\urcorner}, 
\endxy~\gg
&=\alpha^{-w(~\xy (1,-3);(6,2) **@{-},
(1,3);(3.5,0.5) **@{-},
(4.5,-0.5);(6,-2) **@{-},
(6,2);(10,2) **\crv{(8,4)},
(6,-2);(10,-2) **\crv{(8,-4)}, 
(10,2);(10,-2) **\crv{(11.5,0)},
(2.4,1.1) *{\ulcorner}, 
(5.6,1.1) *{\urcorner}, 
\endxy~)}
[[~~\xy (1,-3);(6,2) **@{-},
(1,3);(3.5,0.5) **@{-},
(4.5,-0.5);(6,-2) **@{-},
(6,2);(10,2) **\crv{(8,4)},
(6,-2);(10,-2) **\crv{(8,-4)}, 
(10,2);(10,-2) **\crv{(11.5,0)},
(2.4,1.1) *{\ulcorner}, 
(5.6,1.1) *{\urcorner}, 
\endxy~~ ]]
=\alpha^{-w(~~\xy 
(1.5,-3);(1.5,3) **\crv{(4,0)}, 
(2.4,1.5) *{\ulcorner}, 
\endxy ~~)-1}\alpha
[[~\xy (1.5,-3);(1.5,3) **\crv{(4,0)}, 
(2.4,1.5) *{\ulcorner}, 
\endxy ~]]=\ll~\xy (1.5,-3);(1.5,3) **\crv{(4,0)}, 
(2.4,1.5) *{\ulcorner}, 
\endxy~\gg.
\end{split}
\end{equation*}
Similarly, we obtain 
$\ll~\xy (1,-3);(6,2) **@{-},
(1,3);(3.5,0.5) **@{-},
(4.5,-0.5);(6,-2) **@{-},
(6,2);(10,2) **\crv{(8,4)},
(6,-2);(10,-2) **\crv{(8,-4)}, 
(10,2);(10,-2) **\crv{(11.5,0)},
(2.4,-1.5) *{\llcorner}, 
(5.6,-1.5) *{\lrcorner}, 
\endxy~\gg~ 
=~\ll~\xy (1.5,-3);(1.5,3) **\crv{(4,0)}, 
(2.5,-1.5) *{\llcorner}, 
\endxy~\gg.$

Since $[D]$ and the writhe $w(D)$ for an oriented link diagram $D$ are both regular isotopy invariants, it is direct from (\ref{state formula-ori}) that $\ll\cdot \gg$ is invariant under $\Gamma_2$ and $\Gamma_3$.  
The invariances of $[[ \cdot ]]$ under the moves $\Gamma_4, \Gamma'_4$ and $\Gamma_5$ are seen from Figure~\ref{fig-4-5}. Since the writhe $w(D)$ is also invariant under these moves, we see the invariance of $\ll\cdot\gg$ under $\Gamma_4, \Gamma'_4$ and $\Gamma_5$. This completes the proof.
\end{proof}

\begin{figure}[ht]
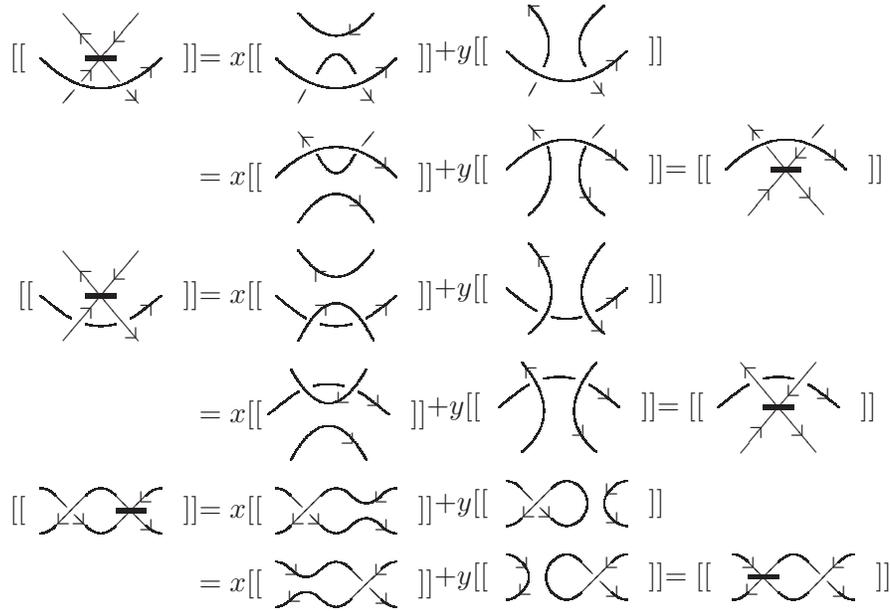

\begin{align*}
\xy (4,6)*{[[},
(7,6);(23,6)  **\crv{(15,-2)}, 
(10,0);(11.5,1.8) **@{-},
(12.5,3);(20,12) **@{-}, 
(10,12);(17.5,3) **@{-}, 
(18.5,1.8);(20,0) **@{-}, 
(13,6);(17,6) **@{-}, 
(13,6.1);(17,6.1) **@{-}, 
(13,5.9);(17,5.9)**@{-}, 
(13,6.2);(17,6.2) **@{-}, 
(13,5.8);(17,5.8) **@{-},
(13,3.1) *{\urcorner}, 
(17.5,8.9) *{\llcorner}, 
(19,1.1) *{\lrcorner},  
(21,3.5) *{\urcorner},
(13,8) *{\ulcorner}, 
(27,6) *{]]},
\endxy
&\xy (1,6)*{= x [[},
 (7,6);(23,6)  **\crv{(15,-2)}, 
(10,0);(11,1.8) **@{-},
(10,12);(20,12) **\crv{(15,6)},
(12.5,4);(17.5,4) **\crv{(15,9)},
(18.6,2);(20,0) **@{-}, 
(17.5,9.9) *{\llcorner}, 
(19,1.1) *{\lrcorner},  
(21,3.5) *{\urcorner},
(27,6) *{]]},
\endxy
\xy (1,6)*{+ y [[},
(7,6);(23,6)  **\crv{(15,-2)}, 
(10,0);(11,1.8) **@{-},
(10,12);(12,4) **\crv{(14,8)},
(20,12);(17.5,4) **\crv{(15,8)},
(18.6,2);(20,0) **@{-},
(19,1.1) *{\lrcorner},  
(21,3.5) *{\urcorner},
(10.7,10.7) *{\ulcorner}, 
(27,6) *{]]},
\endxy\\
&\xy (51,6)*{= x [[},
(57,6);(73,6)  **\crv{(65,14)}, 
(70,12);(68.5,10.2) **@{-},
(67.5,9);(62.5,9)**\crv{(65,4)}, 
(60,0);(70,0) **\crv{(65,7.5)},
(61.5,10.2);(60,12) **@{-}, 
(67.5,3) *{\lrcorner},  
(70.9,8) *{\lrcorner},
(61.3,10) *{\ulcorner}, 
(77,6)*{]]},
\endxy \xy (51,6)*{+ y [[},
(57,6);(73,6)  **\crv{(65,14)}, 
(70,12);(68.5,10.2) **@{-},
(67.5,9);(70,0) **\crv{(65,3.5)}, 
(60,0);(62.5,9) **\crv{(64,3.5)}, 
(61.5,10.2);(60,12) **@{-}, 
(67.5,3) *{\lrcorner},  
(70.9,8) *{\lrcorner},
(61.3,10) *{\ulcorner}, 
(77,6)*{]]},
\endxy
\xy (52,6)*{= [[},
(57,6);(73,6)  **\crv{(65,14)}, 
(70,12);(68.5,10.2) **@{-},
(67.5,9);(60,0) **@{-}, 
(70,0);(62.5,9) **@{-}, 
(61.5,10.2);(60,12) **@{-}, 
(63,6);(67,6) **@{-}, 
(63,6.1);(67,6.1) **@{-}, 
(63,5.9);(67,5.9)**@{-}, 
(63,6.2);(67,6.2) **@{-}, 
(63,5.8);(67,5.8) **@{-},
(62,2) *{\urcorner}, 
(67,8.3) *{\llcorner}, 
(67.5,3) *{\lrcorner},  
(70.9,8) *{\lrcorner},
(61.3,10) *{\ulcorner}, 
(77,6)*{]]},
\endxy\\
\xy (5,6)*{[[},
(13,2.2);(17,2.2)  **\crv{(15,1.7)}, 
(7,6);(11,3)  **\crv{(10,3.5)}, 
(23,6);(19,3)  **\crv{(20,3.5)}, 
(10,0);(20,12) **@{-}, 
(10,12);(20,0) **@{-}, 
(13,6);(17,6) **@{-},
(13,6.1);(17,6.1) **@{-}, 
(13,5.9);(17,5.9)**@{-}, 
(13,6.2);(17,6.2) **@{-}, 
(13,5.8);(17,5.8) **@{-}, 
(13,3.1) *{\urcorner}, 
(17.5,8.9) *{\llcorner}, 
(19,1.1) *{\lrcorner},  
(21,3.5) *{\urcorner},
(13,8) *{\ulcorner},
(27,6)*{]]},
\endxy
&\xy (1,6)*{= x [[},
(13,2.2);(17,2.2)  **\crv{(15,1.7)}, 
(7,6);(11,3)  **\crv{(10,3.5)}, 
(23,6);(19,3)  **\crv{(20,3.5)}, 
(10,0);(20,0) **\crv{(15,10)},
(20,12);(10,12) **\crv{(15,5)},
(13,3.5) *{\urcorner}, 
(21,3.5) *{\urcorner},
(13,8) *{\ulcorner},
(27,6)*{]]},
\endxy
\xy (1,6)*{+ y [[},
(13,2.2);(17,2.2) **\crv{(15,1.7)}, 
(7,6);(11,3) **\crv{(10,3.5)}, 
(23,6);(19,3) **\crv{(20,3.5)}, 
(10,0);(10,12) **\crv{(16,5)},
(20,12);(20,0) **\crv{(14,5)},
(19,1.1) *{\lrcorner},  (21,3.5) *{\urcorner},
(12,9) *{\ulcorner},
(27,6)*{]]},
\endxy\\
&\xy (52,6)*{= x [[},
(63,9.8);(67,9.8)  **\crv{(65,10.3)}, 
  (57,6);(61,9)  **\crv{(60,8.5)}, 
  (73,6);(69,9)  **\crv{(70,8.5)}, 
(60,0);(70,0) **\crv{(65,9)},
(70,12);(60,12) **\crv{(65,3)},
(67,8.3) *{\llcorner}, 
(67.5,3) *{\lrcorner},  
(70.9,8) *{\lrcorner},
(77,6)*{]]},
\endxy
\xy (51,6)*{+ y [[},
(63,9.8);(67,9.8)  **\crv{(65,10.3)}, 
  (57,6);(61,9)  **\crv{(60,8.5)}, 
  (73,6);(69,9)  **\crv{(70,8.5)}, 
(60,0);(60,12) **\crv{(66,5)},
(70,12);(70,0) **\crv{(64,5)},
(67.5,3) *{\lrcorner},  
(70.9,8) *{\lrcorner},
(61.3,10) *{\ulcorner}, 
(77,6)*{]]},
\endxy
\xy (52,6)*{= [[},
(63,9.8);(67,9.8)  **\crv{(65,10.3)}, 
  (57,6);(61,9)  **\crv{(60,8.5)}, 
  (73,6);(69,9)  **\crv{(70,8.5)}, 
(70,12);(60,0) **@{-}, 
(70,0);(60,12) **@{-}, 
(63,6);(67,6) **@{-}, 
(63,6.1);(67,6.1) **@{-}, 
(63,5.9);(67,5.9)**@{-},
(63,6.2);(67,6.2) **@{-}, 
(63,5.8);(67,5.8) **@{-}, 
(62,2) *{\urcorner}, 
(67,8.3) *{\llcorner}, 
(67.5,3) *{\lrcorner},  
(70.9,8) *{\lrcorner},
(61.3,10) *{\ulcorner}, 
(77,6)*{]]},
\endxy\\
\xy (4,4)*{[[},
(9,2);(13,6) **@{-}, 
(9,6);(10.5,4.5) **@{-},
(11.5,3.5);(13,2) **@{-}, 
(17,2);(21,6) **@{-}, 
(17,6);(21,2)**@{-}, 
(13,6);(17,6) **\crv{(15,8)}, 
(13,2);(17,2) **\crv{(15,0)},
(7,7);(9,6) **\crv{(8,7)}, 
(7,1);(9,2) **\crv{(8,1)}, 
(23,7);(21,6)**\crv{(22,7)}, 
(23,1);(21,2) **\crv{(22,1)}, 
(17,4);(21,4) **@{-}, 
(17,4.1);(21,4.1) **@{-}, 
(17,3.9);(21,3.9)**@{-}, 
(17,4.2);(21,4.2) **@{-}, 
(17,3.8);(21,3.8) **@{-},
(10,3) *{\llcorner},  
(12,3) *{\lrcorner}, 
(21,6) *{\llcorner},
(21,2.2) *{\lrcorner},
(27,4)*{]]},
\endxy
&\xy (1,4)*{=x [[},
(9,2);(13,6) **@{-}, 
(9,6);(10.5,4.5) **@{-},
(11.5,3.5);(13,2) **@{-}, 
(17,2);(21,2) **\crv{(19,4)}, 
(17,6);(21,6) **\crv{(19,4)}, 
(13,6);(17,6) **\crv{(15,8)}, 
(13,2);(17,2) **\crv{(15,0)},
(7,7);(9,6) **\crv{(8,7)}, 
(7,1);(9,2) **\crv{(8,1)}, 
(23,7);(21,6) **\crv{(22,7)}, 
(23,1);(21,2) **\crv{(22,1)}, 
(10,3) *{\llcorner},  
(12,3) *{\lrcorner}, 
(21,6) *{\llcorner},
(21,2.2) *{\lrcorner},
(27,4)*{]]},
\endxy
\xy (1,4)*{+ y [[},
(9,2);(13,6) **@{-}, 
(9,6);(10.5,4.5) **@{-},
(11.5,3.5);(13,2) **@{-}, 
(17,2);(17,6) **\crv{(18.5,4)}, 
(21,6);(21,2) **\crv{(19,4)}, 
(13,6);(17,6) **\crv{(15,8)}, 
(13,2);(17,2) **\crv{(15,0)},
(7,7);(9,6) **\crv{(8,7)}, 
(7,1);(9,2) **\crv{(8,1)}, 
(23,7);(21,6)**\crv{(22,7)}, 
(23,1);(21,2) **\crv{(22,1)}, 
(10,3) *{\llcorner},  
(12,3) *{\lrcorner}, 
(21,6) *{\llcorner},
(21,2.2) *{\lrcorner},
(27,4)*{]]},
\endxy\\
&\xy (51,4)*{=x [[},
(59,2);(63,2) **\crv{(61,4)}, 
(59,6);(63,6) **\crv{(61,4)}, 
(67,2);(71,6) **@{-},
(67,6);(68.5,4.5) **@{-}, 
(69.5,3.5);(71,2) **@{-}, 
(63,6);(67,6)**\crv{(65,8)}, 
(63,2);(67,2) **\crv{(65,0)}, 
(57,7);(59,6)**\crv{(58,7)}, 
(57,1);(59,2) **\crv{(58,1)}, 
(73,7);(71,6)**\crv{(72,7)}, 
(73,1);(71,2) **\crv{(72,1)}, 
(59.5,2.5) *{\llcorner},  
(59,6) *{\lrcorner}, 
(71,6) *{\llcorner},
(71,2.2) *{\lrcorner},
(77,4)*{]]},
\endxy
\xy (51,4)*{+ y [[},
(59,2);(59,6) **\crv{(61,4)}, 
(63,6);(63,2) **\crv{(61.5,4)},
(67,2);(71,6) **@{-},
(67,6);(68.5,4.5) **@{-}, 
(69.5,3.5);(71,2) **@{-}, 
(63,6);(67,6)**\crv{(65,8)}, 
(63,2);(67,2) **\crv{(65,0)}, 
(57,7);(59,6)**\crv{(58,7)}, 
(57,1);(59,2) **\crv{(58,1)}, 
(73,7);(71,6)**\crv{(72,7)}, 
(73,1);(71,2) **\crv{(72,1)}, 
(59.5,2.5) *{\llcorner},  
(59,6) *{\lrcorner}, 
(71,6) *{\llcorner},
(71,2.2) *{\lrcorner},
(77,4)*{]]},
\endxy
\xy (51,4)*{= [[},
(59,2);(63,6) **@{-}, 
(59,6);(63,2) **@{-}, 
(67,2);(71,6) **@{-},
(67,6);(68.5,4.5) **@{-}, 
(69.5,3.5);(71,2) **@{-}, 
(63,6);(67,6)**\crv{(65,8)}, 
(63,2);(67,2) **\crv{(65,0)}, 
(57,7);(59,6)**\crv{(58,7)}, 
(57,1);(59,2) **\crv{(58,1)}, 
(73,7);(71,6)**\crv{(72,7)}, 
(73,1);(71,2) **\crv{(72,1)}, 
(63,4);(59,4) **@{-}, 
(63,4.1);(59,4.1) **@{-}, 
(63,3.9);(59,3.9)**@{-}, 
(63,4.2);(59,4.2) **@{-}, 
(63,3.8);(59,3.8) **@{-},
(59.5,2.5) *{\llcorner},  
(59,6) *{\lrcorner}, 
(71,6) *{\llcorner},
(71,2.2) *{\lrcorner},
(77,4)*{]]},
\endxy
\end{align*}
\caption{Invariance of $[[\cdot]]$ under the moves $\Gamma_4, \Gamma'_4$ and $\Gamma_5$}\label{fig-4-5}
\end{figure}
 

\section{Ideal coset invariants for surface-links}\label{sect-ici-sl}

In this section, we formulate a construction of ideal coset invariants for oriented surface-links by means of the polynomial invariants $\ll\cdot\gg$ for oriented marked graphs in 3-space associated with oriented link invariants. When we forget orientations, this formulation also gives a refinement of the construction of ideal coset invariants for surface-links given in \cite{JKL} with a simplification that is more applicable in practice. We describe a way how to find a unique representative of an ideal coset in terms of the polynomial $\ll D\gg$ of a marked graph diagram $D$, which is an invariant of the surface-link presented by $D$.

\subsection{Construction of ideal coset invariants}\label{subs-const-ici}

An {\it oriented $n$-tangle diagram} ($n\geq 1$) is an oriented link diagram $\mathcal T$ in the rectangle $I^2=[0,1]\times [0,1]$ in $\mathbb R^2$ such that $\mathcal T$ transversely intersect with $(0,1)\times\{0\}$ and $(0,1)\times\{1\}$ in $n$ distinct points, respectively. 

\begin{figure}[t]
\centerline{\xy
(-8,8);(8,8) **@{-},
(-8,-8);(8,-8) **@{-},
(8,-8);(8,8) **@{-},
(-8,-8);(-8,8) **@{-},
(-4,5);(-4,8) **@{-}, (-3.95,6.7) *{\wedge},
(-4,-5);(-4,-8) **@{-}, (-3.97,-6) *{\wedge},
(0,5);(0,8) **@{-}, (0.1,6) *{\vee},
(0,-5);(0,-8) **@{-}, (0.1,-7) *{\vee},
(4,5);(4,8) **@{-}, (4.05,6.7) *{\wedge},
(4,-5);(4,-8) **@{-}, (4.04,-6) *{\wedge},
(0,-12) *{(a)},
\endxy
\qquad\qquad\qquad
\xy
(-8,8);(12,8) **@{-},
(-8,-8);(12,-8) **@{-},
(12,-8);(12,8) **@{-},
(-8,-8);(-8,8) **@{-},
(-4,5);(-4,8) **@{-}, (-3.95,6.7) *{\wedge},
(-4,-5);(-4,-8) **@{-}, (-3.97,-6) *{\wedge},
(0,5);(0,8) **@{-}, (0.1,6) *{\vee},
(0,-5);(0,-8) **@{-}, (0.1,-7) *{\vee},
(4,5);(4,8) **@{-}, (4.05,6.7) *{\wedge},
(4,-5);(4,-8) **@{-}, (4.04,-6) *{\wedge},
(8,5);(8,8) **@{-}, (8.1,6) *{\vee},
(8,-5);(8,-8) **@{-}, (8.1,-7) *{\vee},
(2,-12) *{(b)},
\endxy}
\caption{Orientations of arcs in $3, 4$-tangle diagrams}\label{fig-bdary-t}
\end{figure}

\begin{figure}[h]
\begin{center}
\resizebox{0.50\textwidth}{!}{%
  \includegraphics{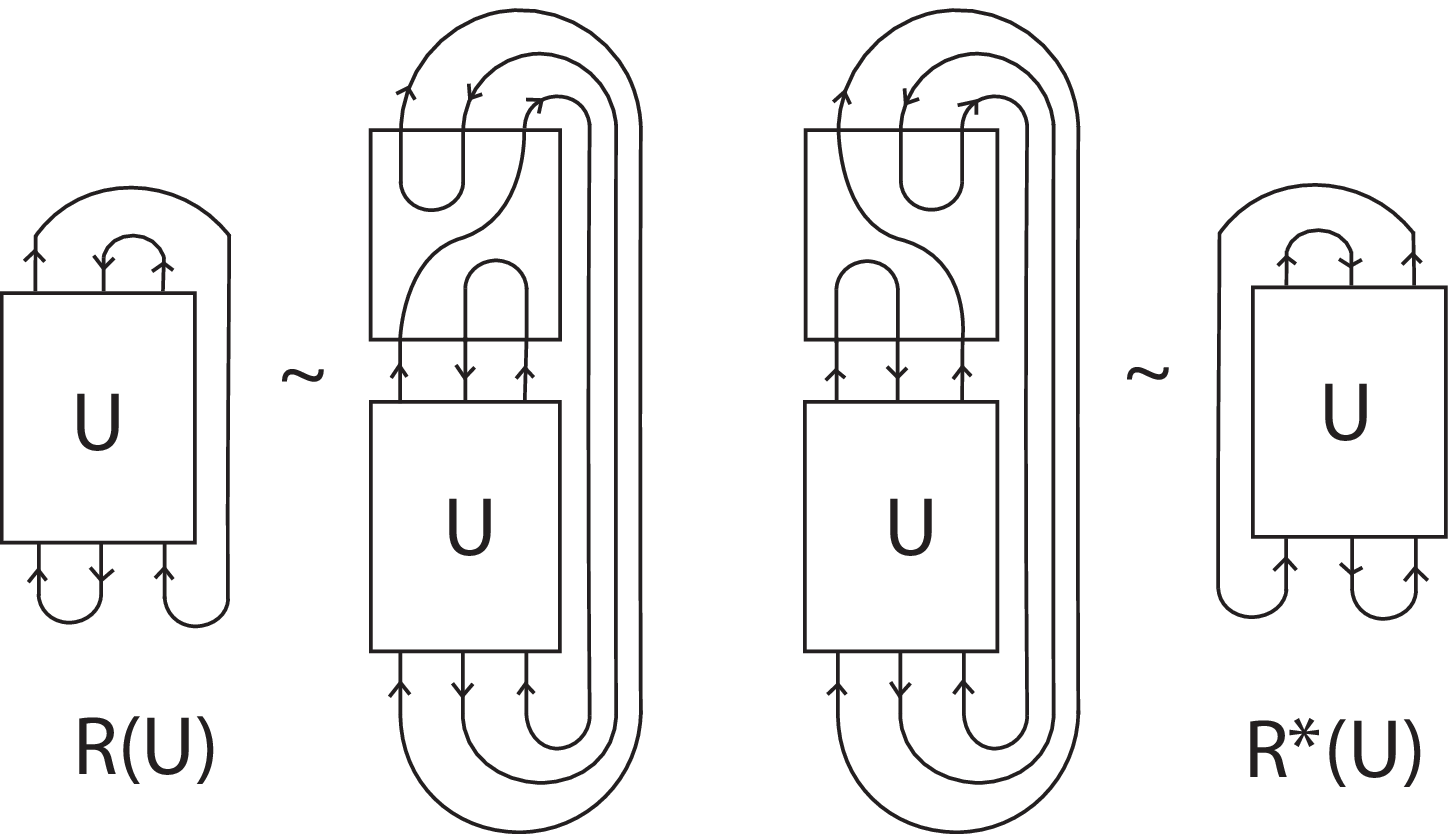}}
\caption{Closing operations $R$ and $R^*$ of a $3$-tangle $U$} \label{fig-m07-gid}
\end{center}
\begin{center}
\resizebox{0.30\textwidth}{!}{%
\includegraphics{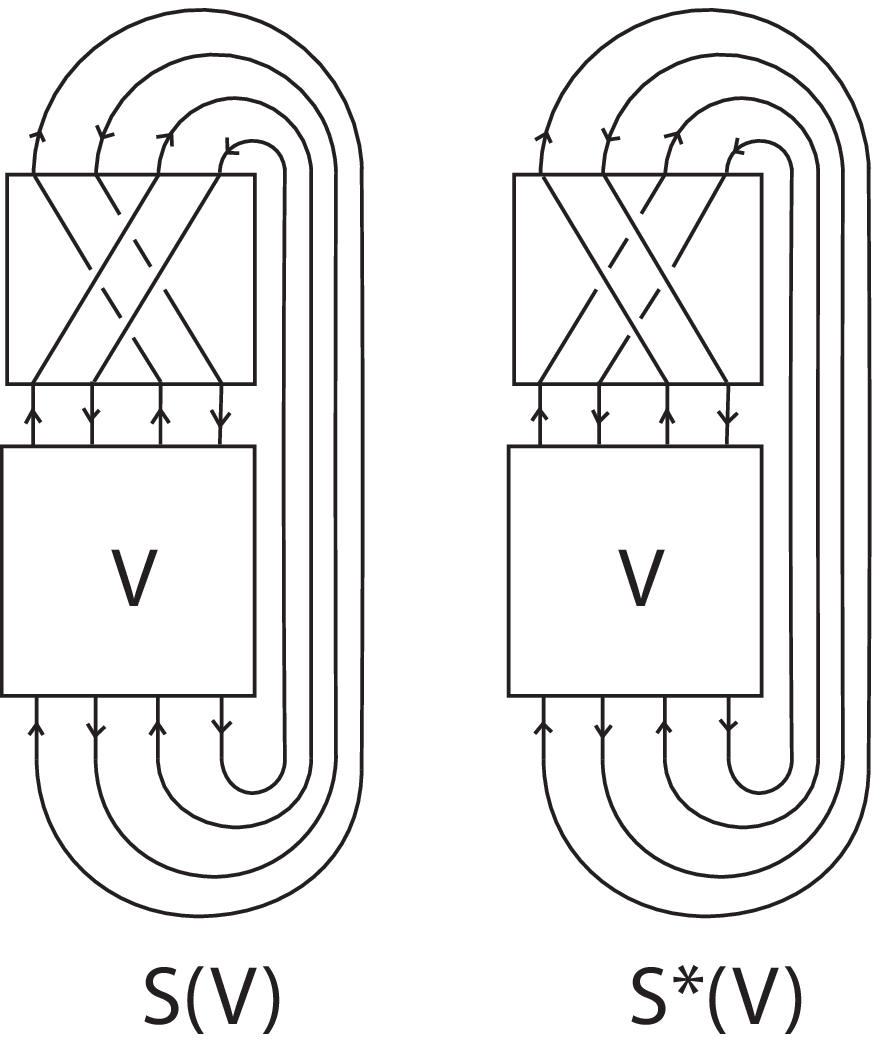}}
\caption{Closing operations $S$ and $S^*$ of a $4$-tangle $V$} \label{fig-m08-gid}
\end{center}
\end{figure}	

Let $\mathcal T^{\rm ori}_3$ and $\mathcal T^{\rm ori}_4$ denote the set of all oriented $3$- and $4$-tangle diagrams such that the orientations of the arcs of the tangles intersecting the boundary of $I^2$ coincide with the orientations as shown in (a) and (b) of Figure~\ref{fig-bdary-t}, respectively.
For $U \in \mathcal T^{\rm ori}_3$ and $V \in \mathcal T^{\rm ori}_4$, let $R(U), R^*(U), S(V)$ and $S^*(V)$ denote the oriented link diagrams obtained from the tangles $U$ and $V$ by {\it closing} $n$ arcs as shown in Figures~\ref{fig-m07-gid} and \ref{fig-m08-gid}.

Let $\mathcal T_3$ and $\mathcal T_4$ denote the set of all $3$- and $4$-tangle diagrams without orientations, respectively. For $U \in \mathcal T_3$ and $V \in \mathcal T_4$, let $R(U), R^*(U), S(V)$ and $S^*(V)$ be the link diagrams obtained by the same way as above forgetting orientations. Then we have the following three lemmas:

\begin{lemma}\label{lem1-Ideal-inv}
For the oriented Yoshikawa moves $\Gamma_6$ and $\Gamma'_6$, we have
\begin{align}
\ll\xy
(1,3);(6,-2) **@{-}, 
(1,-3);(6,2) **@{-},
(6,2);(10,2) **\crv{(8,4)}, 
(6,-2);(10,-2)**\crv{(8,-4)}, 
(10,2);(10,-2) **\crv{(11.5,0)}, 
(4,2);(4,-2) **@{-}, 
(4.1,2);(4.1,-2) **@{-}, 
(3.9,2);(3.9,-2)**@{-},
(1.5,-3) *{\urcorner}, 
(1.5,1.9) *{\ulcorner},
(9.5,-3)*{\urcorner},
\endxy\gg &=
(\delta x+y)\ll\xy 
(1.5,-3);(1.5,3) **\crv{(4,0)},
(2.4,1.5) *{\ulcorner},  
\endxy\gg,\label{pf-inv-y6-1}\\
\ll\xy 
(1,3);(6,-2) **@{-}, 
(1,-3);(6,2) **@{-},
(6,2);(10,2) **\crv{(8,4)}, 
(6,-2);(10,-2)**\crv{(8,-4)}, 
(10,2);(10,-2) **\crv{(11.5,0)}, 
(2,0);(6,0) **@{-}, 
(2,0.1);(6,0.1) **@{-},
(2,0.2);(6,0.2) **@{-}, 
(2,-0.1);(6,-0.1) **@{-},
(2,-0.2);(6,-0.2) **@{-},
(1.5,-3) *{\urcorner}, 
(1.5,1.9) *{\ulcorner},
(9.5,-3)*{\urcorner}, 
\endxy\gg 
&=(x+\delta y)\ll\xy 
(1.5,-3);(1.5,3) **\crv{(4,0)},
(2.4,1.5) *{\ulcorner}, 
\endxy\gg.\label{pf-inv-y6-2}
\end{align}
\end{lemma}	

\begin{proof}
It follows from Definition \ref{defn-poly-ori-skr} together with (\ref{p-clinv-2}) that
\begin{align*}
\Big[\Big[~\xy
(1,3);(6,-2) **@{-}, 
(1,-3);(6,2) **@{-},
(6,2);(10,2) **\crv{(8,4)}, 
(6,-2);(10,-2)**\crv{(8,-4)}, 
(10,2);(10,-2) **\crv{(11.5,0)}, 
(4,2);(4,-2) **@{-}, 
(4.1,2);(4.1,-2) **@{-}, 
(3.9,2);(3.9,-2)**@{-},
(1.5,-3) *{\urcorner}, 
(1.5,1.9) *{\ulcorner},
(9.5,-3)*{\urcorner},
\endxy~\Big]\Big]
&=  \Big[\Big[~\xy 
(1.5,-3);(1.5,3) **\crv{(4,0)},  
(6,2);(6,-2) **\crv{(4.5,0)},
(6,2);(10,2) **\crv{(8,4)},
(6,-2);(10,-2) **\crv{(8,-4)}, 
(10,2);(10,-2) **\crv{(11.5,0)},
(2.4,1.5) *{\ulcorner},
\endxy~\Big]\Big] x
+  \Big[\Big[~\xy 
(1.5,-3);(1.5,3) **\crv{(4,0)},
(2.4,1.5) *{\ulcorner},  
\endxy~\Big]\Big] y
=(\delta x+y)\Big[\Big[~\xy 
(1.5,-3);(1.5,3) **\crv{(4,0)},  
(2.4,1.5) *{\ulcorner},
\endxy~\Big]\Big],\\
\Big[\Big[~\xy 
(1,3);(6,-2) **@{-}, 
(1,-3);(6,2) **@{-},
(6,2);(10,2) **\crv{(8,4)}, 
(6,-2);(10,-2)**\crv{(8,-4)}, 
(10,2);(10,-2) **\crv{(11.5,0)}, 
(2,0);(6,0) **@{-}, 
(2,0.1);(6,0.1) **@{-},
(2,0.2);(6,0.2) **@{-}, 
(2,-0.1);(6,-0.1) **@{-},
(2,-0.2);(6,-0.2) **@{-},
(1.5,-3) *{\urcorner}, 
(1.5,1.9) *{\ulcorner},
(9.5,-3)*{\urcorner},
\endxy~\Big]\Big] 
&=  \Big[\Big[~\xy 
(1.5,-3);(1.5,3) **\crv{(4,0)}, 
(2.4,1.5) *{\ulcorner}, 
\endxy~\Big]\Big] x
+  \Big[\Big[~\xy 
(6,2);(6,-2) **\crv{(4.5,0)},
(6,2);(10,2) **\crv{(8,4)},
(6,-2);(10,-2) **\crv{(8,-4)}, 
(10,2);(10,-2) **\crv{(11.5,0)},
(1.5,-3);(1.5,3) **\crv{(4,0)},
(2.4,1.5) *{\ulcorner},  
\endxy~\Big]\Big] y
=(x+\delta y)\Big[\Big[~\xy 
(1.5,-3);(1.5,3) **\crv{(4,0)}, 
(2.4,1.5) *{\ulcorner}, 
\endxy~\Big]\Big].
\end{align*}
This observation and the fact that the writhe $w(D)$ of an oriented marked graph diagram $D$ is invariant under $\Gamma_6$ and $\Gamma'_6$ (cf. Remark \ref{rmk-self-ori-wr}) yields that
\begin{align*}
\ll\xy
(1,3);(6,-2) **@{-}, 
(1,-3);(6,2) **@{-},
(6,2);(10,2) **\crv{(8,4)}, 
(6,-2);(10,-2)**\crv{(8,-4)}, 
(10,2);(10,-2) **\crv{(11.5,0)}, 
(4,2);(4,-2) **@{-}, 
(4.1,2);(4.1,-2) **@{-}, 
(3.9,2);(3.9,-2)**@{-},
(1.5,-3) *{\urcorner}, 
(1.5,1.9) *{\ulcorner},
(9.5,-3)*{\urcorner},
\endxy\gg &=
\alpha^{-w(~\xy
(1,3);(6,-2) **@{-}, 
(1,-3);(6,2) **@{-},
(6,2);(10,2) **\crv{(8,4)}, 
(6,-2);(10,-2)**\crv{(8,-4)}, 
(10,2);(10,-2) **\crv{(11.5,0)}, 
(4,2);(4,-2) **@{-}, 
(4.1,2);(4.1,-2) **@{-}, 
(3.9,2);(3.9,-2)**@{-},
(1.5,-3) *{\urcorner}, 
(1.5,1.9) *{\ulcorner},
(9.5,-3)*{\urcorner},
\endxy~)}\Big[\Big[~\xy
(1,3);(6,-2) **@{-}, 
(1,-3);(6,2) **@{-},
(6,2);(10,2) **\crv{(8,4)}, 
(6,-2);(10,-2)**\crv{(8,-4)}, 
(10,2);(10,-2) **\crv{(11.5,0)}, 
(4,2);(4,-2) **@{-}, 
(4.1,2);(4.1,-2) **@{-}, 
(3.9,2);(3.9,-2)**@{-},
(1.5,-3) *{\urcorner}, 
(1.5,1.9) *{\ulcorner},
(9.5,-3)*{\urcorner},
\endxy~\Big]\Big]\notag\\
&=\alpha^{-w(~\xy 
(1.5,-3);(1.5,3) **\crv{(4,0)},
(2.4,1.5) *{\ulcorner},  
\endxy~)}
(\delta x+y)\Big[\Big[~\xy 
(1.5,-3);(1.5,3) **\crv{(4,0)},
(2.4,1.5) *{\ulcorner},  
\endxy~\Big]\Big]
=(\delta x+y)\ll\xy 
(1.5,-3);(1.5,3) **\crv{(4,0)},
(2.4,1.5) *{\ulcorner},  
\endxy\gg,\\
\ll\xy 
(1,3);(6,-2) **@{-}, 
(1,-3);(6,2) **@{-},
(6,2);(10,2) **\crv{(8,4)}, 
(6,-2);(10,-2)**\crv{(8,-4)}, 
(10,2);(10,-2) **\crv{(11.5,0)}, 
(2,0);(6,0) **@{-}, 
(2,0.1);(6,0.1) **@{-},
(2,0.2);(6,0.2) **@{-}, 
(2,-0.1);(6,-0.1) **@{-},
(2,-0.2);(6,-0.2) **@{-},
(1.5,-3) *{\urcorner}, 
(1.5,1.9) *{\ulcorner},
(9.5,-3)*{\urcorner}, 
\endxy\gg 
&=\alpha^{-w(~\xy 
(1,3);(6,-2) **@{-}, 
(1,-3);(6,2) **@{-},
(6,2);(10,2) **\crv{(8,4)}, 
(6,-2);(10,-2)**\crv{(8,-4)}, 
(10,2);(10,-2) **\crv{(11.5,0)}, 
(2,0);(6,0) **@{-}, 
(2,0.1);(6,0.1) **@{-},
(2,0.2);(6,0.2) **@{-}, 
(2,-0.1);(6,-0.1) **@{-},
(2,-0.2);(6,-0.2) **@{-},
(1.5,-3) *{\urcorner}, 
(1.5,1.9) *{\ulcorner},
(9.5,-3)*{\urcorner}, 
\endxy~)}\Big[\Big[~\xy 
(1,3);(6,-2) **@{-}, 
(1,-3);(6,2) **@{-},
(6,2);(10,2) **\crv{(8,4)}, 
(6,-2);(10,-2)**\crv{(8,-4)}, 
(10,2);(10,-2) **\crv{(11.5,0)}, 
(2,0);(6,0) **@{-}, 
(2,0.1);(6,0.1) **@{-},
(2,0.2);(6,0.2) **@{-}, 
(2,-0.1);(6,-0.1) **@{-},
(2,-0.2);(6,-0.2) **@{-},
(1.5,-3) *{\urcorner}, 
(1.5,1.9) *{\ulcorner},
(9.5,-3)*{\urcorner}, 
\endxy~\Big]\Big]\notag\\
&= \alpha^{-w(~\xy 
(1.5,-3);(1.5,3) **\crv{(4,0)},
(2.4,1.5) *{\ulcorner}, 
\endxy~)}
(x+\delta y)\Big[\Big[~\xy 
(1.5,-3);(1.5,3) **\crv{(4,0)},
(2.4,1.5) *{\ulcorner}, 
\endxy~\Big]\Big]
=(x+\delta y)\ll\xy 
(1.5,-3);(1.5,3) **\crv{(4,0)},
(2.4,1.5) *{\ulcorner}, 
\endxy\gg.
\end{align*}
This completes the proof.
\end{proof}

\begin{lemma}\label{lem2-Ideal-inv}
Let $D$ be an oriented marked graph diagram and let $D'$ be an oriented marked graph diagram obtained from $D$ by applying a single oriented Yoshikawa move $\Gamma_7$. Then
\begin{align*}
\ll D'\gg-\ll D\gg &=\alpha^{-w(D)}\sum_{k=1}^m\psi_k(x,y)\Big([R(U_k)]-[R^*(U_k)]\Big)xy,
  \end{align*}
  where $U_k \in \mathcal T^{ori}_3 (k=1,2,\ldots,m)$ and $\psi_k(x,y)$ is a polynomial in $R[x,y].$
\end{lemma}

\begin{proof}
Let $D$ be an oriented marked graph diagram and let $D'$ be an oriented marked graph diagram obtained from $D$ by applying a single oriented Yoshikawa move $\Gamma_7$ as shown in Figure~\ref{fig-m07-ort}. 
\begin{figure}[t]
\begin{center}
\resizebox{0.40\textwidth}{!}{%
  \includegraphics{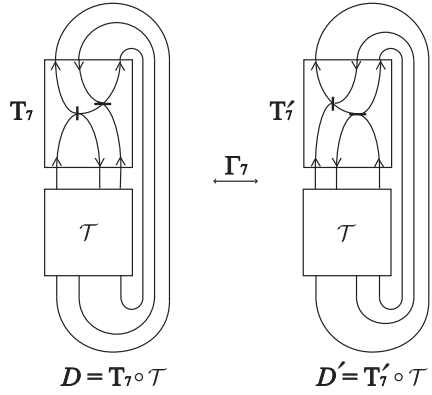}}
\caption{Oriented Yoshikawa move $\Gamma_7$} \label{fig-m07-ort}
\end{center}
\end{figure}
Applying the axioms ({\bf L1}) and ({\bf L2}) in Definition \ref{defn-poly-ori-skr} to the 3-tangle diagram $\mathcal T$ in $D=T_7\circ \mathcal T$, we can express $[[D]]$ as a linear combination of polynomials $[[T_7\circ U_k]] (1 \leq k \leq m)$ for some integer $m \geq 1$, where each $U_k$ is an oriented $3$-tangle diagram that has no marked vertices and hence $U_k \in \mathcal T^{ori}_3$. Hence 
$$[[D]]=[[T_7\circ\mathcal T]]=\sum_{k=1}^m\psi_k(x,y)[[T_7 \circ U_k]],
$$
where $\psi_k(x,y)$ is a polynomial in $R[x,y].$ By applying the same procedure to $\mathcal T$ in $D'=T'_7\circ \mathcal T$, we have
$$[[D']] =[[T'_7\circ\mathcal T]]=\sum_{k=1}^m\psi_k(x,y)[[T'_7 \circ U_k]].$$
By a straightforward computation, we obtain 
\begin{align*}
[[T_7 \circ U_k]] &=[\vec{f}_2\circ U_k] x^2+[\vec{f}_0\circ U_k] xy+[\vec{f}_4\circ U_k] yx+[\vec{f}_1\circ U_k] y^2,\\
[[T_7' \circ U_k]] &=[\vec{f}_2\circ U_k] x^2+[\vec{f}_0\circ U_k] xy+[\vec{f}_3\circ U_k] yx+[\vec{f}_1\circ U_k] y^2,
\end{align*}
where $\vec{f}_0, \ldots, \vec{f}_4$ are the fundamental oriented 3-tangle diagrams shown in Figure~\ref{fig-3tbas-ori}.
 \begin{figure}[h]
\begin{center}
\resizebox{0.60\textwidth}{!}{%
  \includegraphics{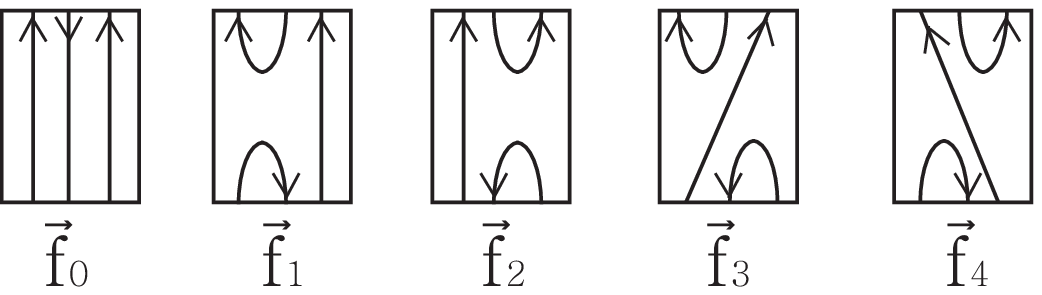} }
\caption{Fundamental oriented $3$-tangle diagrams}\label{fig-3tbas-ori}
\end{center}
\end{figure}

This gives
\begin{align*}
[[D]]-[[D']]&=\sum_{k=1}^m\psi_k(x,y)\Big([[T_7 \circ U_k]]-[[T'_7 \circ U_k]]\Big)\\
&=\sum_{k=1}^m\psi_k(x,y)\Big([\vec{f}_4\circ U_k]-[\vec{f}_3\circ U_k] \Big)xy\\
&=\sum_{k=1}^m(-\psi_k(x,y))\Big([R(U_k)]-[R^*(U_k)]\Big)xy.
\end{align*}
Therefore 
\begin{align*}
\ll D'\gg-\ll D\gg &=\alpha^{-w(D')}[[D']]-\alpha^{-w(D)}[[D]]
=\alpha^{-w(D)}\Big([[D']]-[[D]]\Big)\\
&=\alpha^{-w(D)}\sum_{k=1}^m\psi_k(x,y)\Big([R(U_k)]-[R^*(U_k)]\Big)xy.
  \end{align*}
This completes the proof.
\end{proof}

\begin{lemma}\label{lem3-Ideal-inv}
Let $D$ be an oriented marked graph diagram and let $D'$ be an oriented marked graph diagram obtained from $D$ by applying a single oriented Yoshikawa move $\Gamma_8$. Then
\begin{align*}
\ll D\gg-\ll D'\gg 
&=\alpha^{-w(D)}\sum_{k=1}^n\varphi_k(x,y)\Big([S(V_k)]-[S^*(V_k)]\Big)xy,
\end{align*}
where $V_k \in \mathcal T^{ori}_4 (k=1,2,\ldots,n)$ and $\varphi_k(x,y)$ is a polynomial in $R[x,y].$
\end{lemma}

\begin{proof}
Let $D'$ be an oriented marked graph diagram obtained from $D$ by applying a single oriented Yoshikawa move $\Gamma_8$ as shown in Figure~\ref{fig-m08-ort}. 
\begin{figure}[h]
\begin{center}
\resizebox{0.40\textwidth}{!}{%
  \includegraphics{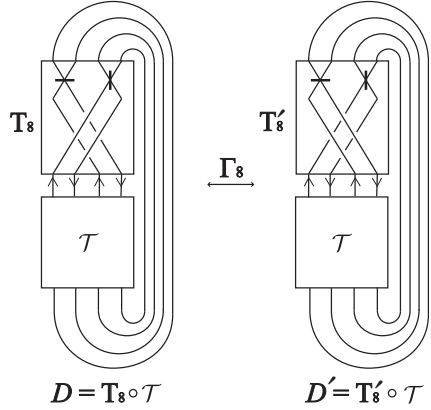}}
\caption{Oriented Yoshikawa move $\Gamma_8$}\label{fig-m08-ort}
\end{center}
\end{figure}
Applying the axioms ({\bf L1}) and ({\bf L2}) in Definition \ref{defn-poly-ori-skr} to the 4-tangle diagram $\mathcal T$ in $D=T_8\circ \mathcal T$, we can express $[[D]]$ as a linear combination of polynomials $[[T_8\circ V_k]] (1 \leq k \leq n)$ for some integer $n \geq 1$, where each $V_k$ is an oriented $4$-tangle diagram that has no marked vertices and hence $V_k \in \mathcal T^{ori}_4$. Hence 
$$[[D]]=[[T_8\circ\mathcal T]]=\sum_{k=1}^n\varphi_k(x,y)[[T_8 \circ V_k]],
$$
where $\varphi_k(x,y)$ is a polynomial in $R[x,y].$ By applying the same procedure to $\mathcal T$ in $D'=T'_8\circ \mathcal T$, we have
$$[[D']] =[[T'_8\circ\mathcal T]]=\sum_{k=1}^n\varphi_k(x,y)[[T'_8 \circ V_k]].$$
By a straightforward computation, we obtain 
\begin{align*}
[[T_8 \circ V_k]] &=[\vec{g}_9\circ V_k] x^2+[\vec{g}_{5}\circ V_k] xy+[\vec{g}\circ V_k] yx+[\vec{g}_{12}\circ V_k] y^2,\\
[[T_8' \circ V_k]] &=[\vec{g}_9\circ V_k] x^2+[\vec{g}_{5}\circ V_k] xy+[\vec{g}^*\circ V_k] yx+[\vec{g}_{12}\circ V_k] y^2,
\end{align*}
where  $\vec{g}_9, \vec{g}_5, \vec{g}_{12}, \vec{g}$ and $\vec{g}^*$ are the fundamental oriented 4-tangle diagrams shown in Figure~\ref{fig-4tbas-ori}.
 \begin{figure}[h]
\begin{center}
\resizebox{0.60\textwidth}{!}{%
  \includegraphics{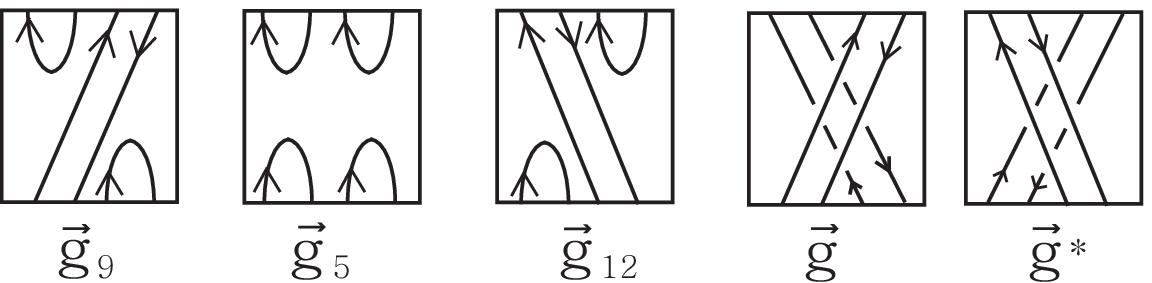} }
\caption{Fundamental oriented $4$-tangle diagrams}\label{fig-4tbas-ori}
\end{center}
\end{figure}

This gives
\begin{align*}
[[D]]-[[D']]&=\sum_{k=1}^n\varphi_k(x,y)\Big([[T_8 \circ V_k]]-[[T'_8 \circ V_k]]\Big)\\
&=\sum_{k=1}^n\varphi_k(x,y)\Big([\vec{g}\circ V_k]-[\vec{g}^*\circ V_k]\Big)xy\\
&=\sum_{k=1}^n\varphi_k(x,y)\Big([S(V_k)]-[S^*(V_k)]\Big)xy.
\end{align*}
Therefore
\begin{align*}
\ll D\gg-\ll D'\gg &=\alpha^{-w(D)}[[D]]-\alpha^{-w(D')}[[D']]=\alpha^{-w(D)}\Big([[D]]-[[D']]\Big)\\
&=\alpha^{-w(D)}\sum_{k=1}^n\varphi_k(x,y)\Big([S(V_k)]-[S^*(V_k)]\Big)xy.
\end{align*}
This completes the proof.
\end{proof}

\begin{definition}\label{defn-obst-id}	
For any given regular or ambient isotopy invariant $$[~~] : \{\text{(resp.~oriented) links in $\mathbb R^3$}\} \longrightarrow R$$ satisfying the properties (\ref{p-clinv-1}) and (\ref{p-clinv-2}) (resp. with orientations), the {\it $[~~]$-obstruction ideal} (simply, {\it $[~~]$ ideal}) or the {\it ideal associated with $[~~]$} is defined to be the ideal of $R[x,y]$ generated by the polynomials in $R[x,y]$:
\begin{equation}\label{eq-polys}
\begin{split}
&P_1(x,y) = \delta x+y-1,\\
&P_2(x,y) = x+\delta y-1,\\
&P_U(x,y)=([R(U)]-[R^*(U)])xy, U \in \mathcal T_3~ (\text{resp.}~\mathcal T^{\rm ori}_3),\\
&P_V(x,y)=([S(V)]-[S^*(V)])xy, V \in \mathcal T_4~ (\text{resp.}~\mathcal T^{\rm ori}_4).
\end{split}
\end{equation}
\end{definition}

We remark that if $[~~]$ is a regular or ambient isotopy invariant for unoriented knots and links in $3$-space, then the $[~~]$ ideal in Definition \ref{defn-obst-id} is identical to the ideal in \cite[Definition 4.1]{JKL} with the specialization given by $(x,y,z,w)=(x,0,0,y)$ and $\phi={\rm identity}$.

\begin{theorem}\label{thm1-Ideal-inv}
Let $[~~]$ be a regular or ambient isotopy invariant for (resp.~oriented) knots and links in the $3$-space $\mathbb R^3$ or $S^3$ and let $I=I([~~])$ denote the ideal associated with $[~~]$. Then the map 
$$\ll~\gg^I : \{\text{(resp.~oriented) marked graph diagrams}\} \longrightarrow R[x,y]/I$$ defined by 
$\ll~\gg^I(D)=\ll D\gg^I:=\ll D\gg +~ I$ for each (resp.~oriented) marked graph diagram $D$ is an invariant for (resp.~oriented) surface-links in the $4$-space $\mathbb R^4$ or $S^4$.
\end{theorem}	

\begin{proof}
By \cite[Theorem 4.2]{JKL}, we see that $\ll~\gg^I$ is an invariant for unoriented surface-links and so the oriented case only remain to be proved.

Let $D$ be an oriented marked graph diagram. By Theorem~\ref{thm-inv-mgs}, we see that $\ll D\gg$ is invariant under the oriented Yoshikawa moves $\Gamma_1, \Gamma'_1, \Gamma_2, \Gamma_3, \Gamma_4, \Gamma'_4$ and $\Gamma_5$ and so clearly do the $[~~]$ ideal coset $\ll D\gg+I$. 

For the moves $\Gamma_6$ and $\Gamma'_6$, it is direct from Lemma \ref{lem1-Ideal-inv} that 
$$\ll\xy
(1,3);(6,-2) **@{-}, 
(1,-3);(6,2) **@{-},
(6,2);(10,2) **\crv{(8,4)}, 
(6,-2);(10,-2)**\crv{(8,-4)}, 
(10,2);(10,-2) **\crv{(11.5,0)}, 
(4,2);(4,-2) **@{-}, 
(4.1,2);(4.1,-2) **@{-}, 
(3.9,2);(3.9,-2)**@{-},
(1.5,-3) *{\urcorner}, 
(1.5,1.9) *{\ulcorner},
(9.5,-3)*{\urcorner},
\endxy\gg+I=
\ll\xy 
(1.5,-3);(1.5,3) **\crv{(4,0)},
(2.4,1.5) *{\ulcorner},  
\endxy\gg+I~\text{and}~
\ll\xy 
(1,3);(6,-2) **@{-}, 
(1,-3);(6,2) **@{-},
(6,2);(10,2) **\crv{(8,4)}, 
(6,-2);(10,-2)**\crv{(8,-4)}, 
(10,2);(10,-2) **\crv{(11.5,0)}, 
(2,0);(6,0) **@{-}, 
(2,0.1);(6,0.1) **@{-},
(2,0.2);(6,0.2) **@{-}, 
(2,-0.1);(6,-0.1) **@{-},
(2,-0.2);(6,-0.2) **@{-},
(1.5,-3) *{\urcorner}, 
(1.5,1.9) *{\ulcorner},
(9.5,-3)*{\urcorner}, 
\endxy\gg +I=\ll\xy 
(1.5,-3);(1.5,3) **\crv{(4,0)},
(2.4,1.5) *{\ulcorner}, 
\endxy\gg + I.$$

Let $D'$ be an oriented marked graph diagram obtained from $D$ by applying a single oriented Yoshikawa move $\Gamma_7$ or $\Gamma_8$. Then it is direct from Lemma \ref{lem2-Ideal-inv} and Lemma \ref{lem3-Ideal-inv} that $\ll D'\gg -\ll D\gg \in I$. This gives $\ll D\gg+I=\ll D'\gg+I$ and completes the proof.
\end{proof}

Let $\mathbb F_R$ be an extension field of $R$. By Hilbert Basis Theorem (cf. \cite{CLO}), the $[~~]$ ideal $I$ in $\mathbb F_R[x,y]$ is completely determined by a finite number of polynomials in $\mathbb F_R[x,y],$ say $p_1, p_2, \ldots, p_r.$ Then $I=<p_1, p_2, \ldots, p_r>$ in $\mathbb F_R[x,y].$

\bigskip

The following corollary is an immediate consequence of Theorem~\ref{thm1-Ideal-inv}.

\begin{corollary}\label{cor1-Ideal-inv}
 Let $S$ be a commutative ring and let $\phi : R[x,y] \to S$ be a homomorphism such that $\phi(I)=0$. Then
the composite map $$\Phi = \phi~\circ\ll~\gg^I : \{\text{(resp.~oriented) marked graph diagrams}\} \longrightarrow S$$ defined by 
$\Phi(D)=\phi(\ll D\gg^I)=\phi(\ll D\gg+I)=\phi(\ll D\gg)$ for each (resp.~oriented) marked graph diagram $D$ is an invariant for (resp.~oriented) surface-links.
\end{corollary}

It is noted that a specialization of the polynomial $\ll~\gg$ valued in some ring $S$ formalized in Corollary \ref{cor1-Ideal-inv} is sometimes useful in practice. In the final section \ref{sect-inv-kbp-spl}, we shall discuss such a specialization of the polynomial $\ll~\gg$ associated with the (normalized) Kauffman bracket polynomial (see Subsection \ref{subsect2-b-poly-MGD}) for (oriented) knots and links in $3$-space, which gives a series of new invariants for (oriented) surface-links in $4$-space.

\subsection{Normal forms of ideal cosets}

In this subsection, we give a brief description of a way how to find a unique representative of the ideal coset invariant $\ll D\gg + I$ in terms of the polynomial $\ll D\gg$ of a marked graph diagram $D$ by using a Groebner basis for the $[~~]$ ideal $I$, which is indeed an invariant of the surface-link in the $4$-space $\mathbb R^4$ or $S^4$ presented by $D$. 

Let $\mathbb F$ be an arbitrary field and let $\mathbb F[x_1,\ldots,x_n]$ be the ring of polynomials in $n$ variables $x_1,\ldots,x_n$ with coefficients in
$\mathbb F.$ Let $I$ denote the ideal of $\mathbb F[x_1,\ldots,x_n]$ generated by the polynomials $f_1, \ldots, f_s$ and let $G=\{g_1, \ldots, g_t\}$ be a Groebner
basis for the ideal $I$ with respect to a fixed monomial order. Then it is well known that for any
polynomial $f \in \mathbb F[x_1,\ldots,x_n],$ there exists a unique polynomial
$r$ with the following properties:
\begin{enumerate}
    \item [(i)] No term of $r$ is divisible by any of
    ${\rm LT}(g_1), \ldots, {\rm LT}(g_t),$ where ${\rm LT}(g_i)$ denotes the leading term of $f$.
    \item [(ii)] There exist $b_i \in \mathbb F[x_1,\ldots,x_n] (1\leq i \leq t)$ such that
    \begin{equation}\label{Division Al-2}
    f =b_1g_1+\cdots+b_tg_t + r.
    \end{equation}
    \item [(iii)] $r$ is the remainder on division of $f$ by
    $G=\{g_1, \ldots, g_t\}$ no matter how the elements of $G$
    are listed when using the Division Algorithm in $\mathbb F[x_1,\ldots,x_n]$.
\end{enumerate}

The unique remainder $r$ in (\ref{Division Al-2}) is called the {\it
normal form} of $f$ on division by the ideal $I$ and denoted by
$\overline{f}^G$. We should note that if $G$ and $G'$ are Groebner basis for the ideal $I$ with respect to the same monomial order in $\mathbb F[x_1,\ldots,x_n]$, then $\overline{f}^G=\overline{f}^{G'}$ for all $f \in \mathbb F[x_1,\ldots,x_n]$ and hence we may denote the unique remainder $\overline{f}^G$ by $\overline{f}^I$ or simply $\overline{f}$ once a monomial order is fixed. For more details, see \cite{JKL} or \cite[Chapter 2]{CLO}. 

\smallskip

Now we turn to the ideal coset invariant for surface-links in the subsection \ref{subs-const-ici}.

\begin{theorem}\label{resid modulo-3}
Let $D$ be an (resp.~oriented) marked graph diagram and let $\mathcal L$ be the (resp.~oriented) surface-link presented by $D$. Let $\ll D\gg$ be the polynomial of $D$ defined in Definition \ref{defn-poly} (resp.~Definition \ref{defn-poly-ori}) and let $I=I([~])$ be the ideal associated with a regular or ambient isotopy invariant $[~~]$ of (resp.~oriented) knots and links in $3$-space. Then for any Groebner basis $G$ for the ideal $I$ with a fixed monomial order, the normal form $\overline{\ll D\gg}$ on division of $\ll D\gg$ in $\mathbb F_R$ by $G$ is uniquely determined by the surface link $\mathcal L$ and therefore it is an invariant of $\mathcal L$. We denote it by $\overline{\ll \mathcal L\gg}$.
\end{theorem}

\begin{proof}
This follows directly from Theorem \ref{thm1-Ideal-inv} and the theory of Groebner bases for ideals in polynomial rings.
\end{proof}

It is worth noting that by using the commercial computer algebra systems ``Maple" or ``Mathematica" one can compute the normal form $\overline{f}$ for any polynomial $f$ in some polynomial rings $\mathbb F[x_1,\ldots,x_n]$ on division by any Groebner basis $G$ for the ideal $I$ with a fixed monomial order such that $f+ I = g+I$ if and only if $\overline{f}=\overline{g}$ for all $f, g \in \mathbb F[x_1,\ldots,x_n]$. 


\section{Kauffman bracket ideal coset invariants}\label{sect-b-poly-MGD}

In this section, we apply the construction formulated in Section \ref{sect-ici-sl} to the Kauffman bracket polynomial for knots and links in the $3$-space $\mathbb R^3$ or $S^3$ and compute the ideal coset invariant for unoriented surface-links in the $4$-space $\mathbb R^4$ or $S^4$ associated with the Kauffman bracket polynomial for unoriented link diagrams and the ideal coset invariant for oriented surface-links associated with the normalized Kauffman bracket polynomial for oriented link diagrams. We also describe how to find unique representatives of the ideal cosets $\ll D\gg+I$ in terms of a Groebner basis for the ideal $I$ and the polynomial $\ll D\gg$.

\subsection{Kauffman bracket ideal coset invariant for unoriented surface-links}\label{subsect1-b-poly-MGD}

Let $L$ be a link diagram. The Kauffman bracket
polynomial of $L$\cite{Kau} is a Laurent polynomial $\langle L\rangle = \langle L \rangle(A) \in R=\mathbb{Z}[A^{\pm 1}]$
defined by the following three rules:
\begin{flalign*}
 &{\rm ({\bf B1})}~ \langle \bigcirc \rangle = 1. &\\
 &{\rm ({\bf B2})}~ \langle \bigcirc\sqcup K \rangle = \delta \langle K\rangle,
 ~\text{where}~\delta=-A^2-A^{-2}.&\\
 &{\rm ({\bf B3})} ~
 \Big\langle~\xy 
 (4,4);(-4,-4) **@{-},
 (-4,4);(-1,1) **@{-}, 
 (1,-1);(4,-4) **@{-},  
\endxy~\Big\rangle 
 = A\Big\langle~\xy 
 (-4,4);(-4,-4) **\crv{(1,0)}, 
 (4,4);(4,-4) **\crv{(-1,0)},  
\endxy~\Big\rangle 
+ A^{-1}\Big\langle~\xy 
(-4,4);(4,4) **\crv{(0,-1)}, 
(4,-4);(-4,-4) **\crv{(0,1)},
\endxy~\Big\rangle.
\end{flalign*}
Note that the Kauffman bracket polynomial is a regular isotopy invariant for unoriented links, namely it is invariant under Reidemeister moves $\Omega_2$ and $\Omega_3$ except the move $\Omega_1$. For $\Omega_1$, we have
\begin{eqnarray*}
\Big\langle~\xy (1,-3);(6,2) **@{-},
(1,3);(3.5,0.5) **@{-},
(4.5,-0.5);(6,-2) **@{-},
(6,2);(10,2) **\crv{(8,4)},
(6,-2);(10,-2) **\crv{(8,-4)}, 
(10,2);(10,-2) **\crv{(11.5,0)}, 
\endxy~\Big\rangle
=-A^3\Big\langle
~~\xy (1.5,-3);(1.5,3) **\crv{(4,0)},  
\endxy~~\Big\rangle,~~
\Big\langle~\xy (4.5,0.6);(6,2) **@{-},
(1,-3);(3.3,-0.7) **@{-},
(1,3);(3.5,0.5) **@{-},
(3.5,0.5);(6,-2) **@{-},
(6,2);(10,2) **\crv{(8,4)},
(6,-2);(10,-2) **\crv{(8,-4)}, 
(10,2);(10,-2) **\crv{(11.5,0)}, 
\endxy~\Big\rangle 
= -A^{-3}\Big\langle~~\xy 
(1.5,-3);(1.5,3) **\crv{(4,0)},  
\endxy~~\Big\rangle.
\end{eqnarray*}

Now let $D$ be a marked graph diagram. Then it follows from Definitions \ref{defn-poly-skr} and \ref{defn-poly} that $\ll D\gg$ is a polynomial in $\mathbb Z[A^{\pm 1}][x,y](=\mathbb Z[A^{\pm 1},x,y])$ given by 
\begin{equation}\label{defn-poly-inv-kb}
\ll D\gg=(-A^3)^{-sw(D)}[[D]],
\end{equation}
where $[[D]]=[[ D ]](A,A^{-1},x,y)$ is a polynomial defined by the two axioms:
\begin{itemize}
\item[{\rm ({\bf L1})}] 
$[[D]]= \langle D\rangle$ 
if $D$ is a knot or link diagram.
\item[{\rm ({\bf L2})}]  
$\Big[\Big[~\xy (-4,4);(4,-4) **@{-}, 
(4,4);(-4,-4) **@{-},  
(3,-0.2);(-3,-0.2) **@{-},
(3,0);(-3,0) **@{-}, 
(3,0.2);(-3,0.2) **@{-}, 
\endxy~\Big]\Big] =
\Big[\Big[~\xy (-4,4);(4,4) **\crv{(0,-1)}, 
(4,-4);(-4,-4) **\crv{(0,1)}, 
\endxy~\Big]\Big] x
+ \Big[\Big[~\xy (-4,4);(-4,-4) **\crv{(1,0)},  
(4,4);(4,-4) **\crv{(-1,0)},  
\endxy~\Big]\Big]y.$
\end{itemize}
Also, the state-sum formula for the polynomial
$\ll D\gg$ in (\ref{state formula}) is read as
\begin{equation*}
\ll D \gg=(-A^3)^{-sw(D)}\sum_{\sigma \in \mathfrak S(D)}\langle D_\sigma \rangle
x^{\sigma(\infty)}y^{\sigma(0)}.
\end{equation*}

\begin{theorem}\label{thm-inv-kb}
(1) The Kauffman bracket ideal $I$ is the ideal of the ring $\mathbb Z[A^{\pm 1},x,y]$ generated by the following three polynomials:
\begin{equation}\label{gens-kbp-ideal}
\begin{split}
&f_1(A,A^{-1},x,y) = (-A^2-A^{-2})x+y-1,\\
&f_2(A,A^{-1},x,y) = x+(-A^2-A^{-2})y-1,\\
&f_3(A,A^{-1},x,y)=(A^4+1+A^{-4})xy.
\end{split}
\end{equation}

(2) For any marked graph diagram $D$, the ideal coset $\ll D\gg^I =\ll D\gg +~ I$  is an invariant for the unoriented surface-link in $\mathbb R^4$ or $S^4$ presented by $D$.
\end{theorem}

\begin{proof}
(1) Since $\delta=-A^2-A^{-2}$, the polynomials $P_1(x,y)$ and $P_2(x,y)$ in (\ref{eq-polys}) are the polynomials $f_1(A,A^{-1},x,y)$ and $f_2(A,A^{-1},x,y)$ in (\ref{gens-kbp-ideal}), respectively.

Let $U$ be any $3$-tangle diagram in $\mathcal T_3$. Applying the Kauffman bracket axioms ({\bf B1})--({\bf B3}) to the 3-tangle diagram $U$ in $R(U)$ and $R^*(U)$, we have
$$\langle R(U)\rangle=\sum_{i=0}^5\psi_i(U)\langle R(f_i)\rangle,~
\langle R^*(U)\rangle=\sum_{i=0}^5\psi_i(U)\langle R^*(f_i)\rangle,
$$
where $\psi_i(U) \in \mathbb Z[A^{\pm 1},x,y]$ and $f_i (0\leq i \leq 4)$ denotes the fundamental oriented $3$-tangle diagram $\vec{f}_i$ in Figure~\ref{fig-3tbas-ori} without orientation. It is easy to check that
\begin{equation*}
\begin{split}
\langle R(f_i)\rangle -\langle R^*(f_i)\rangle 
=\left\{
    \begin{array}{ll}
      0, & \hbox{$i=0, 1, 2;$} \\
  \delta^{2}-1, & \hbox{$i=3$;} \\
  1-\delta^{2}, & \hbox{$i=4$.} 
    \end{array}
  \right.
\end{split}
\end{equation*}
From (\ref{eq-polys}), we have
\begin{align}\label{eq2-pf-m7-1}
&P_U(x,y)=(\langle R(U)\rangle-\langle R^*(U)\rangle)xy
=xy\sum_{i=0}^5\psi_i(U)\Big(\langle R(f_i)\rangle -\langle R^*(f_i)\rangle\Big)\notag\\
&=(\delta^{2}-1)xy\Big(\psi_3(U)-\psi_4(U)\Big)
=(A^4+1+A^{-4})xy\Big(\psi_3(U)-\psi_4(U)\Big).
\end{align}

Similarly, let $V$ be any $4$-tangle diagram in $\mathcal T_4$.
Applying the Kauffman bracket axioms ({\bf B1})--({\bf B3}) to the 4-tangle diagram $V$ in $S(V)$ and $S^*(V)$, we have
$$\langle S(V)\rangle=\sum_{i=0}^{13}\varphi_i(V)\langle S(g_i)\rangle,~
\langle S^*(V)\rangle=\sum_{i=0}^{13}\varphi_i(V)\langle S^*(g_i)\rangle,
$$ where $\varphi_i(V) \in \mathbb Z[A^{\pm 1},x,y]$ and $g_i (0\leq i \leq 13)$ is the fundamental $4$-tangle diagram shown in Figure~\ref{fig-4tbas}.
\begin{figure}[h]
\begin{center}
\resizebox{0.60\textwidth}{!}{%
  \includegraphics{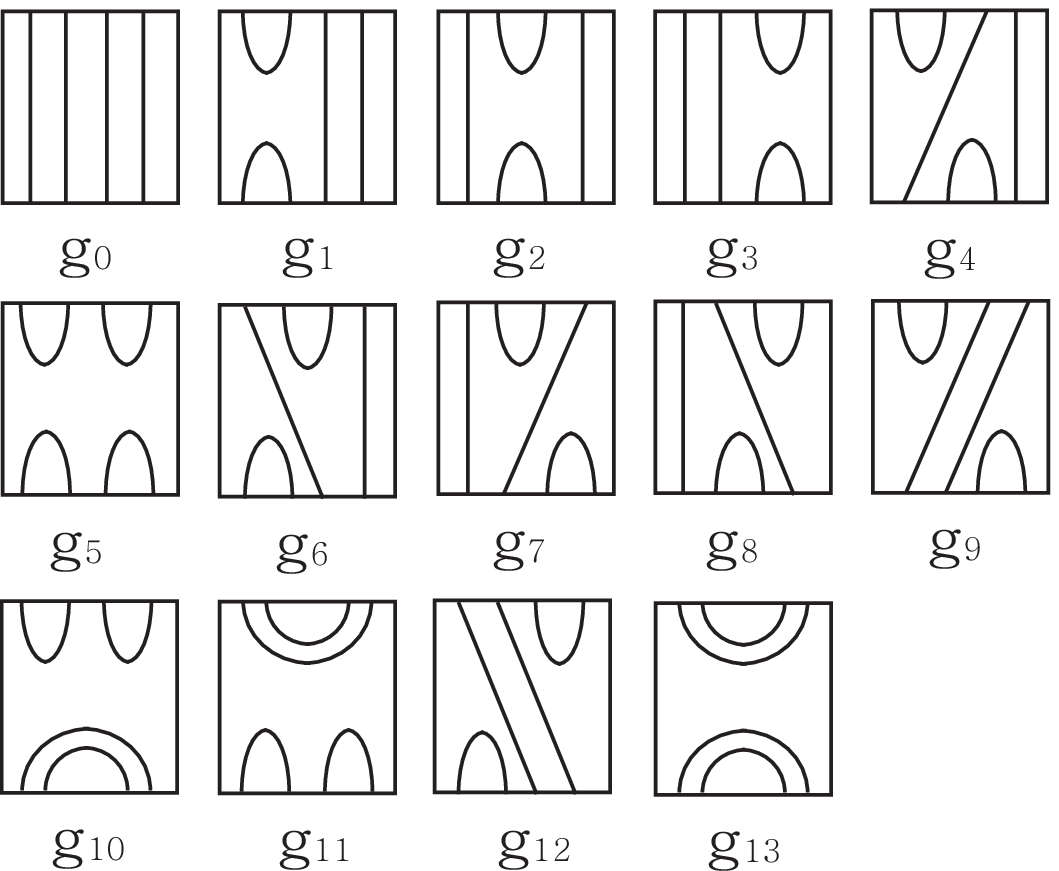} }
\caption{Fundamental $4$-tangle diagrams} \label{fig-4tbas}
\end{center}
\end{figure}
It is easy but tedious to check (cf. \cite[Lemma 4.2]{Le2}) that
\begin{equation*}
\begin{split}
\langle S(g_i)\rangle -\langle S^*(g_i)\rangle 
=\left\{
    \begin{array}{ll}
      -xyA^{-10}(A^{20}+A^{12}-A^{8}-1), & \hbox{$i=0;$} \\
      ~~xyA^{-10}(A^{20}+A^{12}-A^{8}-1), & \hbox{$i=13;$} \\
      ~~0 & \hbox{otherwise.}
    \end{array}
  \right.
\end{split}
\end{equation*}
From (\ref{eq-polys}), we have
\begin{align}\label{eq2-pf-m8-1}
P_V(x,y)&=(\langle S(V)\rangle-\langle S^*(V)\rangle)xy
=xy\sum_{i=0}^{13}\varphi_i(V)\Big(\langle S(g_i)\rangle -\langle S^*(g_i)\rangle\Big)\notag\\
&=-A^{-10}(A^{20}+A^{12}-A^{8}-1)xy\Big(\varphi_0(V)-\varphi_{13}(V)\Big)\notag\\
&=-A^{-6}(A^4+1+A^{-4})xy(A^{12}-A^{8}+A^{4}-1)\Big(\varphi_0(V)-\varphi_{13}(V)\Big).
\end{align}
By (\ref{eq2-pf-m7-1}) and (\ref{eq2-pf-m8-1}), we obtain that for all $U \in \mathcal T_3$ and $V \in \mathcal T_4$, the polynomials $P_U(x,y)$ and $P_V(x,y)$ are generated by $f_3(A,A^{-1},x,y)=(A^4+1+A^{-4})xy$. Therefore $I=<P_1, P_2, \{P_U, P_V | U \in \mathcal T_3, V \in \mathcal T_4\}>=<f_1, f_2, f_3>$. 

The assertion (2) is direct from Theorem~\ref{thm1-Ideal-inv} by taking $[~~]$ to be the Kauffman bracket $\langle~~\rangle$. This completes the proof.
\end{proof}


\subsection{Normalized Kauffman bracket ideal coset invariant for oriented surface-links}\label{subsect2-b-poly-MGD}

Let $L$ be an oriented link diagram and let $\widetilde{L}$ be the link diagram $L$ without orientation. The normalized Kauffman bracket polynomial $\langle L\rangle_N$ of $L$ is defined by 
\begin{equation}\label{nkbp}
\langle L\rangle_N=(-A^3)^{-w(L)}\langle \widetilde{L}\rangle
\end{equation}
and it is an invariant of the oriented link in $\mathbb R^3$ or $S^3$ presented by $L$. It is well known that 
\begin{equation}\label{A2-R1-move-1}
A^{4}\langle~\xy 
(5,5);(-5,-5) **@{-} ?<*\dir{<},
(-5,5);(-2,2) **@{-} ?<*\dir{<}, 
(2,-2);(5,-5) **@{-},
\endxy~\rangle_N - A^{-4} \langle~\xy 
(-5,5);(5,-5) **@{-} ?<*\dir{<},
(5,5);(2,2) **@{-} ?<*\dir{<}, 
(-2,-2);(-5,-5) **@{-},
\endxy~\rangle_N = (A^{-2}-A^2)
\langle~\xy (-4,4);(-4,-4) **\crv{(1,0)},  
(4,4);(4,-4) **\crv{(-1,0)}, 
(-2.5,1.9)*{\ulcorner}, 
(2.5,1.9)*{\urcorner}, 
\endxy~\rangle_N
\end{equation} 
and $V_L(t)=\langle L\rangle_N(t^{-1/4})$ is the Jones polynomial of $L$ \cite{Kau}.

\bigskip

Now let $D$ be an oriented marked graph diagram. Since the normalized Kauffman bracket $\langle~~\rangle_N$ is an ambient isotopy invariant, we have $\alpha=1$. Hence it follows from Definitions \ref{defn-poly-ori-skr} and \ref{defn-poly-ori} that the polynomial $\ll D\gg$ associated with $[~~]=\langle~~\rangle_N$ is just equal to the polynomial $[[D]]=[[ D ]](A,A^{-1},x,y) \in \mathbb Z[A^{\pm 1}][x,y]=\mathbb Z[A^{\pm 1},x,y]$ defined by the two axioms:
\begin{itemize}
\item[{\rm ({\bf L1})}] $[[D]]= \langle D\rangle_N$ if $D$ is an oriented link diagram,
\item[{\rm ({\bf L2})}]  
$[[~\xy (-4,4);(4,-4) **@{-}, 
(4,4);(-4,-4) **@{-}, 
(3,3.2)*{\llcorner}, 
(-3,-3.4)*{\urcorner}, 
(-2.5,2)*{\ulcorner},
(2.5,-2.4)*{\lrcorner}, 
(3,-0.2);(-3,-0.2) **@{-},
(3,0);(-3,0) **@{-}, 
(3,0.2);(-3,0.2) **@{-}, 
\endxy~]] =
[[~\xy (-4,4);(4,4) **\crv{(0,-1)}, 
(4,-4);(-4,-4) **\crv{(0,1)}, 
(-2.5,1.9)*{\ulcorner}, (2.5,-2.4)*{\lrcorner}, 
\endxy~]]x + [[~\xy (-4,4);(-4,-4) **\crv{(1,0)},  
(4,4);(4,-4) **\crv{(-1,0)}, 
(-2.5,1.9)*{\ulcorner}, (2.5,-2.4)*{\lrcorner}, 
\endxy~]]y.$
\end{itemize}

In what follws we denote the polynomial $\ll D\gg$ associated with the normalized bracket $[~~]=\langle~~\rangle_N$ by $\ll D\gg_N$ for our convenience. Note that for any state $\sigma \in \mathfrak S(D)$, $w(D_\sigma)=w(D)$. From (\ref{state formula-ori}) and (\ref{nkbp}), we see that the state-sum formula for the polynomial $\ll D\gg_N=[[D]]$ is given by
\begin{align}\label{state formula-kb-ori}
\ll D \gg_N&=\sum_{\sigma \in \mathfrak S(D)}\langle D_\sigma \rangle_N x^{\sigma(\infty)}y^{\sigma(0)}=\sum_{\sigma \in \mathfrak S(D)}(-A^3)^{-w(D_\sigma)}\langle \widetilde{D_\sigma} \rangle x^{\sigma(\infty)}y^{\sigma(0)}\notag\\
&=(-A^3)^{-w(D)}\sum_{\sigma \in \mathfrak S(D)}\langle \widetilde{D_\sigma} \rangle x^{\sigma(\infty)}y^{\sigma(0)}.
\end{align}

The following theorem \ref{thm-skein-rel-kb} gives a method of computing the polynomial $\ll D\gg_N$ recursively for a given oriented marked graph diagram $D$.

\begin{theorem}\label{thm-skein-rel-kb}
Let $D$ be an oriented marked graph diagram. 
\begin{itemize}
\item[(1)] 
$\ll~\xy (0,0) *\xycircle(3,3){-}, (3,0) *{\wedge}, \endxy~\gg_N = 1.$
\item[(2)] 
If $D$ and $D'$ are two oriented marked graph diagrams related by a finite sequence of oriented Yoshikawa moves generated by the moves $\Gamma_1, \Gamma'_1, \Gamma_2, \Gamma_3, \Gamma_4, \Gamma'_4$ and $\Gamma_5$ of type I, then $\ll D\gg_N=\ll D'\gg_N.$
\item[(3)] 
$\ll~D \sqcup~ \xy (0,0) *\xycircle(3,3){-}, (3,0) *{\wedge}, \endxy~\gg_N = (-A^{-2}-A^2)\ll D  \gg_N.$
\item[(4)] 
$\ll ~\xy (-4,4);(4,-4) **@{-}, 
(4,4);(-4,-4) **@{-}, 
(3,3.2)*{\llcorner}, 
(-3,-3.4)*{\urcorner}, 
(-2.5,2)*{\ulcorner},
(2.5,-2.4)*{\lrcorner}, 
(3,-0.2);(-3,-0.2) **@{-},
(3,0);(-3,0) **@{-}, 
(3,0.2);(-3,0.2) **@{-}, 
\endxy~\gg_N =
\ll~\xy (-4,4);(4,4) **\crv{(0,-1)}, 
(4,-4);(-4,-4) **\crv{(0,1)},  
(-2.5,1.9)*{\ulcorner}, (2.5,-2.4)*{\lrcorner},   
\endxy~\gg_N x +  \ll~\xy (-4,4);(-4,-4) **\crv{(1,0)},  
(4,4);(4,-4) **\crv{(-1,0)}, (-2.5,1.9)*{\ulcorner}, (2.5,-2.4)*{\lrcorner}, 
\endxy~\gg_N y.$
\item[(5)] 
$A^{4}\ll~\xy 
(5,5);(-5,-5) **@{-} ?<*\dir{<},
(-5,5);(-2,2) **@{-} ?<*\dir{<}, 
(2,-2);(5,-5) **@{-},(0,-7)*{D_+},
\endxy~\gg_N - A^{-4} \ll~\xy 
(-5,5);(5,-5) **@{-} ?<*\dir{<},
(5,5);(2,2) **@{-} ?<*\dir{<}, 
(-2,-2);(-5,-5) **@{-},(0,-7)*{D_-},
\endxy~\gg_N = (A^{-2}-A^2)
\ll~\xy (-4,4);(-4,-4) **\crv{(1,0)},  
(4,4);(4,-4) **\crv{(-1,0)}, 
(-2.5,1.9)*{\ulcorner}, 
(2.5,1.9)*{\urcorner}, 
(0,-7)*{D_0},
\endxy~\gg_N,$
where $D_+, D_-$ and $D_0$ are three identical oriented link diagrams except the parts indicated. 
\end{itemize}
\end{theorem}

\begin{proof}
By {\rm ({\bf B1})}, {\rm ({\bf B2})}, Definition \ref{defn-poly-ori-skr} and Theorem \ref{thm-inv-mgs}, the assertions (1), (2) and (3) follow at once. Since $\ll D\gg_N=[[D]]$ by definition, the skein relation (4) is straightforward from the axiom 
(${\bf L2}$). Finally, the skein relation in (5) follows immediately from the skein relation in (\ref{A2-R1-move-1}) for the normalized Kauffman bracket for oriented link diagrams and the axiom (${\bf L1}$). This completes the proof.
\end{proof}

\begin{theorem}\label{thm-rels-polys-kbp}
Let $D$ be an oriented marked graph diagram $D$ and let $\widetilde{D}$ be the marked graph diagram $D$ without orientation. Then
\begin{equation*}
\ll D \gg_N = (-A^3)^{sw(D)-w(D)}\ll\widetilde{D}\gg.
\end{equation*}
\end{theorem}

\begin{proof}
Since $sw(D)=sw(\widetilde{D})$ and $\mathfrak S(D)=\mathfrak S(\widetilde{D})$, it follows from (\ref{state formula-kb-ori}) that 
\begin{align*}
\ll D \gg_N &=(-A^3)^{-w(D)}\sum_{\sigma \in \mathfrak S(D)}\langle \widetilde{D_\sigma} \rangle x^{\sigma(\infty)}y^{\sigma(0)}\\
&=(-A^3)^{-w(D)+sw(D)-sw(D)}\sum_{\sigma \in \mathfrak S(D)}\langle \widetilde{D_\sigma} \rangle x^{\sigma(\infty)}y^{\sigma(0)}\\
&=(-A^3)^{-w(D)+sw(D)}\Big((-A^3)^{-sw(\widetilde{D})}\sum_{\sigma \in \mathfrak S(\widetilde{D})}\langle \widetilde{D_\sigma} \rangle x^{\sigma(\infty)}y^{\sigma(0)}\Big)\\
&= (-A^3)^{sw(D)-w(D)}\ll \widetilde{D}\gg.
\end{align*}
This completes the proof.
\end{proof}

\begin{theorem}\label{thm-inv-nkbp}
(1) The normalzed Kauffman bracket ideal is equal to the Kauffman bracket ideal $I$ in Theorem \ref{thm-inv-kb} (1).

(2) For any oriented marked graph diagram $D$, the ideal coset $\ll D\gg_N^I =\ll D\gg_N +~ I$  is an invariant for the oriented surface-link in $\mathbb R^4$ or $S^4$ presented by $D$.
\end{theorem}

\begin{proof}
This is a direct consequence of Theorems \ref{thm-inv-kb} and \ref{thm-rels-polys-kbp}.
\end{proof}

\subsection{Normal forms of Kauffman bracket ideal cosets}\label{subsect-exmps}

In this subsection, we show how to calculate the normal forms of the Kauffman bracket ideal cosets and the normalized Kauffman bracket ideal cosets implimented by the Groebner basis calculations on Maple 17 and compute the normal forms of surface-links in Yoshikawa's table \cite{Yo} with ch-index $\leq 7$. For our purpose, we substitute $B$ for $A^{-1}$ and consider the ideal $J$ of the polynomial ring $\mathbb Q[A, B, x, y]$ generated by the following four polynomials:
\begin{equation*}
\begin{split}
&p_1(A,B,x,y) = (-A^2-B^2)x+y-1,\\
&p_2(A,B,x,y) = x+(-A^2-B^2)y-1,\\
&p_3(A,B,x,y)=(A^4+1+B^4)xy,\\
&p_4(A,B,x,y)=AB-1.
\end{split}
\end{equation*}
Using Groebner basis calculations on Maple 17, we see that 
\begin{equation}\label{eq-gb-kb}
G=\{x-1+y, AB-1, A^2+B^2+1, B^3+A+B\}
\end{equation} 
is a Groebner basis for the ideal $J=<p_1, p_2, p_3, p_4>$ with respect to the graded reverse lexicographic order "tdeg" (also called "grevlex" in the literature) in $A, B, x, y$. As a consequence of Theorems \ref{resid modulo-3} and \ref{thm-inv-nkbp}, we obtain the following theorem.

\begin{theorem}\label{resid modulo-kbp-nf}
Let $D$ be an (resp.~oriented) marked graph diagram and let $\mathcal L$ be the (resp.~oriented) surface-link presented by $D$. Let $\ll D\gg$ be the polynomial of $D$ in (\ref{defn-poly-inv-kb}) (resp.~(\ref{state formula-kb-ori})). Then the normal form $\overline{\ll D\gg}$ (resp.~$\overline{\ll D\gg}_N$) on division of $\ll D\gg$ (resp.~$\ll D\gg_N$) by the Groebner basis $G$ in (\ref{eq-gb-kb}) is an invariant of $\mathcal L$.
\end{theorem} 

We denote the normal forms $\overline{\ll D\gg}$ and $\overline{\ll D\gg}_N$ in Theorem \ref{resid modulo-kbp-nf} by $\overline{\ll \mathcal L\gg}$ and $\overline{\ll \mathcal L\gg}_N$, respectively. 

In the rest of this subsection, we give various examples which will be used in Section~\ref{sect-inv-kbp-spl} again. First, we recall the normalized Kauffman bracket polynomials of the following oriented knots and links in $3$-space which will be used in our discussion of examples:
\begin{equation}\label{nkbp-exmps}
\begin{split}
\langle2^2_1\rangle_N=&\langle~\xy
(12,2);(10.4,0.4) **@{-}, 
(9.6,-0.4);(8,-2) **@{-},
(12,-2);(8,2) **@{-},
(5.6,-0.4);(4,-2) **@{-}, 
(8,2);(6.6,0.4) **@{-},
(8,-2);(4,2) **@{-},
(4,2);(12,2) **\crv{(8,6)},
(4,-2);(12,-2) **\crv{(8,-6)},
(11,0.5)*{\urcorner}, 
(8.9,0.5)*{\ulcorner}, 
\endxy~\rangle_N
=-A^{10}-A^{2},\\
\langle2^{2*}_1\rangle_N=&\langle~\xy
(12,2);(10.4,0.4) **@{-}, 
(9.6,-0.4);(8,-2) **@{-},
(12,-2);(8,2) **@{-},
(5.6,-0.4);(4,-2) **@{-}, 
(8,2);(6.6,0.4) **@{-},
(8,-2);(4,2) **@{-},
(4,2);(12,2) **\crv{(8,6)},
(4,-2);(12,-2) **\crv{(8,-6)}, 
(11,0.5)*{\urcorner}, 
(11,-0.8)*{\lrcorner},  
\endxy~\rangle_N
=-A^{-10}-A^{-2},\\
\langle3_1\rangle_N=&\langle~\xy
(12,2);(10.4,0.4) **@{-}, 
(9.6,-0.4);(8,-2) **@{-},
(12,-2);(8,2) **@{-},
(5.6,-0.4);(4,-2) **@{-}, 
(8,2);(6.6,0.4) **@{-},
(8,-2);(4,2) **@{-}, 
(11,0.5)*{\urcorner}, 
(11,-0.8)*{\lrcorner},  
(4,2);(2.4,0.4) **@{-}, 
(1.6,-0.4);(0,-2) **@{-},
(4,-2);(0,2) **@{-},
(0,2);(12,2) **\crv{(5,6)},
(0,-2);(12,-2) **\crv{(5,-6)},
\endxy~\rangle_N
=-A^{-16}+A^{-12}+A^{-4},\\
\langle4^2_1\rangle_N=&\langle~\xy
(12,2);(10.4,0.4) **@{-}, 
(9.6,-0.4);(8,-2) **@{-},
(12,-2);(8,2) **@{-},
(5.6,-0.4);(4,-2) **@{-}, 
(8,2);(6.6,0.4) **@{-},
(8,-2);(4,2) **@{-},
(11,0.5)*{\urcorner}, 
(8.9,0.5)*{\ulcorner}, 
(4,2);(2.4,0.4) **@{-}, 
(1.6,-0.4);(0,-2) **@{-},
(4,-2);(0,2) **@{-},
(-3.6,-0.4);(-4,-2) **@{-}, 
(0,2);(-1.6,0.4) **@{-},
(0,-2);(-4,2) **@{-},
(-2.2,-0.4);(-4,-2) **@{-},
(-4,2);(12,2) **\crv{(5,6)},
(-4,-2);(12,-2) **\crv{(5,-6)},
\endxy~\rangle_N
=-A^{18}-A^{10}+A^{6}-A^{2},\\
\langle4_1\rangle_N=&\langle~\xy
(12,2);(10.4,0.4) **@{-}, 
(9.6,-0.4);(8,-2) **@{-},
(12,-2);(8,2) **@{-},
(5.6,-0.4);(4,-2) **@{-}, 
(8,2);(6.6,0.4) **@{-},
(8,-2);(4,2) **@{-},
(11,0.5)*{\urcorner}, (8.9,0.5)*{\ulcorner}, 
(4,-2);(-2,-2) **\crv{(1,-5)},
(4,4);(2,-2) **\crv{(0,1)},
(-2,2);(1.5,4) **\crv{(-1,4.5)},
(-2,-2);(-2,2) **\crv{(-3,0)},
(4,-4);(12,-2) **\crv{(10,-6)},
(4,4);(12,2) **\crv{(10,6)},
\endxy~\rangle_N=A^{-8}-A^{-4}+1-A^4+A^8,\\
\langle2^2_1\sharp 3_1\rangle_N=&\langle~\xy 
(2,0);(4,-2) **@{-}, 
(-2,0);(-4,-2) **@{-},  
(2.4,0.4);(4,2) **@{-},
(0.1,-1.9);(1.7,-0.2) **@{-},  
(0.1,-1.9);(-1.7,-0.3) **@{-},  
(-2.4,0.4);(-4,2) **@{-},
(-8,2);(-6.4,0.4) **@{-}, 
(-5.6,-0.4);(-4,-2) **@{-}, 
(-8,-2);(-4,2) **@{-},
(12,2);(10.4,0.4) **@{-}, 
(9.6,-0.4);(8,-2) **@{-},
(12,-2);(8,2) **@{-},
(5.6,-0.4);(4,-2) **@{-}, 
(8,2);(6.6,0.4) **@{-},
(8,-2);(4,2) **@{-},
(-5,0.4)*{\urcorner}, 
(-5,-0.8)*{\lrcorner}, 
(11,0.5)*{\urcorner}, 
(11,-0.8)*{\lrcorner},  
(-8,2);(-2,4) **\crv{(-10,5)}, 
(2,4);(12,2) **\crv{(14,5)},
(-2,0);(-2,4) **\crv{(1,2)}, 
(2,4);(2,0) **\crv{(-1,2)},
(-8,-2);(0,-5) **\crv{(-10,-5)},
(0,-5);(12,-2) **\crv{(14,-5)},
\endxy~\rangle_N
=(-A^{10}-A^{2})(-A^{-16}+A^{-12}+A^{-4}),\\
\langle3^*_1\sharp 3_1\rangle_N=& \langle~\xy 
(2,0);(4,-2) **@{-}, 
(-2,0);(-4,-2) **@{-},  
(2.4,0.4);(4,2) **@{-},
(0.1,-1.9);(1.7,-0.2) **@{-},  
(0.1,-1.9);(-1.7,-0.3) **@{-},  
(-2.4,0.4);(-4,2) **@{-},
(-12,2);(-10.4,0.4) **@{-}, 
(-9.6,-0.4);(-8,-2) **@{-}, 
(-12,-2);(-8,2) **@{-},
(-8,2);(-6.4,0.4) **@{-}, 
(-5.6,-0.4);(-4,-2) **@{-}, 
(-8,-2);(-4,2) **@{-},
(12,2);(10.4,0.4) **@{-}, 
(9.6,-0.4);(8,-2) **@{-},
(12,-2);(8,2) **@{-},
(5.6,-0.4);(4,-2) **@{-}, 
(8,2);(6.6,0.4) **@{-},
(8,-2);(4,2) **@{-},
(-9,0.4)*{\urcorner}, 
(-9,-0.8)*{\lrcorner}, 
(11,0.5)*{\urcorner}, 
(11,-0.8)*{\lrcorner},  
(-12,2);(-2,4) **\crv{(-14,5)}, 
(2,4);(12,2) **\crv{(14,5)},
(-2,0);(-2,4) **\crv{(1,2)}, 
(2,4);(2,0) **\crv{(-1,2)},
(-12,-2);(0,-5) **\crv{(-14,-5)},
(0,-5);(12,-2) **\crv{(14,-5)},
 \endxy~\rangle_N
=(A^{16}-A^{12}-A^{4})(A^{-16}-A^{-12}-A^{-4}).
\end{split}
\end{equation}
By (5) in Theorem \ref{thm-skein-rel-kb}, we obtain the following useful identities:
\begin{equation}\label{skein-rel-exmp}
\begin{split}
&\ll~\xy 
(5,5);(-5,-5) **@{-} ?<*\dir{<},
(-5,5);(-2,2) **@{-} ?<*\dir{<}, 
(2,-2);(5,-5) **@{-}, 
\endxy~\gg = A^{-8} \ll~\xy 
(-5,5);(5,-5) **@{-} ?<*\dir{<},
(5,5);(2,2) **@{-} ?<*\dir{<}, 
(-2,-2);(-5,-5) **@{-}, 
\endxy~\gg + (A^{-6}-A^{-2})
\ll~\xy (-4,4);(-4,-4) **\crv{(1,0)},  
(4,4);(4,-4) **\crv{(-1,0)}, 
(-2.5,1.9)*{\ulcorner}, (2.5,1.9)*{\urcorner}, 
\endxy~\gg,
\\
&\ll~\xy 
(-5,5);(5,-5) **@{-} ?<*\dir{<},
(5,5);(2,2) **@{-} ?<*\dir{<}, 
(-2,-2);(-5,-5) **@{-}, 
\endxy~\gg =A^{8}\ll~\xy 
(5,5);(-5,-5) **@{-} ?<*\dir{<},
(-5,5);(-2,2) **@{-} ?<*\dir{<}, 
(2,-2);(5,-5) **@{-}, 
\endxy~\gg + (A^{6}-A^{2})
\ll~\xy (-4,4);(-4,-4) **\crv{(1,0)},  
(4,4);(4,-4) **\crv{(-1,0)}, 
(-2.5,1.9)*{\ulcorner}, (2.5,1.9)*{\urcorner}, 
\endxy~\gg.
\end{split}
\end{equation}

\begin{example}\label{examp-T^2} 
Let $2^1_1$ be the unoriented standard torus in Yoshikawa's table \cite{Yo}. Then $sw(2^1_1)=0$ and
\begin{eqnarray*}
\doubleBracket{
(0,0.5);(1,-0.5) **@{-},(0,-0.5);(1,0.5) **@{-},
(0,-0.5);(1,-1.5) **@{-},(0,-1.5);(1,-0.5) **@{-},
(0.5,0.25);(0.5,-0.25) **@{-},
(0.25,-1);(0.75,-1) **@{-},
(0,0.5);(0,-1.5) **\crv{(-0.5,0.4)&(-0.5,-1.4)},
(1,0.5);(1,-1.5)**\crv{(1.5,0.4)&(1.5,-1.4)}, (0.5,-2.2)*{2^1_1}}
&=&\newBracket{(0,0.5);(0,-0.5) **\crv{(0.5,0.4)&(0.5,-0.4)},
(1,0.5);(1,-0.5)**\crv{(0.5,0.4)&(0.5,-0.4)},
(0,-0.5);(1,-0.5) **\crv{(0.1,-1)&(0.9,-1)},
(0,-1.5);(1,-1.5)**\crv{(0.1,-1)&(0.9,-1)},(0,0.5);
(0,-1.5) **\crv{(-0.5,0.4)&(-0.5,-1.4)},
(1,0.5);(1,-1.5) **\crv{(1.5,0.4)&(1.5,-1.4)}}x^2+
 \newBracket{(0,0.5);(0,-0.5) **\crv{(0.5,0.4)&(0.5,-0.4)},
(1,0.5);(1,-0.5)**\crv{(0.5,0.4)&(0.5,-0.4)},
(0,-0.5);(0,-1.5) **\crv{(0.5,-0.4)&(0.5,-1.4)},
(1,-0.5);(1,-1.5)**\crv{(0.5,-0.4)&(0.5,-1.4)},(0,0.5);
(0,-1.5) **\crv{(-0.5,0.4)&(-0.5,-1.4)},
(1,0.5);(1,-1.5) **\crv{(1.5,0.4)&(1.5,-1.4)}}xy+
\newBracket{(0,0.5);(1,0.5) **\crv{(0.1,0)&(0.9,0)},
(0,-0.5);(1,-0.5)**\crv{(0.1,0)&(0.9,0)},
(0,-0.5);(1,-0.5) **\crv{(0.1,-1)&(0.9,-1)},
(0,-1.5);(1,-1.5)**\crv{(0.1,-1)&(0.9,-1)},(0,0.5);
(0,-1.5) **\crv{(-0.5,0.4)&(-0.5,-1.4)},
(1,0.5);(1,-1.5)**\crv{(1.5,0.4)&(1.5,-1.4)}}yx+
\newBracket{(0,0.5);(1,0.5) **\crv{(0.1,0)&(0.9,0)},
(0,-0.5);(1,-0.5)**\crv{(0.1,0)&(0.9,0)},
(0,-0.5);(0,-1.5) **\crv{(0.5,-0.4)&(0.5,-1.4)},
(1,-0.5);(1,-1.5)**\crv{(0.5,-0.4)&(0.5,-1.4)},
(0,0.5);(0,-1.5)**\crv{(-0.5,0.4)&(-0.5,-1.4)},
(1,0.5);(1,-1.5) **\crv{(1.5,0.4)&(1.5,-1.4)}}y^2\\
&=&x^2+2(-A^2-A^{-2})xy+y^2.
\end{eqnarray*}
Substituting $B$ for $A^{-1}$, we get $\ll 2^{1}_1\gg=(-A^3)^{-sw(2^1_1)}[[2^1_1]]=x^2+2(-A^2-B^{2})xy+y^2$. This yields the normal form $\overline{\ll 2^{1}_1\gg}=1.$
Further, for any orientation on $2^1_1$, $w(2^1_1)=sw(2^1_1)=0$ and hence it follows from Theorem \ref{thm-rels-polys-kbp} that
$\overline{\ll 2^{1}_1\gg}_N=\overline{\ll 2^{1}_1\gg}=1.$
\end{example}

\begin{example}\label{exmp-RP2}
Let $2_1^{-1}$ be the positive standard projective plane in Yoshikawa's table \cite{Yo}. Then $sw(2_1^{-1})=0$ and
$$\doubleBracket{(0,-0.5);(0.4,-0.1) **@{-}, 
 (0.6,0.1);(1,0.5) **@{-}, 
(0,0.5);(1,-0.5) **@{-},
(0,-0.5);(1,-1.5) **@{-},
(0,-1.5);(1,-0.5) **@{-},
(0.25,-1);(0.75,-1) **@{-},
(0,0.5);(0,-1.5) **\crv{(-0.5,0.4)&(-0.5,-1.4)},
(1,0.5);(1,-1.5) **\crv{(1.5,0.4)&(1.5,-1.4)},(0.5,-2.2)*{2_1^{-1}}} 
= \newBracket{(0,-0.5);(0.4,-0.1) **@{-}, 
(0.6,0.1);(1,0.5) **@{-}, 
(0,0.5);(1,-0.5) **@{-},
(0,-0.5);(1,-0.5) **\crv{(0.1,-1)&(0.9,-1)},
(0,-1.5);(1,-1.5)**\crv{(0.1,-1)&(0.9,-1)},
(0,0.5);(0,-1.5) **\crv{(-0.5,0.4)&(-0.5,-1.4)},
(1,0.5);(1,-1.5) **\crv{(1.5,0.4)&(1.5,-1.4)}}x +
\newBracket{(0,-0.5);(0.4,-0.1) **@{-}, 
(0.6,0.1);(1,0.5) **@{-}, 
(0,0.5);(1,-0.5) **@{-},
(0,-0.5);(0,-1.5) **\crv{(0.5,-0.4)&(0.5,-1.4)},
(1,-0.5);(1,-1.5)**\crv{(0.5,-0.4)&(0.5,-1.4)},
(0,0.5);(0,-1.5) **\crv{(-0.5,0.4)&(-0.5,-1.4)},
(1,0.5);(1,-1.5) **\crv{(1.5,0.4)&(1.5,-1.4)}}y
=(-A^3)x+(-A^{-3})y.$$
Substituting $B$ for $A^{-1}$, we have $\ll 2^{-1}_1\gg=(-A^3)^{-sw(2_1^{-1})}[[2_1^{-1}]]=-A^3x-B^3y$. This gives the normal form $\overline{\ll 2^{-1}_1\gg}=A+B.$
Let $2_1^{-1*}$ be the negative standard projective plane. Then we also have $\overline{\ll 2_1^{-1*}\gg}=A+B.$
\end{example}

\begin{example}\label{examp-2comp-sl-6^{0,1}_1}
Consider the 2 component surface-link $6^{0,1}_1$ in Yoshikawa's table \cite{Yo} with the orientation indicated below. By Theorem \ref{thm-skein-rel-kb} together with (\ref{skein-rel-exmp}) and (\ref{nkbp-exmps}), we obtain
\begin{align*}
\ll~&\xy 
(-8,8);(-2,8) **@{-}, %
(2,8);(8,8) **@{-}, %
(-8,8);(-8,2) **@{-}, %
(8,8);(8,2) **@{-}, %
(-8,-8);(-2,-8) **@{-}, %
(-2,-8);(8,-8) **@{-}, %
(-8,-8);(-8,-2) **@{-}, %
(8,-8);(8,-2) **@{-},%
(-2,8);(2,4) **@{-}, 
(-2,4);(2,8) **@{-}, 
(-2,4);(4,-2) **@{-}, 
(2,4);(-4,-2) **@{-}, 
%
%
(2.4,0.4);(4,2) **@{-},
(0.1,-1.9);(1.7,-0.2) **@{-},  
(0.1,-1.9);(-1.7,-0.3) **@{-},  
(-2.4,0.4);(-4,2) **@{-},
(-8,2);(-6.4,0.4) **@{-}, 
(-5.6,-0.4);(-4,-2) **@{-}, 
(-8,-2);(-4,2) **@{-},
(5.6,-0.4);(4,-2) **@{-}, 
(8,2);(6.6,0.4) **@{-},
(8,-2);(4,2) **@{-},
(-0.1,4.5);(-0.1,7.5) **@{-},
(0,4.5);(0,7.5) **@{-},
(0.1,4.5);(0.1,7.5) **@{-},
(-1.5,1.9);(1.5,1.9) **@{-},
(-1.5,2);(1.5,2) **@{-},
(-1.5,2.1);(1.5,2.1) **@{-}, 
(-5,0.4)*{\urcorner}, 
(-5,-0.8)*{\lrcorner}, 
(7,0.5)*{\urcorner}, 
(7,-0.8)*{\lrcorner}, 
(0.5,-12.2) *{6^{0,1}_1},
\endxy\gg_N =
x^2 \ll~\xy 
(-8,8);(-2,8) **@{-}, %
(2,8);(8,8) **@{-}, %
(-8,8);(-8,2) **@{-}, %
(8,8);(8,2) **@{-}, %
(-8,-8);(-2,-8) **@{-}, %
(-2,-8);(8,-8) **@{-}, %
(-8,-8);(-8,-2) **@{-}, %
(8,-8);(8,-2) **@{-},%
(2,0);(4,-2) **@{-}, 
(-2,0);(-4,-2) **@{-}, 
%
%
(2.4,0.4);(4,2) **@{-},
(0.1,-1.9);(1.7,-0.2) **@{-},  
(0.1,-1.9);(-1.7,-0.3) **@{-},  
(-2.4,0.4);(-4,2) **@{-},
(-8,2);(-6.4,0.4) **@{-}, 
(-5.6,-0.4);(-4,-2) **@{-}, 
(-8,-2);(-4,2) **@{-},
(5.6,-0.4);(4,-2) **@{-}, 
(8,2);(6.6,0.4) **@{-},
(8,-2);(4,2) **@{-},
(-5,0.4)*{\urcorner}, 
(-5,-0.8)*{\lrcorner}, 
(7,0.5)*{\urcorner}, 
(7,-0.8)*{\lrcorner}, 
(-2,4);(-2,8) **\crv{(1,6)}, 
(2,8);(2,4) **\crv{(-1,6)},
(-2,4);(2,4) **\crv{(0,1)}, 
(-2,0);(2,0) **\crv{(0,3)},
 \endxy\gg_N 
 +xy \ll~\xy 
(-8,8);(-2,8) **@{-}, %
(2,8);(8,8) **@{-}, %
(-8,8);(-8,2) **@{-}, %
(8,8);(8,2) **@{-}, %
(-8,-8);(-2,-8) **@{-}, %
(-2,-8);(8,-8) **@{-}, %
(-8,-8);(-8,-2) **@{-}, %
(8,-8);(8,-2) **@{-},%
(2,0);(4,-2) **@{-}, 
(-2,0);(-4,-2) **@{-},  
%
%
(2.4,0.4);(4,2) **@{-},
(0.1,-1.9);(1.7,-0.2) **@{-},  
(0.1,-1.9);(-1.7,-0.3) **@{-},  
(-2.4,0.4);(-4,2) **@{-},
(-8,2);(-6.4,0.4) **@{-}, 
(-5.6,-0.4);(-4,-2) **@{-}, 
(-8,-2);(-4,2) **@{-},
(5.6,-0.4);(4,-2) **@{-}, 
(8,2);(6.6,0.4) **@{-},
(8,-2);(4,2) **@{-},
(-5,0.4)*{\urcorner}, 
(-5,-0.8)*{\lrcorner}, 
(7,0.5)*{\urcorner}, 
(7,-0.8)*{\lrcorner}, 
(-2,4);(-2,8) **\crv{(1,6)}, 
(2,8);(2,4) **\crv{(-1,6)},
(-2,0);(-2,4) **\crv{(1,2)}, 
(2,4);(2,0) **\crv{(-1,2)},
 \endxy\gg_N +\\
 &\hskip 2.8cm yx
 \ll~\xy 
(-8,8);(-2,8) **@{-}, %
(2,8);(8,8) **@{-}, %
(-8,8);(-8,2) **@{-}, %
(8,8);(8,2) **@{-}, %
(-8,-8);(-2,-8) **@{-}, %
(-2,-8);(8,-8) **@{-}, %
(-8,-8);(-8,-2) **@{-}, %
(8,-8);(8,-2) **@{-},%
(2,0);(4,-2) **@{-}, 
(-2,0);(-4,-2) **@{-}, 
%
%
(2.4,0.4);(4,2) **@{-},
(0.1,-1.9);(1.7,-0.2) **@{-},  
(0.1,-1.9);(-1.7,-0.3) **@{-},  
(-2.4,0.4);(-4,2) **@{-},
(-8,2);(-6.4,0.4) **@{-}, 
(-5.6,-0.4);(-4,-2) **@{-}, 
(-8,-2);(-4,2) **@{-},
(5.6,-0.4);(4,-2) **@{-}, 
(8,2);(6.6,0.4) **@{-},
(8,-2);(4,2) **@{-},
(-5,0.4)*{\urcorner}, 
(-5,-0.8)*{\lrcorner}, 
(7,0.5)*{\urcorner}, 
(7,-0.8)*{\lrcorner}, 
(-2,8);(2,8) **\crv{(0,5)}, 
(-2,4);(2,4) **\crv{(0,7)},
(-2,4);(2,4) **\crv{(0,1)}, 
(-2,0);(2,0) **\crv{(0,3)},
 \endxy\gg_N + y^2
 \ll~\xy 
(-8,8);(-2,8) **@{-}, %
(2,8);(8,8) **@{-}, %
(-8,8);(-8,2) **@{-}, %
(8,8);(8,2) **@{-}, %
(-8,-8);(-2,-8) **@{-}, %
(-2,-8);(8,-8) **@{-}, %
(-8,-8);(-8,-2) **@{-}, %
(8,-8);(8,-2) **@{-},%
(2,0);(4,-2) **@{-}, 
(-2,0);(-4,-2) **@{-},  
%
%
(2.4,0.4);(4,2) **@{-},
(0.1,-1.9);(1.7,-0.2) **@{-},  
(0.1,-1.9);(-1.7,-0.3) **@{-},  
(-2.4,0.4);(-4,2) **@{-},
(-8,2);(-6.4,0.4) **@{-}, 
(-5.6,-0.4);(-4,-2) **@{-}, 
(-8,-2);(-4,2) **@{-},
(5.6,-0.4);(4,-2) **@{-}, 
(8,2);(6.6,0.4) **@{-},
(8,-2);(4,2) **@{-},
(-5,0.4)*{\urcorner}, 
(-5,-0.8)*{\lrcorner}, 
(7,0.5)*{\urcorner}, 
(7,-0.8)*{\lrcorner}, 
(-2,0);(-2,4) **\crv{(1,2)}, 
(2,4);(2,0) **\crv{(-1,2)},
(-2,8);(2,8) **\crv{(0,5)}, 
(-2,4);(2,4) **\crv{(0,7)},
 \endxy\gg_N\\
 &= \Big\langle
 ~\xy (0,0) *\xycircle(3,3){-},\endxy
 ~\xy (0,0) *\xycircle(3,3){-},\endxy
 ~\Big\rangle_N (x^2+y^2) + xy\Big\langle~\xy 
(-8,8);(-2,8) **@{-}, %
(2,8);(8,8) **@{-}, %
(-8,8);(-8,2) **@{-}, %
(8,8);(8,2) **@{-}, %
(-8,-8);(-2,-8) **@{-}, %
(-2,-8);(8,-8) **@{-}, %
(-8,-8);(-8,-2) **@{-}, %
(8,-8);(8,-2) **@{-},%
(2,0);(4,-2) **@{-}, 
(-2,0);(-4,-2) **@{-},  
%
%
(2.4,0.4);(4,2) **@{-},
(0.1,-1.9);(1.7,-0.2) **@{-},  
(0.1,-1.9);(-1.7,-0.3) **@{-},  
(-2.4,0.4);(-4,2) **@{-},
(-8,2);(-6.4,0.4) **@{-}, 
(-5.6,-0.4);(-4,-2) **@{-}, 
(-8,-2);(-4,2) **@{-},
(5.6,-0.4);(4,-2) **@{-}, 
(8,2);(6.6,0.4) **@{-},
(8,-2);(4,2) **@{-},
(-5,0.4)*{\urcorner}, 
(-5,-0.8)*{\lrcorner}, 
(7,0.5)*{\urcorner}, 
(7,-0.8)*{\lrcorner},
(-2,4);(-2,8) **\crv{(1,6)}, 
(2,8);(2,4) **\crv{(-1,6)},
(-2,0);(-2,4) **\crv{(1,2)}, 
(2,4);(2,0) **\crv{(-1,2)},
 \endxy~\Big\rangle_N + yx\Big\langle~\xy
(0,0) *\xycircle(3,3){-}, \endxy~
\xy (0,0) *\xycircle(3,3){-}, \endxy~
\xy (0,0) *\xycircle(3,3){-}, \endxy~\Big\rangle_N\\
&=(-A^{-2}-A^{2})(x^2+y^2)+(-A^{10}-A^2)(-A^{-10}-A^{-2})xy+ (-A^{-2}-A^{2})^2xy\\
&=(-A^{2}-A^{-2})(x^2+y^2) + (A^4+4+A^{-4}+A^{-8}+A^8)xy.
 \end{align*}
 Substituting $B$ for $A^{-1}$, we get $$\ll 6^{0,1}_1\gg_N=(-A^{2}-B^{2})(x^2+y^2) + (A^4+4+B^{4}+B^{8}+A^8) xy.$$ This yields the normal form $\overline{\ll 6^{0,1}_1\gg}_N=1.$
Further, forgetting the orientation on $6^{0,1}_1$, we see that $sw(\widetilde{6^{0,1}_1})=w(6^{0,1}_1)=0$ and hence it follows from Theorem \ref{thm-rels-polys-kbp} that $\overline{\ll \widetilde{6^{0,1}_1}\gg}=\overline{\ll 6^{0,1}_1\gg}_N=1.$
\end{example}
 
 
\begin{example}\label{examp-2comp-sl-7^{0,-2}_1}
Consider the nonorientable two component surface-link $7^{0,-2}_1$ in Yoshikawa's table \cite{Yo}. Then $sw(7^{0,-2}_1)=1$ and  
\begin{align*}
\ll~&\xy 
(-2,12);(2,8) **@{-}, 
(-2,8);(-0.7,9.5) **@{-},
(0.7,10.6);(2,12) **@{-},
(-8,12);(-2,12) **@{-}, %
(2,12);(8,12) **@{-}, %
(-8,12);(-8,2) **@{-}, %
(8,12);(8,2) **@{-},%
(-8,-8);(-2,-8) **@{-}, %
(-2,-8);(8,-8) **@{-}, %
(-8,-8);(-8,-2) **@{-}, %
(8,-8);(8,-2) **@{-},%
(-2,8);(2,4) **@{-}, 
(-2,4);(2,8) **@{-}, 
(-2,4);(4,-2) **@{-}, 
(2,4);(-4,-2) **@{-}, 
(2.4,0.4);(4,2) **@{-},
(0.1,-1.9);(1.7,-0.2) **@{-},  
(0.1,-1.9);(-1.7,-0.3) **@{-},  
(-2.4,0.4);(-4,2) **@{-},
(-8,2);(-6.4,0.4) **@{-}, 
(-5.6,-0.4);(-4,-2) **@{-}, 
(-8,-2);(-4,2) **@{-},
(5.6,-0.4);(4,-2) **@{-}, 
(8,2);(6.6,0.4) **@{-},
(8,-2);(4,2) **@{-},
(-0.1,4.5);(-0.1,7.5) **@{-},
(0,4.5);(0,7.5) **@{-},
(0.1,4.5);(0.1,7.5) **@{-},
(-1.5,1.9);(1.5,1.9) **@{-},
(-1.5,2);(1.5,2) **@{-},
(-1.5,2.1);(1.5,2.1) **@{-}, 
(0.5,-12.2) *{7^{0,-2}_1},
\endxy\gg 
= (-A^3)^{-1}\bigg(x^2
[[~\xy 
(-2,12);(2,8) **@{-}, 
(-2,8);(-0.7,9.5) **@{-},
(0.7,10.6);(2,12) **@{-},
(-8,12);(-2,12) **@{-}, %
(2,12);(8,12) **@{-}, %
(-8,12);(-8,2) **@{-}, %
(8,12);(8,2) **@{-},%
(-8,-8);(-2,-8) **@{-}, %
(-2,-8);(8,-8) **@{-}, %
(-8,-8);(-8,-2) **@{-}, %
(8,-8);(8,-2) **@{-},%
(2,0);(4,-2) **@{-}, 
(-2,0);(-4,-2) **@{-}, 
(2.4,0.4);(4,2) **@{-},
(0.1,-1.9);(1.7,-0.2) **@{-},  
(0.1,-1.9);(-1.7,-0.3) **@{-},  
(-2.4,0.4);(-4,2) **@{-},
(-8,2);(-6.4,0.4) **@{-}, 
(-5.6,-0.4);(-4,-2) **@{-}, 
(-8,-2);(-4,2) **@{-},
(5.6,-0.4);(4,-2) **@{-}, 
(8,2);(6.6,0.4) **@{-},
(8,-2);(4,2) **@{-}, 
(-2,4);(-2,8) **\crv{(1,6)}, 
(2,8);(2,4) **\crv{(-1,6)},
(-2,4);(2,4) **\crv{(0,1)}, 
(-2,0);(2,0) **\crv{(0,3)},
 \endxy~]] 
 +xy
 [[~\xy 
(-2,12);(2,8) **@{-}, 
(-2,8);(-0.7,9.5) **@{-},
(0.7,10.6);(2,12) **@{-},
(-8,12);(-2,12) **@{-}, %
(2,12);(8,12) **@{-}, %
(-8,12);(-8,2) **@{-}, %
(8,12);(8,2) **@{-},%
(-8,-8);(-2,-8) **@{-}, %
(-2,-8);(8,-8) **@{-}, %
(-8,-8);(-8,-2) **@{-}, %
(8,-8);(8,-2) **@{-},%
(2,0);(4,-2) **@{-}, 
(-2,0);(-4,-2) **@{-},  
(2.4,0.4);(4,2) **@{-},
(0.1,-1.9);(1.7,-0.2) **@{-},  
(0.1,-1.9);(-1.7,-0.3) **@{-},  
(-2.4,0.4);(-4,2) **@{-},
(-8,2);(-6.4,0.4) **@{-}, 
(-5.6,-0.4);(-4,-2) **@{-}, 
(-8,-2);(-4,2) **@{-},
(5.6,-0.4);(4,-2) **@{-}, 
(8,2);(6.6,0.4) **@{-},
(8,-2);(4,2) **@{-},
(-2,4);(-2,8) **\crv{(1,6)}, 
(2,8);(2,4) **\crv{(-1,6)},
(-2,0);(-2,4) **\crv{(1,2)}, 
(2,4);(2,0) **\crv{(-1,2)},
 \endxy~]] \\
 &\hskip 3.7cm + yx
 [[~\xy 
(-2,12);(2,8) **@{-}, 
(-2,8);(-0.7,9.5) **@{-},
(0.7,10.6);(2,12) **@{-},
(-8,12);(-2,12) **@{-}, %
(2,12);(8,12) **@{-}, %
(-8,12);(-8,2) **@{-}, %
(8,12);(8,2) **@{-},%
(-8,-8);(-2,-8) **@{-}, %
(-2,-8);(8,-8) **@{-}, %
(-8,-8);(-8,-2) **@{-}, %
(8,-8);(8,-2) **@{-},%
(2,0);(4,-2) **@{-}, 
(-2,0);(-4,-2) **@{-}, 
(2.4,0.4);(4,2) **@{-},
(0.1,-1.9);(1.7,-0.2) **@{-},  
(0.1,-1.9);(-1.7,-0.3) **@{-},  
(-2.4,0.4);(-4,2) **@{-},
(-8,2);(-6.4,0.4) **@{-}, 
(-5.6,-0.4);(-4,-2) **@{-}, 
(-8,-2);(-4,2) **@{-},
(5.6,-0.4);(4,-2) **@{-}, 
(8,2);(6.6,0.4) **@{-},
(8,-2);(4,2) **@{-}, 
(-2,8);(2,8) **\crv{(0,5)}, 
(-2,4);(2,4) **\crv{(0,7)},
(-2,4);(2,4) **\crv{(0,1)}, 
(-2,0);(2,0) **\crv{(0,3)},
 \endxy~]] + y^2
 [[~\xy 
(-2,12);(2,8) **@{-}, 
(-2,8);(-0.7,9.5) **@{-},
(0.7,10.6);(2,12) **@{-},
(-8,12);(-2,12) **@{-}, %
(2,12);(8,12) **@{-}, %
(-8,12);(-8,2) **@{-}, %
(8,12);(8,2) **@{-},%
(-8,-8);(-2,-8) **@{-}, %
(-2,-8);(8,-8) **@{-}, %
(-8,-8);(-8,-2) **@{-}, %
(8,-8);(8,-2) **@{-},%
(2,0);(4,-2) **@{-}, 
(-2,0);(-4,-2) **@{-},  
(2.4,0.4);(4,2) **@{-},
(0.1,-1.9);(1.7,-0.2) **@{-},  
(0.1,-1.9);(-1.7,-0.3) **@{-},  
(-2.4,0.4);(-4,2) **@{-},
(-8,2);(-6.4,0.4) **@{-}, 
(-5.6,-0.4);(-4,-2) **@{-}, 
(-8,-2);(-4,2) **@{-},
(5.6,-0.4);(4,-2) **@{-}, 
(8,2);(6.6,0.4) **@{-},
(8,-2);(4,2) **@{-},
(-2,0);(-2,4) **\crv{(1,2)}, 
(2,4);(2,0) **\crv{(-1,2)},
(-2,8);(2,8) **\crv{(0,5)}, 
(-2,4);(2,4) **\crv{(0,7)},
 \endxy~]]\bigg)
 \end{align*}
 \begin{align*}
= & x^2 \langle~\xy (0,0) *\xycircle(3,3){-}, \endxy~
\xy (0,0) *\xycircle(3,3){-}, \endxy~\rangle 
-A^{-3}xy \bigg(A
\langle~\xy
(0,0) *\xycircle(3,3){-}, \endxy~
\xy (0,0) *\xycircle(3,3){-},
\endxy~
\rangle
+ A^{-1}\langle~
\xy
(-8,8);(-2,8) **@{-}, %
(2,8);(8,8) **@{-}, %
(-8,8);(-8,2) **@{-}, %
(8,8);(8,2) **@{-}, %
(-8,-8);(-2,-8) **@{-}, %
(-2,-8);(8,-8) **@{-}, %
(-8,-8);(-8,-2) **@{-}, %
(8,-8);(8,-2) **@{-},%
(2,0);(4,-2) **@{-}, 
(-2,0);(-4,-2) **@{-},  
%
%
(2.4,0.4);(4,2) **@{-},
(0.1,-1.9);(1.7,-0.2) **@{-},  
(0.1,-1.9);(-1.7,-0.3) **@{-},  
(-2.4,0.4);(-4,2) **@{-},
(-8,2);(-6.4,0.4) **@{-}, 
(-5.6,-0.4);(-4,-2) **@{-}, 
(-8,-2);(-4,2) **@{-},
(5.6,-0.4);(4,-2) **@{-}, 
(8,2);(6.6,0.4) **@{-},
(8,-2);(4,2) **@{-},
(-2,4);(-2,8) **\crv{(1,6)}, 
(2,8);(2,4) **\crv{(-1,6)},
(-2,0);(-2,4) **\crv{(1,2)}, 
(2,4);(2,0) **\crv{(-1,2)},
\endxy~\rangle\bigg)\\
&  + yx \langle~\xy
(0,0) *\xycircle(3,3){-}, \endxy~
\xy (0,0) *\xycircle(3,3){-}, \endxy~
\xy (0,0) *\xycircle(3,3){-}, \endxy ~\rangle
+ y^2 \langle~\xy (0,0) *\xycircle(3,3){-}, \endxy~
\xy (0,0) *\xycircle(3,3){-}, \endxy~\rangle\\
= &  (-A^{-2}-A^{2})(x^2+y^2)+(-A^{-2}-A^{2})^2 xy \\
& -A^{-3} \Big(A(-A^{-2}-A^2) +A^{-1}(-A^{-4}-A^{4})^2\Big) xy\\
=& (-A^{-2}-A^{2})(x^2+y^2)+(3-A^{-12})xy. 
 \end{align*}
Substituting $B$ for $A^{-1}$, we get $\ll 7^{0,-2}_1\gg=(-A^{2}-B^{2})(x^2+y^2) + (3-B^{12}) xy.$ This gives $\overline{\ll 7^{0,-2}_1\gg}=1.$
\end{example}

A virtual marked graph diagram is a marked graph diagram possibly with virtual crossings indicated by small circles as usual in virtual link diagrams. In \cite{Kau2}, L. H. Kauffman suggested the notion of isotopy of virtual surface-links in four space by means of virtual marked graph diagrams modulo a generalization of Yoshikawa moves on marked graph diagrams in purpose to investigate the relationships between this diagrammatic definition and more geometric approaches to virtual 2-knots. In \cite{NR}, S. Nelson and P. Rivera introduced an isotopy invariant of virtual surface-links presented by virtual marked graph diagrams by using ribbon biquandles.

Note that the (normalized) Kauffman bracket polynomial is an invariant for virtual links \cite{Kau1}. This confirms that the construction of the ideal coset invariant associated with the (normalized) Kauffman bracket polynomial can be extended to (oriented) virtual surface-links presented by virtual marked graph diagrams and consequently the ideal coset invariant associated with the (normalized) Kauffman bracket polynomial is also invariant for (oriented) virtual surface-links. In a separate paper \cite{Le4}, this extension will be dealt with full details in a more general setting. It is noted that the examples \ref{examp-T^2}-\ref{examp-2comp-sl-7^{0,-2}_1} above are implicit that the (normalized) Kauffman bracket ideal coset invariant seems to be almost trivial for surface-links and that this conjecture will be discussed in \cite{Le4} in details. On the other hand, the following example \ref{examp-virtual-sknot} shows that the (normalized) Kauffman bracket ideal coset invariant is highly nontrivial for (oriented) virtual surface-links. 
 
\begin{example}\label{examp-virtual-sknot}
Consider the oriented virtual $S^2$-knot $D$ below. Since
\begin{align*}
\langle~&\xy 
(2,0);(4,-2) **@{-}, 
(-2,0);(-4,-2) **@{-},  
(-2,0) *\xycircle(1,1){-},
(2,0) *\xycircle(1,1){-},
%
(2.4,0.4);(4,2) **@{-},
(0.1,-1.9);(1.7,-0.2) **@{-},  
(0.1,-1.9);(-1.7,-0.3) **@{-},  
(-2.4,0.4);(-4,2) **@{-},
(-12,2);(-10.4,0.4) **@{-}, 
(-9.6,-0.4);(-8,-2) **@{-}, 
(-12,-2);(-8,2) **@{-},
(-8,2);(-6.4,0.4) **@{-}, 
(-5.6,-0.4);(-4,-2) **@{-}, 
(-8,-2);(-4,2) **@{-},
(12,2);(10.4,0.4) **@{-}, 
(9.6,-0.4);(8,-2) **@{-},
(12,-2);(8,2) **@{-},
(5.6,-0.4);(4,-2) **@{-}, 
(8,2);(6.6,0.4) **@{-},
(8,-2);(4,2) **@{-},
(-9,0.4)*{\urcorner}, 
(-9,-0.8)*{\lrcorner}, 
(11,0.5)*{\urcorner}, 
(11,-0.8)*{\lrcorner},  
(-12,2);(-2,4) **\crv{(-14,5)}, 
(2,4);(12,2) **\crv{(14,5)},
(-2,0);(-2,4) **\crv{(1,2)}, 
(2,4);(2,0) **\crv{(-1,2)},
(-12,-2);(0,-5) **\crv{(-14,-5)},
(0,-5);(12,-2) **\crv{(14,-5)},
 \endxy~\rangle
=A\langle~\xy
(12,2);(10.4,0.4) **@{-}, 
(9.6,-0.4);(8,-2) **@{-},
(12,-2);(8,2) **@{-},
(5.6,-0.4);(4,-2) **@{-}, 
(8,2);(6.6,0.4) **@{-},
(8,-2);(4,2) **@{-},
(11,0.5)*{\urcorner}, 
(11,-0.8)*{\lrcorner},  
(2,0) *\xycircle(1,1){-},
%
(4,2);(2.4,0.4) **@{-}, 
(1.6,-0.4);(0,-2) **@{-},
(4,-2);(0,2) **@{-},
(0,2);(12,2) **\crv{(5,6)},
(0,-2);(12,-2) **\crv{(5,-6)},
\endxy~\rangle
+A^{-1}\langle~\xy 
(2,0);(4,-2) **@{-}, 
(-2,0);(-4,-2) **@{-},  
(-2,0) *\xycircle(1,1){-},
(2,0) *\xycircle(1,1){-},
%
(2.4,0.4);(4,2) **@{-},
(0.1,-1.9);(1.7,-0.2) **@{-},  
(0.1,-1.9);(-1.7,-0.3) **@{-},  
(-2.4,0.4);(-4,2) **@{-},
(-8,2);(-6.4,0.4) **@{-}, 
(-5.6,-0.4);(-4,-2) **@{-}, 
(-8,-2);(-4,2) **@{-},
(12,2);(10.4,0.4) **@{-}, 
(9.6,-0.4);(8,-2) **@{-},
(12,-2);(8,2) **@{-},
(5.6,-0.4);(4,-2) **@{-}, 
(8,2);(6.6,0.4) **@{-},
(8,-2);(4,2) **@{-},
(-5,0.4)*{\urcorner}, 
(-5,-0.8)*{\lrcorner}, 
(11,0.5)*{\urcorner}, 
(11,-0.8)*{\lrcorner},  
(-8,2);(-2,4) **\crv{(-10,5)}, 
(2,4);(12,2) **\crv{(14,5)},
(-2,0);(-2,4) **\crv{(1,2)}, 
(2,4);(2,0) **\crv{(-1,2)},
(-8,-2);(0,-5) **\crv{(-10,-5)},
(0,-5);(12,-2) **\crv{(14,-5)},
 \endxy~\rangle\\
&=A\langle~\xy (0,0) *\xycircle(3,3){-},(3,0) *{\wedge},\endxy~\rangle+
A^{-1}\bigg[A\langle~\xy
(12,2);(10.4,0.4) **@{-}, 
(9.6,-0.4);(8,-2) **@{-},
(12,-2);(8,2) **@{-},
(5.6,-0.4);(4,-2) **@{-}, 
(8,2);(6.6,0.4) **@{-},
(8,-2);(4,2) **@{-},
(11,0.5)*{\urcorner}, 
(11,-0.8)*{\lrcorner},  
(2,0) *\xycircle(1,1){-},
%
(4,2);(2.4,0.4) **@{-}, 
(1.6,-0.4);(0,-2) **@{-},
(4,-2);(0,2) **@{-},
(0,2);(12,2) **\crv{(5,6)},
(0,-2);(12,-2) **\crv{(5,-6)},
\endxy~\rangle+A^{-1}\langle~\xy
(12,2);(10.4,0.4) **@{-}, 
(9.6,-0.4);(8,-2) **@{-},
(12,-2);(8,2) **@{-},
(5.6,-0.4);(4,-2) **@{-}, 
(8,2);(6.6,0.4) **@{-},
(8,-2);(4,2) **@{-},
(11,0.5)*{\urcorner}, 
(11,-0.8)*{\lrcorner},  
(2,0) *\xycircle(1,1){-},
%
(4,2);(2.4,0.4) **@{-}, 
(1.6,-0.4);(0,-2) **@{-},
(4,-2);(0,2) **@{-},
(0,2);(12,2) **\crv{(5,6)},
(0,-2);(12,-2) **\crv{(5,-6)},
\endxy~\rangle\bigg]=A+1+A^{-2},
\end{align*}
we have
\begin{align*}
\ll~&\xy 
(-12,8);(-2,8) **@{-}, 
(2,8);(12,8) **@{-}, 
(-12,8);(-12,2) **@{-}, 
(12,8);(12,2) **@{-}, 
(-12,-8);(-2,-8) **@{-}, 
(-2,-8);(12,-8) **@{-}, 
(-12,-8);(-12,-2) **@{-}, 
(12,-8);(12,-2) **@{-},
(-2,8);(2,4) **@{-}, 
(-2,4);(2,8) **@{-}, 
(-2,4);(4,-2) **@{-}, 
(2,4);(-4,-2) **@{-}, 
(-2,0) *\xycircle(1,1){-},
(2,0) *\xycircle(1,1){-},
%
(2.4,0.4);(4,2) **@{-},
(0.1,-1.9);(1.7,-0.2) **@{-},  
(0.1,-1.9);(-1.7,-0.3) **@{-},  
(-2.4,0.4);(-4,2) **@{-},
(-12,2);(-10.4,0.4) **@{-}, 
(-9.6,-0.4);(-8,-2) **@{-}, 
(-12,-2);(-8,2) **@{-},
(-8,2);(-6.4,0.4) **@{-}, 
(-5.6,-0.4);(-4,-2) **@{-}, 
(-8,-2);(-4,2) **@{-},
(12,2);(10.4,0.4) **@{-}, 
(9.6,-0.4);(8,-2) **@{-},
(12,-2);(8,2) **@{-},
(5.6,-0.4);(4,-2) **@{-}, 
(8,2);(6.6,0.4) **@{-},
(8,-2);(4,2) **@{-},
(-0.1,4.5);(-0.1,7.5) **@{-},
(0,4.5);(0,7.5) **@{-},
(0.1,4.5);(0.1,7.5) **@{-},
(-1.5,1.9);(1.5,1.9) **@{-},
(-1.5,2);(1.5,2) **@{-},
(-1.5,2.1);(1.5,2.1) **@{-}, 
(-9,0.4)*{\urcorner}, 
(-9,-0.8)*{\lrcorner}, 
(11,0.5)*{\urcorner}, 
(11,-0.8)*{\lrcorner}, 
(0,12) *{D},
\endxy\gg_N 
= x^2
\ll~\xy 
(-12,8);(-2,8) **@{-}, 
(2,8);(12,8) **@{-}, 
(-12,8);(-12,2) **@{-}, 
(12,8);(12,2) **@{-}, 
(-12,-8);(-2,-8) **@{-}, 
(-2,-8);(12,-8) **@{-}, 
(-12,-8);(-12,-2) **@{-}, 
(12,-8);(12,-2) **@{-},
(2,0);(4,-2) **@{-}, 
(-2,0);(-4,-2) **@{-}, 
(-2,0) *\xycircle(1,1){-},
(2,0) *\xycircle(1,1){-},
%
(2.4,0.4);(4,2) **@{-},
(0.1,-1.9);(1.7,-0.2) **@{-},  
(0.1,-1.9);(-1.7,-0.3) **@{-},  
(-2.4,0.4);(-4,2) **@{-},
(-12,2);(-10.4,0.4) **@{-}, 
(-9.6,-0.4);(-8,-2) **@{-}, 
(-12,-2);(-8,2) **@{-},
(-8,2);(-6.4,0.4) **@{-}, 
(-5.6,-0.4);(-4,-2) **@{-}, 
(-8,-2);(-4,2) **@{-},
(12,2);(10.4,0.4) **@{-}, 
(9.6,-0.4);(8,-2) **@{-},
(12,-2);(8,2) **@{-},
(5.6,-0.4);(4,-2) **@{-}, 
(8,2);(6.6,0.4) **@{-},
(8,-2);(4,2) **@{-},
(-9,0.4)*{\urcorner}, 
(-9,-0.8)*{\lrcorner}, 
(11,0.5)*{\urcorner}, 
(11,-0.8)*{\lrcorner}, 
(-2,4);(-2,8) **\crv{(1,6)}, 
(2,8);(2,4) **\crv{(-1,6)},
(-2,4);(2,4) **\crv{(0,1)}, 
(-2,0);(2,0) **\crv{(0,3)},
 \endxy\gg_N 
 +xy
 \ll~\xy 
(-12,8);(-2,8) **@{-}, 
(2,8);(12,8) **@{-}, 
(-12,8);(-12,2) **@{-}, 
(12,8);(12,2) **@{-}, 
(-12,-8);(-2,-8) **@{-}, 
(-2,-8);(12,-8) **@{-}, 
(-12,-8);(-12,-2) **@{-}, 
(12,-8);(12,-2) **@{-},
(2,0);(4,-2) **@{-}, 
(-2,0);(-4,-2) **@{-},  
(-2,0) *\xycircle(1,1){-},
(2,0) *\xycircle(1,1){-},
%
(2.4,0.4);(4,2) **@{-},
(0.1,-1.9);(1.7,-0.2) **@{-},  
(0.1,-1.9);(-1.7,-0.3) **@{-},  
(-2.4,0.4);(-4,2) **@{-},
(-12,2);(-10.4,0.4) **@{-}, 
(-9.6,-0.4);(-8,-2) **@{-}, 
(-12,-2);(-8,2) **@{-},
(-8,2);(-6.4,0.4) **@{-}, 
(-5.6,-0.4);(-4,-2) **@{-}, 
(-8,-2);(-4,2) **@{-},
(12,2);(10.4,0.4) **@{-}, 
(9.6,-0.4);(8,-2) **@{-},
(12,-2);(8,2) **@{-},
(5.6,-0.4);(4,-2) **@{-}, 
(8,2);(6.6,0.4) **@{-},
(8,-2);(4,2) **@{-},
(-9,0.4)*{\urcorner}, 
(-9,-0.8)*{\lrcorner}, 
(11,0.5)*{\urcorner}, 
(11,-0.8)*{\lrcorner}, 
(-2,4);(-2,8) **\crv{(1,6)}, 
(2,8);(2,4) **\crv{(-1,6)},
(-2,0);(-2,4) **\crv{(1,2)}, 
(2,4);(2,0) **\crv{(-1,2)},
 \endxy\gg_N +\\
 &\hskip 3.55cm yx
 \ll~\xy 
(-12,8);(-2,8) **@{-}, 
(2,8);(12,8) **@{-}, 
(-12,8);(-12,2) **@{-}, 
(12,8);(12,2) **@{-}, 
(-12,-8);(-2,-8) **@{-}, 
(-2,-8);(12,-8) **@{-}, 
(-12,-8);(-12,-2) **@{-}, 
(12,-8);(12,-2) **@{-},
(2,0);(4,-2) **@{-}, 
(-2,0);(-4,-2) **@{-}, 
(-2,0) *\xycircle(1,1){-},
(2,0) *\xycircle(1,1){-},
%
(2.4,0.4);(4,2) **@{-},
(0.1,-1.9);(1.7,-0.2) **@{-},  
(0.1,-1.9);(-1.7,-0.3) **@{-},  
(-2.4,0.4);(-4,2) **@{-},
(-12,2);(-10.4,0.4) **@{-}, 
(-9.6,-0.4);(-8,-2) **@{-}, 
(-12,-2);(-8,2) **@{-},
(-8,2);(-6.4,0.4) **@{-}, 
(-5.6,-0.4);(-4,-2) **@{-}, 
(-8,-2);(-4,2) **@{-},
(12,2);(10.4,0.4) **@{-}, 
(9.6,-0.4);(8,-2) **@{-},
(12,-2);(8,2) **@{-},
(5.6,-0.4);(4,-2) **@{-}, 
(8,2);(6.6,0.4) **@{-},
(8,-2);(4,2) **@{-},
(-9,0.4)*{\urcorner}, 
(-9,-0.8)*{\lrcorner}, 
(11,0.5)*{\urcorner}, 
(11,-0.8)*{\lrcorner},  
(-2,8);(2,8) **\crv{(0,5)}, 
(-2,4);(2,4) **\crv{(0,7)},
(-2,4);(2,4) **\crv{(0,1)}, 
(-2,0);(2,0) **\crv{(0,3)},
 \endxy\gg_N + y^2
 \ll~\xy 
(-12,8);(-2,8) **@{-}, 
(2,8);(12,8) **@{-}, 
(-12,8);(-12,2) **@{-}, 
(12,8);(12,2) **@{-}, 
(-12,-8);(-2,-8) **@{-}, 
(-2,-8);(12,-8) **@{-}, 
(-12,-8);(-12,-2) **@{-}, 
(12,-8);(12,-2) **@{-},
(2,0);(4,-2) **@{-}, 
(-2,0);(-4,-2) **@{-},  
(-2,0) *\xycircle(1,1){-},
(2,0) *\xycircle(1,1){-},
%
(2.4,0.4);(4,2) **@{-},
(0.1,-1.9);(1.7,-0.2) **@{-},  
(0.1,-1.9);(-1.7,-0.3) **@{-},  
(-2.4,0.4);(-4,2) **@{-},
(-12,2);(-10.4,0.4) **@{-}, 
(-9.6,-0.4);(-8,-2) **@{-}, 
(-12,-2);(-8,2) **@{-},
(-8,2);(-6.4,0.4) **@{-}, 
(-5.6,-0.4);(-4,-2) **@{-}, 
(-8,-2);(-4,2) **@{-},
(12,2);(10.4,0.4) **@{-}, 
(9.6,-0.4);(8,-2) **@{-},
(12,-2);(8,2) **@{-},
(5.6,-0.4);(4,-2) **@{-}, 
(8,2);(6.6,0.4) **@{-},
(8,-2);(4,2) **@{-},
(-9,0.4)*{\urcorner}, 
(-9,-0.8)*{\lrcorner}, 
(11,0.5)*{\urcorner}, 
(11,-0.8)*{\lrcorner}, 
(-2,0);(-2,4) **\crv{(1,2)}, 
(2,4);(2,0) **\crv{(-1,2)},
(-2,8);(2,8) **\crv{(0,5)}, 
(-2,4);(2,4) **\crv{(0,7)},
 \endxy\gg_N\\
 &= x^2\langle~\xy 
 (0,0) *\xycircle(3,3){-}, (3,0) *{\wedge},
 \endxy~\xy 
 (0,0) *\xycircle(3,3){-}, (3,0) *{\wedge},
 \endxy~\rangle_N + xy\langle~\xy 
(2,0);(4,-2) **@{-}, 
(-2,0);(-4,-2) **@{-},  
(-2,0) *\xycircle(1,1){-},
(2,0) *\xycircle(1,1){-},
%
(2.4,0.4);(4,2) **@{-},
(0.1,-1.9);(1.7,-0.2) **@{-},  
(0.1,-1.9);(-1.7,-0.3) **@{-},  
(-2.4,0.4);(-4,2) **@{-},
(-12,2);(-10.4,0.4) **@{-}, 
(-9.6,-0.4);(-8,-2) **@{-}, 
(-12,-2);(-8,2) **@{-},
(-8,2);(-6.4,0.4) **@{-}, 
(-5.6,-0.4);(-4,-2) **@{-}, 
(-8,-2);(-4,2) **@{-},
(12,2);(10.4,0.4) **@{-}, 
(9.6,-0.4);(8,-2) **@{-},
(12,-2);(8,2) **@{-},
(5.6,-0.4);(4,-2) **@{-}, 
(8,2);(6.6,0.4) **@{-},
(8,-2);(4,2) **@{-},
(-9,0.4)*{\urcorner}, 
(-9,-0.8)*{\lrcorner}, 
(11,0.5)*{\urcorner}, 
(11,-0.8)*{\lrcorner}, 
(-12,2);(-2,4) **\crv{(-14,5)}, 
(2,4);(12,2) **\crv{(14,5)},
(-2,0);(-2,4) **\crv{(1,2)}, 
(2,4);(2,0) **\crv{(-1,2)},
(-12,-2);(0,-5) **\crv{(-14,-5)},
(0,-5);(12,-2) **\crv{(14,-5)},
 \endxy~\rangle_N+\\
&\hskip 3.3cm yx\langle~\xy
(0,0) *\xycircle(3,3){-}, (3,0) *{\wedge},\endxy~\xy (0,0) *\xycircle(3,3){-}, (3,0) *{\wedge},\endxy~\xy (0,0) *\xycircle(3,3){-}, (3,0) *{\wedge},\endxy~\rangle_N
+y^2\langle~\xy
(0,0) *\xycircle(3,3){-}, (3,0) *{\wedge},\endxy~\xy (0,0) *\xycircle(3,3){-}, (3,0) *{\wedge},\endxy~\rangle_N\\
&= (-A^{-2}-A^{2})(x^2+y^2)+ (-A^{-2}-A^{2})^2xy+(A+1+A^{-2})xy\\
&=(-A^{-2}-A^{2})(x^2+y^2)+ (A^4+A+3+A^{-2}+A^{-4})xy.
 \end{align*}
 Substituting $B$ for $A^{-1}$, we get $$\ll D\gg_N=(-A^{2}-B^{2})(x^2+y^2) + (A^4+A+3+B^{2}+B^{4}) xy.$$ This yields the normal form $$\overline{\ll D\gg}_N=-B^2y^2-Ay^2+B^2y+Ay+1.$$
\end{example}

In \cite{Le4}, it is seen that the normalized Kauffman bracket ideal coset invariant distills genuine oriented virtual marked graph diagrams from oriented virtual marked graph diagrams. 

\bigskip

\noindent{\bf Question.} Is there a regular or an ambient isotopy invariant $[~~]$ for knots and links in $3$-space such that the associated $[~~]$ ideal is $< 0 >$?

\bigskip

If such a link invariant $[~~]$ exists, then it follows from Theorem \ref{thm1-Ideal-inv} that it is naturally extended to a polynomial
invariant for surface-links in $4$-space.


\section{A modification of the Kauffman bracket ideal}
\label{sect-inv-kbp-spl}

In this section, we consider a modification of the construction of the Kauffman bracket ideal in view of Corollary \ref{cor1-Ideal-inv} which leads new invariants for (oriented) surface-links. We give the modification in the subsection \ref{sect-inv-kbp-spl-1} and compute the invariants for various surface-links in Yoshikawa's table \cite{Yo} in the subsection \ref{sect-inv-kbp-spl-2}.

\subsection{A family of invariants for surface-links}
\label{sect-inv-kbp-spl-1}

Let 
\begin{equation*}
\begin{split}
&z(t)=\frac{1}{2}\Big(\sqrt{3-t^{-1}}+\mathbf{i}\sqrt{1+t^{-1}}\Big),~~\overline{z}(t)=\frac{1}{2}\Big(\sqrt{3-t^{-1}}-\mathbf{i}\sqrt{1+t^{-1}}\Big),
\end{split}
\end{equation*}
where $t\not= 0$ and $\mathbf{i}=\sqrt{-1}$. We observe that
$z(t) \overline{z}(t)=\frac{1}{4}(3-t^{-1}-\mathbf{i}^2(1+t^{-1}))=1$ and so
we have $\overline{z}(t)=z(t)^{-1}$. For any polynomial $f=f(A, A^{-1}, x, y) \in \mathbb Z[A, A^{-1}, x, y],$ we define $\phi f$ by $\phi f=f(z(t), z(t)^{-1}, t, t)=f(z(t), \overline{z}(t), t, t).$ 
It is obvious that for every polynomial $f \in \mathbb Z[A, A^{-1}, x, y]$, the evaluation $\phi f$ is expressed as the form:
\begin{align}
\phi f &=a_0(t)+\mathbf{i} b_0(t)+\Big(a_1(t)+\mathbf{i} b_1(t)\Big)\sqrt{1+t^{-1}}
\notag\\&+\Big(a_2(t)+\mathbf{i}b_2(t)\Big)
\sqrt{3-t^{-1}}+\Big(a_3(t)+\mathbf{i}b_3(t)\Big)\sqrt{3+2t^{-1}-t^{-2}},
\label{irr-ft-form}
\end{align}
where $a_i(t), b_i(t) \in \mathbb Z[2^{-1}, t^{-1}, t] (0 \leq i \leq 3)$, namely, polynomials in variables $2^{-1}$ and $t^{\pm 1}$ with integral coefficients.
Now we make the following definition:

\begin{definition}\label{defn-sp=poly-1}
Let $D$ be an oriented marked graph diagram and let $\ll D\gg_N$ be the polynomial of $D$ associated with the normalized Kauffman bracket polynomial $\langle~~~\rangle_N$. We define ${\mathbf K}(D;t)_N$ (for short, ${\mathbf K}(D)_N$) by the formula:
\begin{equation*}
{\mathbf K}(D)_N={\mathbf K}(D;t)_N=\phi\ll D\gg_N = (-z(t)^3)^{-w(D)}[[\widetilde{D}]](z(t), \overline{z}(t), t, t).
\end{equation*}
\end{definition}

\begin{lemma}\label{lem-inv-y6-rkb}
Let $D$ be an oriented marked graph diagram. Then ${\mathbf K}(D)_N$ is an invariant for all oriented Yoshikawa moves, except for the moves $\Gamma_7$ and $\Gamma_8$.
\end{lemma}

\begin{proof}
The invariance for oriented Yoshikawa moves of type I is direct from Theorem \ref{thm-inv-mgs}. To prove the invariance of ${\mathbf K}(D)_N$ for the oriented Yoshikawa moves $\Gamma_6$ and $\Gamma'_6$, we observe that
$-z(t)^2-z(t)^{-2}=-z(t)^2-\overline{z}(t)^{2}=t^{-1}-1.$
Note that $\delta=-A^2-A^{-2}$ for the Kauffman bracket. In (\ref{pf-inv-y6-1}) and (\ref{pf-inv-y6-2}) in Lemma \ref{lem1-Ideal-inv}, 
\begin{align*}
\phi(\delta x+y)=(-z(t)^2-z(t)^{-2})t+t=(t^{-1}-1)t+t=1,\\
\phi(x+\delta y)=t+(-z(t)^2-z(t)^{-2})t=t+(t^{-1}-1)t=1.
\end{align*} 
This shows that ${\mathbf K}(D)_N$ is invariant under $\Gamma_6$ and $\Gamma'_6$. This completes the proof.
\end{proof}

\begin{lemma}\label{lem-inv-y78-rkb}
Let $D$ be an oriented marked graph diagram and let $D'$ be an oriented marked graph diagram obtained from $D$ by applying a single Yoshikawa move $\Gamma_7$ or $\Gamma_8$. Then
\begin{equation}
{\mathbf K}(D')_N={\mathbf K}(D)_N + (2t-1)\Psi(t),\label{eqn2-inv-y78}
\end{equation}
where $\Psi(t)$ is of the form in (\ref{irr-ft-form}).
\end{lemma}

\begin{proof}
Let $D$ be an oriented marked graph diagram and let $D'$ be the oriented marked graph diagram obtained from $D$ by applying a single Yoshikawa move $\Gamma_7$. 
Since $z(t)^4+1+\overline{z}(t)^{4}
=(-z(t)^2-\overline{z}(t)^{2})^2-1=(t^{-1}-1)^2-1=t^{-2}-2t^{-1},$
it follows from (\ref{eq2-pf-m7-1}) with the substitutions $A=z(t), A^{-1}=\overline{z}(t), x=y=t$ that
$$(\langle R(U)\rangle-\langle R^*(U)\rangle)t^2
=(1-2t)\Big(\psi_3(U)-\psi_4(U)\Big).$$
By Lemma \ref{lem2-Ideal-inv} by taking the normalized Kauffman bracket, we have
\begin{align*}
{\mathbf K}(D')_N&-{\mathbf K}(D)_N=\phi(\ll D'\gg)-\phi(\ll D\gg)\\
&=(-z(t)^3)^{-w(D)}\sum_{k=1}^m\phi(\psi_k(A,A^{-1},x,y))(1-2t)\phi\Big(\psi_3(U_k)-\psi_4(U_k)\Big)\\
&=(2t-1)\Psi(t),
\end{align*}
where 
$\Psi(t)=(-z(t)^3)^{-w(D)}\sum_{k=1}^m\phi(\psi_k(A,A^{-1},x,y))
\phi\Big(\psi_4(U_k)-\psi_3(U_k)\Big).$ Since $\psi_4(U_k)-\psi_3(U_k) \in \mathbb Z[A,A^{-1},x,y]$ for all $k=1,2,\ldots,m$, it is clear that $\Psi(t)$ has the form in (\ref{irr-ft-form}). 

Now let $D$ be an oriented marked graph diagram and let $D'$ be an oriented marked graph diagram obtained from $D$ by applying a single oriented Yoshikawa move $\Gamma_8$ (cf. Figure~\ref{fig-m08-ort}). By a similar argument combined with Lemma \ref{lem3-Ideal-inv} and (\ref{eq2-pf-m8-1}), we obtain the identity (\ref{eqn2-inv-y78}). This completes the proof.
\end{proof}

The following theorem \ref{thm-skein-rel-kb-1}, which gives a method of computing the polynomial ${\mathbf K}(D)_N$ recursively for a given oriented marked graph diagram $D$.

\begin{theorem}\label{thm-skein-rel-kb-1}
Let $D$ be an oriented marked graph diagram. 
\begin{itemize}
\item[(1)] 
${\mathbf K}(~\xy (0,0) *\xycircle(3,3){-}, (3,0) *{\wedge}, \endxy~)_N = 1.$
\item[(2)] 
If $D$ and $D'$ are two oriented marked graph diagrams related by a finite sequence of oriented Yoshikawa moves generated by the moves $\Gamma_1, \Gamma'_1, \Gamma_2, \Gamma_3, \Gamma_4,$ $\Gamma'_4, \Gamma_5, \Gamma_6$ and $\Gamma'_6$, then ${\mathbf K}(D)_N={\mathbf K}(D')_N.$
\item[(3)] 
${\mathbf K}(D \sqcup~ \xy (0,0) *\xycircle(3,3){-}, (3,0) *{\wedge}, \endxy~)_N = 
(t^{-1}-1){\mathbf K}(D )_N.$
\item[(4)] 
${\mathbf K}\Big(~\xy (-4,4);(4,-4) **@{-}, 
(4,4);(-4,-4) **@{-}, 
(3,3.2)*{\llcorner}, 
(-3,-3.4)*{\urcorner}, 
(-2.5,2)*{\ulcorner},
(2.5,-2.4)*{\lrcorner}, 
(3,-0.2);(-3,-0.2) **@{-},
(3,0);(-3,0) **@{-}, 
(3,0.2);(-3,0.2) **@{-}, 
\endxy~\Big)_N =t\bigg[
{\mathbf K}\Big(~\xy (-4,4);(4,4) **\crv{(0,-1)}, 
(4,-4);(-4,-4) **\crv{(0,1)},  
(-2.5,1.9)*{\ulcorner}, (2.5,-2.4)*{\lrcorner},   
\endxy~\Big)_N  + {\mathbf K}\Big(~\xy (-4,4);(-4,-4) **\crv{(1,0)},  
(4,4);(4,-4) **\crv{(-1,0)}, (-2.5,1.9)*{\ulcorner}, (2.5,-2.4)*{\lrcorner}, 
\endxy~\Big)_N \bigg].$
\item[(5)] 
$\lambda(t){\mathbf K}\Big(~\xy 
(5,5);(-5,-5) **@{-} ?<*\dir{<},
(-5,5);(-2,2) **@{-} ?<*\dir{<}, 
(2,-2);(5,-5) **@{-},(0,-7)*{D_+},
\endxy~\Big)_N - \overline{\lambda}(t){\mathbf K}\Big(~\xy 
(-5,5);(5,-5) **@{-} ?<*\dir{<},
(5,5);(2,2) **@{-} ?<*\dir{<}, 
(-2,-2);(-5,-5) **@{-},(0,-7)*{D_-},
\endxy~\Big)_N =\mathbf{i}\sqrt{1+t^{-1}}\sqrt{3-t^{-1}}~
{\mathbf K}\Big(~\xy (-4,4);(-4,-4) **\crv{(1,0)},  
(4,4);(4,-4) **\crv{(-1,0)}, 
(-2.5,1.9)*{\ulcorner}, 
(2.5,1.9)*{\urcorner}, 
(0,-7)*{D_0},
\endxy~\Big)_N,$
where $D_+, D_-, D_0$ are three identical oriented link diagrams except the parts indicated and $\lambda(t)=2^{-1}t^{-2}\Big((t^2+2t-1)-\mathbf{i}(t^2-t)\sqrt{1+t^{-1}}\sqrt{3-t^{-1}}\Big)$.
\end{itemize}
\end{theorem}

\begin{proof}
From Theorem \ref{thm-skein-rel-kb} with the substitutions $A=z(t), A^{-1}=\overline{z}(t), x=y=t$ and Lemma \ref{lem-inv-y6-rkb}, we obtain
the assertions (1), (2), (3) and (4) directly. Observe that
\begin{align*}
&z(t)^4=-2^{-1}t^{-2}\Big((t^2+2t-1)-(t^2-t)\mathbf{i}\sqrt{1+t^{-1}}\sqrt{3-t^{-1}}\Big),\\
&\overline{z}(t)^4=-2^{-1}t^{-2}\Big(t^2+2t-1+(t^2-t)\mathbf{i}\sqrt{1+t^{-1}}\sqrt{3-t^{-1}}\Big),\\
&\overline{z}(t)^2-z(t)^2=-\mathbf{i}\sqrt{1+t^{-1}}\sqrt{3-t^{-1}}.
\end{align*}
Hence the assertions (5) follows immediately from the skein relation in (5) in Theorem \ref{thm-skein-rel-kb}. This completes the proof.
\end{proof}

Let $\mathbb C$ denote the field of all complex numbers. For each integer $n \geq 2$, let $\phi_n : \mathbb Z[A, A^{-1}, x, y] \to \mathbb C$ be the function defined by $$\phi_n(f)=\phi f|_{t=n}=f(z(n), \overline{z}(n), n, n)$$ for each polynomial $f = f(A, A^{-1}, x, y) \in \mathbb Z[A, A^{-1}, x, y]$. Then $\phi_n$ is a ring homomorphism and the image of $\phi_n$ is a subring of $\mathbb C$. From (\ref{irr-ft-form}), we see that $\phi_n(f)$ is expressed as a complex number of the form:
\begin{align}
\phi_n(f)&=a_0(n)+\mathbf{i} b_0(n)+\Big(a_1(n)+\mathbf{i} b_1(n)\Big)\sqrt{1+n^{-1}}+
\Big(a_2(n)+\mathbf{i}b_2(n)\Big)\sqrt{3-n^{-1}}\notag\\
&+\Big(a_3(n)+\mathbf{i}b_3(n)\Big)\sqrt{1+n^{-1}}\sqrt{3-n^{-1}},
\label{irr-ft-form-1}
\end{align}
where $a_i(n), b_i(n) \in \mathbb Z[2^{-1}, n^{-1}] (0 \leq i \leq 3)$, namely, polynomials in variables $2^{-1}$ and $n^{-1}$ with integral coefficients. Let
$$\Lambda_{n}=\{\phi_n(f)~|~ f \in \mathbb Z[A, A^{-1}, x, y]\}.$$ 
Now let $\mathbb Z_{2n-1}=\{[0], [1], [2],\ldots, [2n-1]\}$ denote the factor ring $\mathbb Z/(2n-1)\mathbb Z$ of the ring $\mathbb Z$ of integers modulo $2n-1$ and let $\overline{\phi_n(f)}$ denote $\phi_n(f)$ reducing integral coefficients of each terms $a_i(n)$ and $b_i(n) (0 \leq i \leq 3)$ modulo $2n-1$ (hence the reduced coefficients are all in $\mathbb Z_{2n-1}$). 
Then we have the following:

\begin{theorem}\label{thm-inv-kb-1} 
Let $\mathcal L$ be an oriented surface-link and let $D$ be an oriented marked graph diagram presenting $\mathcal L$. Then for each integer $n \geq 2$, ${\mathbf K}_{2n-1}(D)_N=\overline{{\mathbf K}(D;n)_N}$ is an invariant of $\mathcal L$, and denoted by ${\mathbf K}_{2n-1}(\mathcal L)_N$.
\end{theorem}

\begin{proof}
By Lemma \ref{lem-inv-y6-rkb}, ${\mathbf K}(D;n)$ is invariant under oriented Yoshikawa moves of type I and the moves $\Gamma_6$ and $\Gamma'_6$ and so does ${\mathbf K}_{2n-1}(D)=\overline{{\mathbf K}(D;n)}$ for each integer $n \geq 2$. 

Let $D$ be an oriented marked graph diagram and let $D'$ be an oriented marked graph diagram obtained from $D$ by applying a single oriented Yoshikawa move $\Gamma_7$ or $\Gamma_8$. By Lemma \ref{lem-inv-y78-rkb}, we see that for each integer $n \geq 2$,
\begin{equation*}
{\mathbf K}(D;n)_N={\mathbf K}(D';n)_N + (2n-1)\Psi(n),
\end{equation*}
where $\Psi(n) \in \Lambda_n.$ 
Hence we have
\begin{align*}
{\mathbf K}_{2n-1}(D)_N&=\overline{{\mathbf K}(D;n)}_N=\overline{{\mathbf K}(D';n)_N + (2n-1)\Psi(n)}\\&=\overline{{\mathbf K}(D';n)}_N={\mathbf K}_{2n-1}(D')_N.
\end{align*}
This completes the proof.
\end{proof}

\begin{definition}\label{defn-sp=poly-2}
Let $D$ be a marked graph diagram and let $\ll D\gg$ be the polynomial of $D$ associated with the Kauffman bracket polynomial. Define ${\mathbf K}(D)$ by 
\begin{equation*}
{\mathbf K}(D)={\mathbf K}(D;t)=\phi\ll D\gg = (-z(t)^3)^{-sw(D)}[[D]](z(t), \overline{z}(t), t, t).
\end{equation*}
\end{definition}

By parallel argument of the proof forgetting orientation, we obtain the corresponding Lemmas \ref{lem-inv-y6-rkb} and \ref{lem-inv-y78-rkb} for ${\mathbf K}(D)$ and consequently we obtain the following:

\begin{theorem}\label{thm-inv-kb-2} 
Let $\mathcal L$ be a surface-link and let $D$ be a marked graph diagram presenting $\mathcal L$. Then for each integer $n \geq 2$, ${\mathbf K}_{2n-1}(D)=\overline{{\mathbf K}(D;n)}$ is an invariant of $\mathcal L$, and denoted by ${\mathbf K}_{2n-1}(\mathcal L)$.
\end{theorem}

\begin{remark}
For each pair $\epsilon=(\epsilon_1, \epsilon_2) \in \{(1,-1), (-1,1), (-1,-1)\}$, let
\begin{equation*}
\begin{split}
&z_\epsilon(t)=\frac{1}{2}\Big(\epsilon_1\sqrt{3-t^{-1}}+\epsilon_2\mathbf{i}\sqrt{1+t^{-1}}\Big),~~
\overline{z_\epsilon(t)}=\frac{1}{2}\Big(\epsilon_1\sqrt{3-t^{-1}}-\epsilon_2\mathbf{i}\sqrt{1+t^{-1}}\Big),
\end{split}
\end{equation*}
where $t\not= 0$. It is easily seen that $z_\epsilon(t)\overline{z}_\epsilon(t)=1$ and
$-z_\epsilon(t)^2-z_\epsilon(t)^{-2}=-z_\epsilon(t)^2-\overline{z}_\epsilon(t)^{2}=t^{-1}-1.$ In the proof of Lemma \ref{lem-inv-y6-rkb} with $\delta=-A^2-A^{-2}$, it is checked that
\begin{align*}
\phi(\delta x+y)=(-z_\epsilon(t)^2-z_\epsilon(t)^{-2})t+t=(t^{-1}-1)t+t=1,\\
\phi(x+\delta y)=t+(-z_\epsilon(t)^2-z_\epsilon(t)^{-2})t=t+(t^{-1}-1)t=1.
\end{align*} 
In the proof of Lemma \ref{lem-inv-y78-rkb}, it is checked that
$z_\epsilon(t)^4+1+\overline{z}_\epsilon(t)^{4}
=(-z_\epsilon(t)^2-\overline{z}_\epsilon(t)^{2})^2-1=(t^{-1}-1)^2-1=t^{-2}-2t^{-1}$ and it also follows from (\ref{eq2-pf-m7-1}) with the substitution $A=z_\epsilon(t), A^{-1}=z_\epsilon(t)^{-1}, x=y=t$ that
$(\langle R(U)\rangle-\langle R^*(U)\rangle)t^2
=(1-2t)\Big(\psi_3(U)-\psi_4(U)\Big).$ Now we define
\begin{align*}
&{\mathbf K}_\epsilon(D;t)=\phi_\epsilon\ll D\gg =(-z_\epsilon(t)^3)^{-sw(D)}[[D]](z_\epsilon(t), \overline{z}_\epsilon(t), t, t),\\
&{\mathbf K}_\epsilon(D;t)_N=\phi_\epsilon\ll D\gg_N = (-z_\epsilon(t)^3)^{-w(D)}[[\widetilde{D}]](z_\epsilon(t), \overline{z}_\epsilon(t), t, t),
\end{align*}
where $\phi_\epsilon f=f(z_\epsilon(t), \overline{z_\epsilon(t)}, t, t)$ for $f \in \mathbb Z[A, A^{-1}, x, y]$. By the same argument of the proof of Theorem \ref{thm-inv-kb-1} with $\phi=\phi_\epsilon$, we obtain that for each integer $n \geq 2$, ${\mathbf K}^\epsilon_{2n-1}(D)=\overline{{\mathbf K}_\epsilon(D;n)}$ (resp.~${\mathbf K}^\epsilon_{2n-1}(D)_N=\overline{{\mathbf K}_\epsilon(D;n)_N}$) is also an invariant of an (resp.~oriented) surface-link $\mathcal L$ presented by an (resp.~oriented) marked graph diagram $D$.
\end{remark} 


\subsection{Examples}
\label{sect-inv-kbp-spl-2}

In this subsection, we calculate the invariants ${\mathbf K}_{2n-1}(\mathcal L)$ and ${\mathbf K}_{2n-1}(\mathcal L)_N$ for various surface-links $\mathcal L$.

\begin{example}\label{examp-T^2-1} 
Let $2^1_1$ be the standard torus of genus one with the orientation in Example \ref{examp-T^2}, where it is seen that
$\ll \widetilde{2^{1}_1}\gg=\ll 2^{1}_1\gg_N=x^2+2(-A^2-A^{-2})xy+y^2.$
This gives that for each integer $n \geq 2$,
${\mathbf K}( \widetilde{2^{1}_1};n)={\mathbf K}(2^1_1;n)_N
=n^2+2(-z(n)^2-\overline{z}(n)^{2})n^2+n^2=2n.$ Therefore $${\mathbf K}_{2n-1}( \widetilde{2^{1}_1})={\mathbf K}_{2n-1}(2^1_1)_N=\overline{2n}=[1]~ (n\geq 2).$$
\end{example}

\begin{example}\label{exmp-RP2-1}
Let $2_1^{-1}$ be the positive standard projective plane. Then it is seen from Example \ref{exmp-RP2} that $\ll 2_1^{-1}\gg=(-A^3)x+(-A^{-3})y.$
This gives that for each integer $n \geq 2$,
${\mathbf K}(2_1^{-1};n)=(-z(n)^3)n+(-\overline{z}(n)^{3})n=n^{-\frac{1}{2}}\sqrt{3n-1}.$ Therefore $${\mathbf K}_{2n-1}(2_1^{-1})=[1]n^{-\frac{1}{2}}\sqrt{3n-1}~(n\geq 2).$$
Similarly, let $2_1^{-1*}$ be the negative standard projective plane. Then for each $n \geq 2$, we have ${\mathbf K}_{2n-1}(2_1^{-1*})={\mathbf K}_{2n-1}(2_1^{-1})=[1]n^{-\frac{1}{2}}\sqrt{3n-1}.$
\end{example}

\begin{example}\label{examp-2compsl-6^01_1-1}
Let $6^{0,1}_1$ be the two component orientable surface-link with the orientation in Example \ref{examp-2comp-sl-6^{0,1}_1}, where we have $\ll 6^{0,1}_1\gg_N=(-A^{2}-A^{-2})(x^2+y^2) + (A^4+4+A^{-4}+A^{-8}+A^8)xy.$ This gives
${\mathbf K}(6^{0,1}_1;t)_N=\ll 6^{0,1}_1\gg_N(z(t), \overline{z}(t), t, t)=(4t^3+3t^2-4t+1)t^{-2}$. Hence for each integer $n \geq 2$,
$${\mathbf K}_{2n-1}(6^{0,1}_1)_N=\overline{{\mathbf K}(6^{0,1}_1;n)_N}=[4n^3+3n^2-4n+1]n^{-2}.$$
Further, forgetting the orientation on $6^{0,1}_1$, we see from Example \ref{examp-2comp-sl-6^{0,1}_1} that $\ll 6^{0,1}_1\gg_N$ $=\ll \widetilde{6^{0,1}_1}\gg$ and therefore ${\mathbf K}_{2n-1}(\widetilde{6^{0,1}_1})={\mathbf K}_{2n-1}(6^{0,1}_1)_N$ for all $n\geq 2$.
\end{example}

\begin{example}\label{examp-2compsl-7^02_1-1}
Let $7^{0,-2}_1$ be the two component nonorientable surface-link in Example \ref{examp-2comp-sl-7^{0,-2}_1}, where we get $\ll 7^{0,-2}_1\gg=(-A^{-2}-A^{2})(x^2+y^2)+(3-A^{-12})xy.$ This gives
${\mathbf K}(7^{0,-2}_1;t)=\ll 7^{0,-2}_1\gg(z(t), \overline{z}(t), t, t)=(4t^5+12t^4-4t^3-9t^2+6t-1)2^{-1}t^{-4}
+(4t^4-2t^3-6t^2+5t-1)2^{-1}t^{-4}\mathbf{i}\sqrt{(t+1)(3t-1)}.$
Hence for each integer $n \geq 2$, 
\begin{align*}
{\mathbf K}_{2n-1}(7^{0,-2}_1)
&=\overline{{\mathbf K}(7^{0,-2}_1;n)}
=[4n^5+12n^4-4n^3-9n^2+6n-1]2^{-1}n^{-4}\\
&+[4n^4-2n^3-6n^2+5n-1]2^{-1}n^{-4}\mathbf{i}\sqrt{(n+1)(3n-1)}.
\end{align*}
\end{example}

\begin{example}\label{examp-spun-2-knot-8_1-1}
Let $8_1$ be the spun $2$-knot of the trefoil in Yoshikawa's table (cf. \cite{Kaw,Yo}) with the orientation indicated below. By Theorem \ref{thm-skein-rel-kb} together with (\ref{nkbp-exmps}) and (\ref{skein-rel-exmp}), we have
\begin{align*}
\ll~&\xy 
(-12,8);(-2,8) **@{-}, 
(2,8);(12,8) **@{-}, 
(-12,8);(-12,2) **@{-}, 
(12,8);(12,2) **@{-}, 
(-12,-8);(-2,-8) **@{-}, 
(-2,-8);(12,-8) **@{-}, 
(-12,-8);(-12,-2) **@{-}, 
(12,-8);(12,-2) **@{-},
(-2,8);(2,4) **@{-}, 
(-2,4);(2,8) **@{-}, 
(-2,4);(4,-2) **@{-}, 
(2,4);(-4,-2) **@{-}, 
%
%
(2.4,0.4);(4,2) **@{-},
(0.1,-1.9);(1.7,-0.2) **@{-},  
(0.1,-1.9);(-1.7,-0.3) **@{-},  
(-2.4,0.4);(-4,2) **@{-},
(-12,2);(-10.4,0.4) **@{-}, 
(-9.6,-0.4);(-8,-2) **@{-}, 
(-12,-2);(-8,2) **@{-},
(-8,2);(-6.4,0.4) **@{-}, 
(-5.6,-0.4);(-4,-2) **@{-}, 
(-8,-2);(-4,2) **@{-},
(12,2);(10.4,0.4) **@{-}, 
(9.6,-0.4);(8,-2) **@{-},
(12,-2);(8,2) **@{-},
(5.6,-0.4);(4,-2) **@{-}, 
(8,2);(6.6,0.4) **@{-},
(8,-2);(4,2) **@{-},
(-0.1,4.5);(-0.1,7.5) **@{-},
(0,4.5);(0,7.5) **@{-},
(0.1,4.5);(0.1,7.5) **@{-},
(-1.5,1.9);(1.5,1.9) **@{-},
(-1.5,2);(1.5,2) **@{-},
(-1.5,2.1);(1.5,2.1) **@{-}, 
(-9,0.4)*{\urcorner}, 
(-9,-0.8)*{\lrcorner}, 
(11,0.5)*{\urcorner}, 
(11,-0.8)*{\lrcorner}, 
(0.5,-11.2) *{8_1},
\endxy\gg_N =
x^2 \ll~\xy 
(-12,8);(-2,8) **@{-}, 
(2,8);(12,8) **@{-}, 
(-12,8);(-12,2) **@{-}, 
(12,8);(12,2) **@{-}, 
(-12,-8);(-2,-8) **@{-}, 
(-2,-8);(12,-8) **@{-}, 
(-12,-8);(-12,-2) **@{-}, 
(12,-8);(12,-2) **@{-},
(2,0);(4,-2) **@{-}, 
(-2,0);(-4,-2) **@{-}, 
%
%
(2.4,0.4);(4,2) **@{-},
(0.1,-1.9);(1.7,-0.2) **@{-},  
(0.1,-1.9);(-1.7,-0.3) **@{-},  
(-2.4,0.4);(-4,2) **@{-},
(-12,2);(-10.4,0.4) **@{-}, 
(-9.6,-0.4);(-8,-2) **@{-}, 
(-12,-2);(-8,2) **@{-},
(-8,2);(-6.4,0.4) **@{-}, 
(-5.6,-0.4);(-4,-2) **@{-}, 
(-8,-2);(-4,2) **@{-},
(12,2);(10.4,0.4) **@{-}, 
(9.6,-0.4);(8,-2) **@{-},
(12,-2);(8,2) **@{-},
(5.6,-0.4);(4,-2) **@{-}, 
(8,2);(6.6,0.4) **@{-},
(8,-2);(4,2) **@{-},
(-9,0.4)*{\urcorner}, 
(-9,-0.8)*{\lrcorner}, 
(11,0.5)*{\urcorner}, 
(11,-0.8)*{\lrcorner}, 
(-2,4);(-2,8) **\crv{(1,6)}, 
(2,8);(2,4) **\crv{(-1,6)},
(-2,4);(2,4) **\crv{(0,1)}, 
(-2,0);(2,0) **\crv{(0,3)},
 \endxy\gg_N 
 +xy \ll~\xy 
(-12,8);(-2,8) **@{-}, 
(2,8);(12,8) **@{-}, 
(-12,8);(-12,2) **@{-}, 
(12,8);(12,2) **@{-}, 
(-12,-8);(-2,-8) **@{-}, 
(-2,-8);(12,-8) **@{-}, 
(-12,-8);(-12,-2) **@{-}, 
(12,-8);(12,-2) **@{-},
(2,0);(4,-2) **@{-}, 
(-2,0);(-4,-2) **@{-},  
%
%
(2.4,0.4);(4,2) **@{-},
(0.1,-1.9);(1.7,-0.2) **@{-},  
(0.1,-1.9);(-1.7,-0.3) **@{-},  
(-2.4,0.4);(-4,2) **@{-},
(-12,2);(-10.4,0.4) **@{-}, 
(-9.6,-0.4);(-8,-2) **@{-}, 
(-12,-2);(-8,2) **@{-},
(-8,2);(-6.4,0.4) **@{-}, 
(-5.6,-0.4);(-4,-2) **@{-}, 
(-8,-2);(-4,2) **@{-},
(12,2);(10.4,0.4) **@{-}, 
(9.6,-0.4);(8,-2) **@{-},
(12,-2);(8,2) **@{-},
(5.6,-0.4);(4,-2) **@{-}, 
(8,2);(6.6,0.4) **@{-},
(8,-2);(4,2) **@{-},
(-9,0.4)*{\urcorner}, 
(-9,-0.8)*{\lrcorner}, 
(11,0.5)*{\urcorner}, 
(11,-0.8)*{\lrcorner}, 
(-2,4);(-2,8) **\crv{(1,6)}, 
(2,8);(2,4) **\crv{(-1,6)},
(-2,0);(-2,4) **\crv{(1,2)}, 
(2,4);(2,0) **\crv{(-1,2)},
 \endxy\gg_N +\\
 &\hskip 3.6cm yx
 \ll~\xy 
(-12,8);(-2,8) **@{-}, 
(2,8);(12,8) **@{-}, 
(-12,8);(-12,2) **@{-}, 
(12,8);(12,2) **@{-}, 
(-12,-8);(-2,-8) **@{-}, 
(-2,-8);(12,-8) **@{-}, 
(-12,-8);(-12,-2) **@{-}, 
(12,-8);(12,-2) **@{-},
(2,0);(4,-2) **@{-}, 
(-2,0);(-4,-2) **@{-}, 
%
%
(2.4,0.4);(4,2) **@{-},
(0.1,-1.9);(1.7,-0.2) **@{-},  
(0.1,-1.9);(-1.7,-0.3) **@{-},  
(-2.4,0.4);(-4,2) **@{-},
(-12,2);(-10.4,0.4) **@{-}, 
(-9.6,-0.4);(-8,-2) **@{-}, 
(-12,-2);(-8,2) **@{-},
(-8,2);(-6.4,0.4) **@{-}, 
(-5.6,-0.4);(-4,-2) **@{-}, 
(-8,-2);(-4,2) **@{-},
(12,2);(10.4,0.4) **@{-}, 
(9.6,-0.4);(8,-2) **@{-},
(12,-2);(8,2) **@{-},
(5.6,-0.4);(4,-2) **@{-}, 
(8,2);(6.6,0.4) **@{-},
(8,-2);(4,2) **@{-},
(-9,0.4)*{\urcorner}, 
(-9,-0.8)*{\lrcorner}, 
(11,0.5)*{\urcorner}, 
(11,-0.8)*{\lrcorner},  
(-2,8);(2,8) **\crv{(0,5)}, 
(-2,4);(2,4) **\crv{(0,7)},
(-2,4);(2,4) **\crv{(0,1)}, 
(-2,0);(2,0) **\crv{(0,3)},
 \endxy\gg_N + y^2
 \ll~\xy 
(-12,8);(-2,8) **@{-}, 
(2,8);(12,8) **@{-}, 
(-12,8);(-12,2) **@{-}, 
(12,8);(12,2) **@{-}, 
(-12,-8);(-2,-8) **@{-}, 
(-2,-8);(12,-8) **@{-}, 
(-12,-8);(-12,-2) **@{-}, 
(12,-8);(12,-2) **@{-},
(2,0);(4,-2) **@{-}, 
(-2,0);(-4,-2) **@{-},  
%
%
(2.4,0.4);(4,2) **@{-},
(0.1,-1.9);(1.7,-0.2) **@{-},  
(0.1,-1.9);(-1.7,-0.3) **@{-},  
(-2.4,0.4);(-4,2) **@{-},
(-12,2);(-10.4,0.4) **@{-}, 
(-9.6,-0.4);(-8,-2) **@{-}, 
(-12,-2);(-8,2) **@{-},
(-8,2);(-6.4,0.4) **@{-}, 
(-5.6,-0.4);(-4,-2) **@{-}, 
(-8,-2);(-4,2) **@{-},
(12,2);(10.4,0.4) **@{-}, 
(9.6,-0.4);(8,-2) **@{-},
(12,-2);(8,2) **@{-},
(5.6,-0.4);(4,-2) **@{-}, 
(8,2);(6.6,0.4) **@{-},
(8,-2);(4,2) **@{-},
(-9,0.4)*{\urcorner}, 
(-9,-0.8)*{\lrcorner}, 
(11,0.5)*{\urcorner}, 
(11,-0.8)*{\lrcorner}, 
(-2,0);(-2,4) **\crv{(1,2)}, 
(2,4);(2,0) **\crv{(-1,2)},
(-2,8);(2,8) **\crv{(0,5)}, 
(-2,4);(2,4) **\crv{(0,7)},
 \endxy\gg_N\\
 &= \Big\langle
 ~\xy (0,0) *\xycircle(3,3){-},\endxy
 ~\xy (0,0) *\xycircle(3,3){-},\endxy
 ~\Big\rangle_N (x^2+y^2) + xy\Big\langle~\xy 
(-12,8);(-2,8) **@{-}, 
(2,8);(12,8) **@{-}, 
(-12,8);(-12,2) **@{-}, 
(12,8);(12,2) **@{-}, 
(-12,-8);(-2,-8) **@{-}, 
(-2,-8);(12,-8) **@{-}, 
(-12,-8);(-12,-2) **@{-}, 
(12,-8);(12,-2) **@{-},
(2,0);(4,-2) **@{-}, 
(-2,0);(-4,-2) **@{-},  
%
%
(2.4,0.4);(4,2) **@{-},
(0.1,-1.9);(1.7,-0.2) **@{-},  
(0.1,-1.9);(-1.7,-0.3) **@{-},  
(-2.4,0.4);(-4,2) **@{-},
(-12,2);(-10.4,0.4) **@{-}, 
(-9.6,-0.4);(-8,-2) **@{-}, 
(-12,-2);(-8,2) **@{-},
(-8,2);(-6.4,0.4) **@{-}, 
(-5.6,-0.4);(-4,-2) **@{-}, 
(-8,-2);(-4,2) **@{-},
(12,2);(10.4,0.4) **@{-}, 
(9.6,-0.4);(8,-2) **@{-},
(12,-2);(8,2) **@{-},
(5.6,-0.4);(4,-2) **@{-}, 
(8,2);(6.6,0.4) **@{-},
(8,-2);(4,2) **@{-},
(-9,0.4)*{\urcorner}, 
(-9,-0.8)*{\lrcorner}, 
(11,0.5)*{\urcorner}, 
(11,-0.8)*{\lrcorner}, 
(-2,4);(-2,8) **\crv{(1,6)}, 
(2,8);(2,4) **\crv{(-1,6)},
(-2,0);(-2,4) **\crv{(1,2)}, 
(2,4);(2,0) **\crv{(-1,2)},
 \endxy~\Big\rangle_N + yx\Big\langle~\xy
(0,0) *\xycircle(3,3){-}, \endxy~
\xy (0,0) *\xycircle(3,3){-}, \endxy~
\xy (0,0) *\xycircle(3,3){-}, \endxy~\Big\rangle_N\\
&=(-A^{-2}-A^{2})(x^2+y^2)+ (-A^{-2}-A^{2})^2xy+\\
&\hskip 3.3cm (-A^{16}+A^{12}+A^{4})(-A^{-16}+A^{-12}+A^{-4})xy\\
&=(-A^{2}-A^{-2})(x^2+y^2) + (5-A^{12}-A^{-12}+A^8+A^{-8}) xy.
 \end{align*}
This yields 
${\mathbf K}(8_1;t)_N=\ll 8_1\gg_N(z(t), \overline{z}(t), t, t)
=(6t^5+14t^4-8t^3-8t^2+6t-1)t^{-4}$.
Hence for each integer $n \geq 2$,
\begin{align*}
{\mathbf K}_{2n-1}(8_1)_N &=\overline{{\mathbf K}(8_1;n)_N}= [6n^5+14n^4-8n^3-8n^2+6n-1]n^{-4}.
\end{align*}
Further, forgetting the orientation on $8_1$, we see that $\ll 8_1\gg_N$ $=\ll \widetilde{8_1}\gg$ and therefore
${\mathbf K}_{2n-1}(\widetilde{8_1})={\mathbf K}_{2n-1}(8_1)_N$ for all $n\geq 2.$
\end{example}
 
 
\begin{example}\label{examp-9_1-2}
Let $9_1$ be the ribbon $2$-knot associated with the knot $6_1$ in Yoshikawa's table with the orientation indicated below. By Theorem \ref{thm-skein-rel-kb} together with (\ref{nkbp-exmps}) and (\ref{skein-rel-exmp}), we obtain 

\begin{align*}
\ll~&\xy 
(-2,12);(2,8) **@{-}, 
(-2,8);(-0.7,9.5) **@{-},
(0.7,10.6);(2,12) **@{-},
(-12,12);(-2,12) **@{-}, 
(2,12);(12,12) **@{-}, 
(-12,12);(-12,2) **@{-}, 
(12,12);(12,2) **@{-},
(-12,-8);(-2,-8) **@{-}, 
(-2,-8);(12,-8) **@{-}, 
(-12,-8);(-12,-2) **@{-}, 
(12,-8);(12,-2) **@{-},
(-2,8);(2,4) **@{-}, 
(-2,4);(2,8) **@{-}, 
(-2,4);(4,-2) **@{-}, 
(2,4);(-4,-2) **@{-}, 
(2.4,0.4);(4,2) **@{-},
(0.1,-1.9);(1.7,-0.2) **@{-},  
(0.1,-1.9);(-1.7,-0.3) **@{-},  
(-2.4,0.4);(-4,2) **@{-},
(-12,2);(-10.4,0.4) **@{-}, 
(-9.6,-0.4);(-8,-2) **@{-}, 
(-12,-2);(-8,2) **@{-},
(-8,2);(-6.4,0.4) **@{-}, 
(-5.6,-0.4);(-4,-2) **@{-}, 
(-8,-2);(-4,2) **@{-},
(12,2);(10.4,0.4) **@{-}, 
(9.6,-0.4);(8,-2) **@{-},
(12,-2);(8,2) **@{-},
(5.6,-0.4);(4,-2) **@{-}, 
(8,2);(6.6,0.4) **@{-},
(8,-2);(4,2) **@{-},
(-0.1,4.5);(-0.1,7.5) **@{-},
(0,4.5);(0,7.5) **@{-},
(0.1,4.5);(0.1,7.5) **@{-},
(-1.5,1.9);(1.5,1.9) **@{-},
(-1.5,2);(1.5,2) **@{-},
(-1.5,2.1);(1.5,2.1) **@{-},
(-3,-1.2)*{\llcorner}, 
(-11,-1.2)*{\llcorner}, 
(-8.9,-1.2)*{\lrcorner}, 
(11,0.5)*{\urcorner}, 
(8.9,0.5)*{\ulcorner}, 
(0.5,-11.2) *{9_1},
\endxy\gg_N 
= x^2
\ll~\xy 
(-2,12);(2,8) **@{-}, 
(-2,8);(-0.7,9.5) **@{-},
(0.7,10.6);(2,12) **@{-},
(-12,12);(-2,12) **@{-}, 
(2,12);(12,12) **@{-}, 
(-12,12);(-12,2) **@{-}, 
(12,12);(12,2) **@{-},
(-12,-8);(-2,-8) **@{-}, 
(-2,-8);(12,-8) **@{-}, 
(-12,-8);(-12,-2) **@{-}, 
(12,-8);(12,-2) **@{-},
(2,0);(4,-2) **@{-}, 
(-2,0);(-4,-2) **@{-}, 
(2.4,0.4);(4,2) **@{-},
(0.1,-1.9);(1.7,-0.2) **@{-},  
(0.1,-1.9);(-1.7,-0.3) **@{-},  
(-2.4,0.4);(-4,2) **@{-},
(-12,2);(-10.4,0.4) **@{-}, 
(-9.6,-0.4);(-8,-2) **@{-}, 
(-12,-2);(-8,2) **@{-},
(-8,2);(-6.4,0.4) **@{-}, 
(-5.6,-0.4);(-4,-2) **@{-}, 
(-8,-2);(-4,2) **@{-},
(12,2);(10.4,0.4) **@{-}, 
(9.6,-0.4);(8,-2) **@{-},
(12,-2);(8,2) **@{-},
(5.6,-0.4);(4,-2) **@{-}, 
(8,2);(6.6,0.4) **@{-},
(8,-2);(4,2) **@{-},
(-3,-1.2)*{\llcorner}, 
(-11,-1.2)*{\llcorner}, 
(-8.9,-1.2)*{\lrcorner}, 
(11,0.5)*{\urcorner}, 
(8.9,0.5)*{\ulcorner}, 
(-2,4);(-2,8) **\crv{(1,6)}, 
(2,8);(2,4) **\crv{(-1,6)},
(-2,4);(2,4) **\crv{(0,1)}, 
(-2,0);(2,0) **\crv{(0,3)},
 \endxy\gg_N 
 +xy
 \ll~\xy 
(-2,12);(2,8) **@{-}, 
(-2,8);(-0.7,9.5) **@{-},
(0.7,10.6);(2,12) **@{-},
(-12,12);(-2,12) **@{-}, 
(2,12);(12,12) **@{-}, 
(-12,12);(-12,2) **@{-}, 
(12,12);(12,2) **@{-},
(-12,-8);(-2,-8) **@{-}, 
(-2,-8);(12,-8) **@{-}, 
(-12,-8);(-12,-2) **@{-}, 
(12,-8);(12,-2) **@{-},
(2,0);(4,-2) **@{-}, 
(-2,0);(-4,-2) **@{-},  
(2.4,0.4);(4,2) **@{-},
(0.1,-1.9);(1.7,-0.2) **@{-},  
(0.1,-1.9);(-1.7,-0.3) **@{-},  
(-2.4,0.4);(-4,2) **@{-},
(-12,2);(-10.4,0.4) **@{-}, 
(-9.6,-0.4);(-8,-2) **@{-}, 
(-12,-2);(-8,2) **@{-},
(-8,2);(-6.4,0.4) **@{-}, 
(-5.6,-0.4);(-4,-2) **@{-}, 
(-8,-2);(-4,2) **@{-},
(12,2);(10.4,0.4) **@{-}, 
(9.6,-0.4);(8,-2) **@{-},
(12,-2);(8,2) **@{-},
(5.6,-0.4);(4,-2) **@{-}, 
(8,2);(6.6,0.4) **@{-},
(8,-2);(4,2) **@{-},
(-3,-1.2)*{\llcorner}, (-11,-1.2)*{\llcorner}, (-8.9,-1.2)*{\lrcorner}, (11,0.5)*{\urcorner}, (8.9,0.5)*{\ulcorner}, 
(-2,4);(-2,8) **\crv{(1,6)}, 
(2,8);(2,4) **\crv{(-1,6)},
(-2,0);(-2,4) **\crv{(1,2)}, 
(2,4);(2,0) **\crv{(-1,2)},
 \endxy\gg_N \\
 &\hskip 3.1cm + yx
 \ll~\xy 
(-2,12);(2,8) **@{-}, 
(-2,8);(-0.7,9.5) **@{-},
(0.7,10.6);(2,12) **@{-},
(-12,12);(-2,12) **@{-}, 
(2,12);(12,12) **@{-}, 
(-12,12);(-12,2) **@{-}, 
(12,12);(12,2) **@{-},
(-12,-8);(-2,-8) **@{-}, 
(-2,-8);(12,-8) **@{-}, 
(-12,-8);(-12,-2) **@{-}, 
(12,-8);(12,-2) **@{-},
(2,0);(4,-2) **@{-}, 
(-2,0);(-4,-2) **@{-}, 
(2.4,0.4);(4,2) **@{-},
(0.1,-1.9);(1.7,-0.2) **@{-},  
(0.1,-1.9);(-1.7,-0.3) **@{-},  
(-2.4,0.4);(-4,2) **@{-},
(-12,2);(-10.4,0.4) **@{-}, 
(-9.6,-0.4);(-8,-2) **@{-}, 
(-12,-2);(-8,2) **@{-},
(-8,2);(-6.4,0.4) **@{-}, 
(-5.6,-0.4);(-4,-2) **@{-}, 
(-8,-2);(-4,2) **@{-},
(12,2);(10.4,0.4) **@{-}, 
(9.6,-0.4);(8,-2) **@{-},
(12,-2);(8,2) **@{-},
(5.6,-0.4);(4,-2) **@{-}, 
(8,2);(6.6,0.4) **@{-},
(8,-2);(4,2) **@{-},
(-3,-1.2)*{\llcorner}, 
(-11,-1.2)*{\llcorner}, 
(-8.9,-1.2)*{\lrcorner}, 
(11,0.5)*{\urcorner}, 
(8.9,0.5)*{\ulcorner}, 
(-2,8);(2,8) **\crv{(0,5)}, 
(-2,4);(2,4) **\crv{(0,7)},
(-2,4);(2,4) **\crv{(0,1)}, 
(-2,0);(2,0) **\crv{(0,3)},
 \endxy\gg_N + y^2
 \ll~\xy 
(-2,12);(2,8) **@{-}, 
(-2,8);(-0.7,9.5) **@{-},
(0.7,10.6);(2,12) **@{-},
(-12,12);(-2,12) **@{-}, 
(2,12);(12,12) **@{-}, 
(-12,12);(-12,2) **@{-}, 
(12,12);(12,2) **@{-},
(-12,-8);(-2,-8) **@{-}, 
(-2,-8);(12,-8) **@{-}, 
(-12,-8);(-12,-2) **@{-}, 
(12,-8);(12,-2) **@{-},
(2,0);(4,-2) **@{-}, 
(-2,0);(-4,-2) **@{-},  
(2.4,0.4);(4,2) **@{-},
(0.1,-1.9);(1.7,-0.2) **@{-},  
(0.1,-1.9);(-1.7,-0.3) **@{-},  
(-2.4,0.4);(-4,2) **@{-},
(-12,2);(-10.4,0.4) **@{-}, 
(-9.6,-0.4);(-8,-2) **@{-}, 
(-12,-2);(-8,2) **@{-},
(-8,2);(-6.4,0.4) **@{-}, 
(-5.6,-0.4);(-4,-2) **@{-}, 
(-8,-2);(-4,2) **@{-},
(12,2);(10.4,0.4) **@{-}, 
(9.6,-0.4);(8,-2) **@{-},
(12,-2);(8,2) **@{-},
(5.6,-0.4);(4,-2) **@{-}, 
(8,2);(6.6,0.4) **@{-},
(8,-2);(4,2) **@{-},
(-3,-1.2)*{\llcorner}, (-11,-1.2)*{\llcorner}, (-8.9,-1.2)*{\lrcorner}, (11,0.5)*{\urcorner}, (8.9,0.5)*{\ulcorner}, 
(-2,0);(-2,4) **\crv{(1,2)}, 
(2,4);(2,0) **\crv{(-1,2)},
(-2,8);(2,8) **\crv{(0,5)}, 
(-2,4);(2,4) **\crv{(0,7)},
 \endxy\gg_N\\
= & x^2 \langle~\xy (0,0) *\xycircle(3,3){-}, \endxy~
\xy (0,0) *\xycircle(3,3){-}, \endxy~\rangle_N 
+ xy \bigg(A^{-8}
\langle~\xy
(12,2);(10.4,0.4) **@{-}, 
(9.6,-0.4);(8,-2) **@{-},
(12,-2);(8,2) **@{-},
(5.6,-0.4);(4,-2) **@{-}, 
(8,2);(6.6,0.4) **@{-},
(8,-2);(4,2) **@{-},
(11,0.5)*{\urcorner}, (8.9,0.5)*{\ulcorner}, 
(4,-2);(-2,-2) **\crv{(1,-5)},
(4,4);(2,-2) **\crv{(0,1)},
(-2,2);(1.5,4) **\crv{(-1,4.5)},
(-2,-2);(-2,2) **\crv{(-3,0)},
(4,-4);(12,-2) **\crv{(10,-6)},
(4,4);(12,2) **\crv{(10,6)},
\endxy~
\rangle_N 
+ (A^{-6}-A^{-2})\langle~
\xy
(12,2);(10.4,0.4) **@{-}, 
(9.6,-0.4);(8,-2) **@{-},
(12,-2);(8,2) **@{-},
(5.6,-0.4);(4,-2) **@{-}, 
(8,2);(6.6,0.4) **@{-},
(8,-2);(4,2) **@{-},
(4,2);(12,2) **\crv{(8,6)},
(4,-2);(12,-2) **\crv{(8,-6)},
(11,0.5)*{\urcorner}, (8.9,0.5)*{\ulcorner}, 
\endxy~\rangle_N\bigg)\\
&  + yx \langle~\xy
(0,0) *\xycircle(3,3){-}, \endxy~
\xy (0,0) *\xycircle(3,3){-}, \endxy~
\xy (0,0) *\xycircle(3,3){-}, \endxy ~\rangle_N
+ y^2 \langle~\xy (0,0) *\xycircle(3,3){-}, \endxy~
\xy (0,0) *\xycircle(3,3){-}, \endxy~\rangle_N\\
= &  (-A^{-2}-A^{2})(x^2+y^2)+(-A^{-2}-A^{2})^2 xy \\
& + \Big(A^{-8}(A^{-8}-A^{-4}+1-A^4+A^8) +(A^{-6}-A^{-2})(-A^{10}-A^{2})\Big) xy\\
=& (-A^{-2}-A^{2})(x^2+y^2)+(4-A^{-4}+A^{-16}-A^{-12}+A^{-8}+A^8)xy. 
 \end{align*}
This gives 
${\mathbf K}(9_1;t)_N =\ll 9_1\gg_N(z(t), \overline{z}(t), t, t)=(6t^7+27t^6+12t^5-37t^4-2t^3+19t^2-8t+1)2^{-1}t^{-6}
+(6t^6+7t^5-21t^4+14t^2-7t+1)2^{-1}t^{-6}\mathbf{i}\sqrt{(t+1)(3t-1)}.$
Hence for each integer $n \geq 2$,
\begin{align*}
{\mathbf K}_{2n-1}(9_1)_N &=\overline{{\mathbf K}(9_1;n)_N}\\
&=[6n^7+27n^6+12n^5-37n^4-2n^3+19n^2-8n+1]2^{-1}n^{-6}\\
&+[6n^6+7n^5-21n^4+14n^2-7n+1]2^{-1}n^{-6}\mathbf{i}\sqrt{(n+1)(3n-1)} .
\end{align*}
Further, forgetting the orientation on $9_1$, we see that $\ll 9_1\gg_N$ $=\ll \widetilde{9_1}\gg$ and therefore
${\mathbf K}_{2n-1}(\widetilde{9_1})={\mathbf K}_{2n-1}(9_1)_N$ for all $n \geq 2.$
\end{example}

\begin{example}\label{examp-10_2-1} 
Let $10_2$ be the $2$-twist spun $2$-knot of the trefoil in Yoshikawa's table with the orientation indicated below. By Theorem \ref{thm-skein-rel-kb} together with (\ref{nkbp-exmps}) and (\ref{skein-rel-exmp}), we obtain 

\begin{align*}
\ll~&\xy (-12,8);(-2,8) **@{-}, (2,8);(12,8) **@{-}, (-12,8);(-12,2) **@{-}, 
(12,8);(12,2) **@{-}, (-12,-8);(-2,-8) **@{-}, (2,-8);(12,-8) **@{-}, 
(-12,-8);(-12,-2) **@{-}, (12,-8);(12,-2) **@{-},
(-2,8);(2,4) **@{-}, (-2,4);(-0.4,5.4) **@{-},
(0.4,6.4);(2,8) **@{-}, (-2,4);(4,-2) **@{-}, (2,4);(-4,-2) **@{-}, (-2,-8);(2,-4) **@{-}, (-2,-4);(2,-8) **@{-},
(-2,-4);(1.4,-0.4) **@{-}, (2.4,0.4);(4,2) **@{-},
(2,-4);(0.6,-2.4) **@{-}, (-0.4,-1.6);(-1.7,-0.3) **@{-},  (-2.4,0.4);(-4,2) **@{-},
(-12,2);(-10.4,0.4) **@{-}, (-9.6,-0.4);(-8,-2) **@{-}, (-12,-2);(-8,2) **@{-},
(-8,2);(-6.4,0.4) **@{-}, (-5.6,-0.4);(-4,-2) **@{-}, (-8,-2);(-4,2) **@{-},
(12,2);(10.4,0.4) **@{-}, (9.6,-0.4);(8,-2) **@{-},
(12,-2);(8,2) **@{-},
(5.6,-0.4);(4,-2) **@{-}, (8,2);(6.6,0.4) **@{-},
(8,-2);(4,2) **@{-},
(-1.5,1.9);(1.5,1.9) **@{-},
(-1.5,2);(1.5,2) **@{-},
(-1.5,2.1);(1.5,2.1) **@{-},
(-1.5,-5.9);(1.5,-5.9) **@{-},
(-1.5,-6);(1.5,-6) **@{-},
(-1.5,-6.1);(1.5,-6.1) **@{-},
(-1,6.4)*{\ulcorner}, (-1.4,-8)*{\urcorner}, (-3,-1.2)*{\llcorner}, (-11,-1.2)*{\llcorner}, (-8.9,-1.2)*{\lrcorner}, (11,0.5)*{\urcorner}, (8.9,0.5)*{\ulcorner}, (1,-3)*{\lrcorner},
(0.5,-11.2) *{10_2},
 \endxy\gg_N 
 = x^2 \ll~\xy (-12,8);(-2,8) **@{-}, (2,8);(12,8) **@{-}, (-12,8);(-12,2) **@{-}, (12,8);(12,2) **@{-}, (-12,-8);(-2,-8) **@{-}, (2,-8);(12,-8) **@{-}, (-12,-8);(-12,-2) **@{-}, (12,-8);(12,-2) **@{-},
(-2,8);(2,4) **@{-}, (-2,4);(-0.4,5.4) **@{-},
(0.4,6.4);(2,8) **@{-}, (2,0);(4,-2) **@{-}, (-2,0);(-4,-2) **@{-},
(-2,-4);(1.4,-0.4) **@{-}, (2.4,0.4);(4,2) **@{-},
(2,-4);(0.6,-2.4) **@{-}, (-0.4,-1.6);(-1.7,-0.3) **@{-},  (-2.4,0.4);(-4,2) **@{-},
(-12,2);(-10.4,0.4) **@{-}, (-9.6,-0.4);(-8,-2) **@{-}, (-12,-2);(-8,2) **@{-},
(-8,2);(-6.4,0.4) **@{-}, (-5.6,-0.4);(-4,-2) **@{-}, (-8,-2);(-4,2) **@{-},
(12,2);(10.4,0.4) **@{-}, (9.6,-0.4);(8,-2) **@{-},
(12,-2);(8,2) **@{-},
(5.6,-0.4);(4,-2) **@{-}, (8,2);(6.6,0.4) **@{-},
(8,-2);(4,2) **@{-},
(-1,6.4)*{\ulcorner}, (-1.4,-8)*{\urcorner}, (-3,-1.2)*{\llcorner}, (-11,-1.2)*{\llcorner}, (-8.9,-1.2)*{\lrcorner}, (11,0.5)*{\urcorner}, (8.9,0.5)*{\ulcorner}, (1,-3)*{\lrcorner},
(-2,4);(2,4) **\crv{(0,1)}, (-2,0);(2,0) **\crv{(0,3)},
(-2,-4);(2,-4) **\crv{(0,-7)}, (-2,-8);(2,-8) **\crv{(0,-5)},
 \endxy~\gg_N
 + xy \ll~\xy (-12,8);(-2,8) **@{-}, (2,8);(12,8) **@{-}, (-12,8);(-12,2) **@{-}, (12,8);(12,2) **@{-}, (-12,-8);(-2,-8) **@{-}, (2,-8);(12,-8) **@{-}, (-12,-8);(-12,-2) **@{-}, (12,-8);(12,-2) **@{-},
(-2,8);(2,4) **@{-}, (-2,4);(-0.4,5.4) **@{-},
(0.4,6.4);(2,8) **@{-}, (2,0);(4,-2) **@{-}, (-2,0);(-4,-2) **@{-},
(-2,-4);(1.4,-0.4) **@{-}, (2.4,0.4);(4,2) **@{-},
(2,-4);(0.6,-2.4) **@{-}, (-0.4,-1.6);(-1.7,-0.3) **@{-},  (-2.4,0.4);(-4,2) **@{-},
(-12,2);(-10.4,0.4) **@{-}, (-9.6,-0.4);(-8,-2) **@{-}, (-12,-2);(-8,2) **@{-},
(-8,2);(-6.4,0.4) **@{-}, (-5.6,-0.4);(-4,-2) **@{-}, (-8,-2);(-4,2) **@{-},
(12,2);(10.4,0.4) **@{-}, (9.6,-0.4);(8,-2) **@{-},
(12,-2);(8,2) **@{-},
(5.6,-0.4);(4,-2) **@{-}, (8,2);(6.6,0.4) **@{-},
(8,-2);(4,2) **@{-},
(-1,6.4)*{\ulcorner}, (-1.4,-8)*{\urcorner}, (-3,-1.2)*{\llcorner}, (-11,-1.2)*{\llcorner}, (-8.9,-1.2)*{\lrcorner}, (11,0.5)*{\urcorner}, (8.9,0.5)*{\ulcorner}, (1,-3)*{\lrcorner},
(-2,4);(2,4) **\crv{(0,1)}, (-2,0);(2,0) **\crv{(0,3)},
(-2,-4);(-2,-8) **\crv{(1,-6)}, (2,-8);(2,-4) **\crv{(-1,-6)},
 \endxy~\gg_N\\
&\hskip 3.20cm +yx \ll\xy (-12,8);(-2,8) **@{-}, (2,8);(12,8) **@{-}, (-12,8);(-12,2) **@{-}, (12,8);(12,2) **@{-}, (-12,-8);(-2,-8) **@{-}, (2,-8);(12,-8) **@{-}, (-12,-8);(-12,-2) **@{-}, (12,-8);(12,-2) **@{-},
(-2,8);(2,4) **@{-}, (-2,4);(-0.4,5.4) **@{-},
(0.4,6.4);(2,8) **@{-}, (2,0);(4,-2) **@{-}, (-2,0);(-4,-2) **@{-},
(-2,-4);(1.4,-0.4) **@{-}, (2.4,0.4);(4,2) **@{-},
(2,-4);(0.6,-2.4) **@{-}, (-0.4,-1.6);(-1.7,-0.3) **@{-},  (-2.4,0.4);(-4,2) **@{-},
(-12,2);(-10.4,0.4) **@{-}, (-9.6,-0.4);(-8,-2) **@{-}, (-12,-2);(-8,2) **@{-},
(-8,2);(-6.4,0.4) **@{-}, (-5.6,-0.4);(-4,-2) **@{-}, (-8,-2);(-4,2) **@{-},
(12,2);(10.4,0.4) **@{-}, (9.6,-0.4);(8,-2) **@{-},
(12,-2);(8,2) **@{-},
(5.6,-0.4);(4,-2) **@{-}, (8,2);(6.6,0.4) **@{-},
(8,-2);(4,2) **@{-},
(-1,6.4)*{\ulcorner}, (-1.4,-8)*{\urcorner}, (-3,-1.2)*{\llcorner}, (-11,-1.2)*{\llcorner}, (-8.9,-1.2)*{\lrcorner}, (11,0.5)*{\urcorner}, (8.9,0.5)*{\ulcorner}, (1,-3)*{\lrcorner},
(-2,4);(-2,0) **\crv{(1,2)}, (2,0);(2,4) **\crv{(-1,2)},
(-2,-4);(2,-4) **\crv{(0,-7)}, (-2,-8);(2,-8) **\crv{(0,-5)},
 \endxy\gg_N
 ~+y^2 \ll\xy (-12,8);(-2,8) **@{-}, (2,8);(12,8) **@{-}, (-12,8);(-12,2) **@{-}, (12,8);(12,2) **@{-}, (-12,-8);(-2,-8) **@{-}, (2,-8);(12,-8) **@{-}, (-12,-8);(-12,-2) **@{-}, (12,-8);(12,-2) **@{-},
(-2,8);(2,4) **@{-}, (-2,4);(-0.4,5.4) **@{-},
(0.4,6.4);(2,8) **@{-}, (2,0);(4,-2) **@{-}, (-2,0);(-4,-2) **@{-},
(-2,-4);(1.4,-0.4) **@{-}, (2.4,0.4);(4,2) **@{-},
(2,-4);(0.6,-2.4) **@{-}, (-0.4,-1.6);(-1.7,-0.3) **@{-},  (-2.4,0.4);(-4,2) **@{-},
(-12,2);(-10.4,0.4) **@{-}, (-9.6,-0.4);(-8,-2) **@{-}, (-12,-2);(-8,2) **@{-},
(-8,2);(-6.4,0.4) **@{-}, (-5.6,-0.4);(-4,-2) **@{-}, (-8,-2);(-4,2) **@{-},
(12,2);(10.4,0.4) **@{-}, (9.6,-0.4);(8,-2) **@{-},
(12,-2);(8,2) **@{-},
(5.6,-0.4);(4,-2) **@{-}, (8,2);(6.6,0.4) **@{-},
(8,-2);(4,2) **@{-},
(-1,6.4)*{\ulcorner}, (-1.4,-8)*{\urcorner}, (-3,-1.2)*{\llcorner}, (-11,-1.2)*{\llcorner}, (-8.9,-1.2)*{\lrcorner}, (11,0.5)*{\urcorner}, (8.9,0.5)*{\ulcorner}, (1,-3)*{\lrcorner},
(-2,4);(-2,0) **\crv{(1,2)}, (2,0);(2,4) **\crv{(-1,2)},
(-2,-4);(-2,-8) **\crv{(1,-6)}, (2,-8);(2,-4) **\crv{(-1,-6)},
 \endxy~\gg_N
 \end{align*}
 \begin{align*}
= & x^2 \langle~\xy (0,0) *\xycircle(3,3){-}, \endxy~
\xy (0,0) *\xycircle(3,3){-}, \endxy ~\rangle_N
+ xy \Big(A^{-8}\langle~\xy (0,0) *\xycircle(3,3){-}, 
\endxy~~\rangle_N + (A^{-6}-A^{-2}) \langle~ \xy
(12,2);(10.4,0.4) **@{-}, 
(9.6,-0.4);(8,-2) **@{-},
(12,-2);(8,2) **@{-},
(5.6,-0.4);(4,-2) **@{-}, 
(8,2);(6.6,0.4) **@{-},
(8,-2);(4,2) **@{-},
(11,0.5)*{\urcorner}, (8.9,0.5)*{\ulcorner}, 
(4,2);(2.4,0.4) **@{-}, 
(1.6,-0.4);(0,-2) **@{-},
(4,-2);(0,2) **@{-},
(-3.6,-0.4);(-4,-2) **@{-}, 
(0,2);(-1.6,0.4) **@{-},
(0,-2);(-4,2) **@{-},
(-2.2,-0.4);(-4,-2) **@{-},
(-4,2);(12,2) **\crv{(5,6)},
(-4,-2);(12,-2) **\crv{(5,-6)},
\endxy~\rangle_N\Big)\\
&+ xy\Big(A^{-8}
\langle~\xy
(12,2);(10.4,0.4) **@{-}, 
(9.6,-0.4);(8,-2) **@{-},
(12,-2);(8,2) **@{-},
(5.6,-0.4);(4,-2) **@{-}, 
(8,2);(6.6,0.4) **@{-},
(8,-2);(4,2) **@{-},
(11,0.5)*{\urcorner}, (8.9,0.5)*{\ulcorner}, 
(4,-2);(-2,-2) **\crv{(1,-5)},
(4,4);(2,-2) **\crv{(0,1)},
(-2,2);(1.5,4) **\crv{(-1,4.5)},
(-2,-2);(-2,2) **\crv{(-3,0)},
(4,-4);(12,-2) **\crv{(10,-6)},
(4,4);(12,2) **\crv{(10,6)},
\endxy~
\rangle_N
+(A^{-6}-A^{-2}) \langle~
\xy
(12,2);(10.4,0.4) **@{-}, 
(9.6,-0.4);(8,-2) **@{-},
(12,-2);(8,2) **@{-},
(5.6,-0.4);(4,-2) **@{-}, 
(8,2);(6.6,0.4) **@{-},
(8,-2);(4,2) **@{-},
(4,2);(12,2) **\crv{(8,6)},
(4,-2);(12,-2) **\crv{(8,-6)},
(11,0.5)*{\urcorner}, (8.9,0.5)*{\ulcorner}, 
\endxy~
\rangle_N\Big)\\
& + y^2 \Big(A^{-8} \langle~ \xy
(12,2);(10.4,0.4) **@{-}, 
(9.6,-0.4);(8,-2) **@{-},
(12,-2);(8,2) **@{-},
(5.6,-0.4);(4,-2) **@{-}, 
(8,2);(6.6,0.4) **@{-},
(8,-2);(4,2) **@{-},
(4,2);(12,2) **\crv{(8,6)},
(4,-2);(12,-2) **\crv{(8,-6)},
(11,0.5)*{\urcorner}, (8.9,0.5)*{\ulcorner}, 
\endxy~\rangle_N
+ (A^{-6}-A^{-2})\langle~\xy (0,0) *\xycircle(3,3){-}, 
\endxy~ ~\rangle_N \Big)\\
=&x^2(-A^{-2}-A^2) +xy \Big(A^{-8} + 
(A^{-6}-A^{-2})(-A^{18}-A^{10}+A^{6}-A^{2}) \Big) \\
& + xy \Big( A^{-8}(A^{-8}-A^{-4}+1-A^4+A^8)   
+(A^{-6}-A^{-2})(-A^{10}-A^{2}) \Big) \\
& + y^2\Big(A^{-8}(-A^{10}-A^{2})+(A^{-6}-A^{-2})\Big)\\
=& (-A^{-2}-A^2)(x^2+y^2)+\\
&~(A^{-16}-A^{-12}+2A^{-8}-3A^{-4}+4-3A^4+2A^8-A^{12}+A^{16})xy. 
\end{align*}
This gives ${\mathbf K}(10_2;t)_N=\ll 10_2\gg_N(z(t), \overline{z}(t), t, t)
=(8t^7+25t^6+12t^5-37t^4-2t^3+19t^2-8t+1)t^{-6}$.
Hence for each integer $n \geq 2$,
\begin{align*}
{\mathbf K}_{2n-1}(10_2)_N &=\overline{{\mathbf K}(10_2;n)_N}=[8n^7+25n^6+12n^5-37n^4-2n^3+19n^2-8n+1]n^{-6}.
\end{align*}
Further, forgetting the orientation on $10_2$, we see that $\ll 10_2\gg_N$ $=\ll \widetilde{10_2}\gg$ and therefore
${\mathbf K}_{2n-1}(\widetilde{10_2})={\mathbf K}_{2n-1}(10_2)_N$ for all $n \geq 2.$
\end{example}

\begin{example}\label{examp-8^{-1,-1}_1} 
Consider the two component nonorientable surface-link $8^{-1,-1}_1$ in Yoshikawa's table. Observe that $sw(8^{-1,-1}_1)=0$. By (\ref{defn-poly-inv-kb}), (\ref{skein-rel-exmp}) and Definition \ref{defn-sp=poly-2}, we obtain 
\begin{align*}
&\ll~\xy  (-8,8);(-2,8) **@{-}, 
 (2,8);(8,8) **@{-}, 
 (-8,8);(-8,2) **@{-}, 
 (8,8);(8,2) **@{-}, 
 (-8,-8);(-2,-8) **@{-}, 
 (2,-8);(8,-8) **@{-}, 
 (-8,-8);(-8,-2) **@{-}, 
 (8,-8);(8,-2) **@{-},%
(-2,8);(2,4) **@{-}, (-2,4);(-0.4,5.4) **@{-},
(0.4,6.4);(2,8) **@{-}, (-2,4);(4,-2) **@{-}, (2,4);(-4,-2) **@{-}, (-2,-8);(2,-4) **@{-}, (-2,-4);(2,-8) **@{-},
(-2,-4);(1.4,-0.4) **@{-}, (2.4,0.4);(4,2) **@{-},
(2,-4);(0.6,-2.4) **@{-}, (-0.4,-1.6);(-1.7,-0.3) **@{-},  (-2.4,0.4);(-4,2) **@{-},
(-8,2);(-6.4,0.4) **@{-}, (-5.6,-0.4);(-4,-2) **@{-}, (-8,-2);(-4,2) **@{-},
(5.6,-0.4);(4,-2) **@{-}, (8,2);(6.6,0.4) **@{-},
(8,-2);(4,2) **@{-},
(-1.5,1.9);(1.5,1.9) **@{-},
(-1.5,2);(1.5,2) **@{-},
(-1.5,2.1);(1.5,2.1) **@{-},
(-1.5,-5.9);(1.5,-5.9) **@{-},
(-1.5,-6);(1.5,-6) **@{-},
(-1.5,-6.1);(1.5,-6.1) **@{-},
(0.5,-11.5) *{8^{-1,-1}_1},
 \endxy~\gg
 =[[~\xy  (-8,8);(-2,8) **@{-}, 
 (2,8);(8,8) **@{-}, 
 (-8,8);(-8,2) **@{-}, 
 (8,8);(8,2) **@{-}, 
 (-8,-8);(-2,-8) **@{-}, 
 (2,-8);(8,-8) **@{-}, 
 (-8,-8);(-8,-2) **@{-}, 
 (8,-8);(8,-2) **@{-},%
(-2,8);(2,4) **@{-}, (-2,4);(-0.4,5.4) **@{-},
(0.4,6.4);(2,8) **@{-}, (-2,4);(4,-2) **@{-}, (2,4);(-4,-2) **@{-}, (-2,-8);(2,-4) **@{-}, (-2,-4);(2,-8) **@{-},
(-2,-4);(1.4,-0.4) **@{-}, (2.4,0.4);(4,2) **@{-},
(2,-4);(0.6,-2.4) **@{-}, (-0.4,-1.6);(-1.7,-0.3) **@{-},  (-2.4,0.4);(-4,2) **@{-},
(-8,2);(-6.4,0.4) **@{-}, (-5.6,-0.4);(-4,-2) **@{-}, (-8,-2);(-4,2) **@{-},
(5.6,-0.4);(4,-2) **@{-}, (8,2);(6.6,0.4) **@{-},
(8,-2);(4,2) **@{-},
(-1.5,1.9);(1.5,1.9) **@{-},
(-1.5,2);(1.5,2) **@{-},
(-1.5,2.1);(1.5,2.1) **@{-},
(-1.5,-5.9);(1.5,-5.9) **@{-},
(-1.5,-6);(1.5,-6) **@{-},
(-1.5,-6.1);(1.5,-6.1) **@{-},
 \endxy~]]
 = x^2 [[~\xy  (-8,8);(-2,8) **@{-}, 
 (2,8);(8,8) **@{-}, 
 (-8,8);(-8,2) **@{-}, 
 (8,8);(8,2) **@{-}, 
 (-8,-8);(-2,-8) **@{-}, 
 (2,-8);(8,-8) **@{-}, 
 (-8,-8);(-8,-2) **@{-}, 
 (8,-8);(8,-2) **@{-},%
(-2,8);(2,4) **@{-}, (-2,4);(-0.4,5.4) **@{-},
(0.4,6.4);(2,8) **@{-}, (2,0);(4,-2) **@{-}, (-2,0);(-4,-2) **@{-},
(-2,-4);(1.4,-0.4) **@{-}, (2.4,0.4);(4,2) **@{-},
(2,-4);(0.6,-2.4) **@{-}, (-0.4,-1.6);(-1.7,-0.3) **@{-},  (-2.4,0.4);(-4,2) **@{-},
(-8,2);(-6.4,0.4) **@{-}, (-5.6,-0.4);(-4,-2) **@{-}, (-8,-2);(-4,2) **@{-},
(5.6,-0.4);(4,-2) **@{-}, (8,2);(6.6,0.4) **@{-},
(8,-2);(4,2) **@{-},
(-2,4);(2,4) **\crv{(0,1)}, (-2,0);(2,0) **\crv{(0,3)},
(-2,-4);(2,-4) **\crv{(0,-7)}, (-2,-8);(2,-8) **\crv{(0,-5)},
 \endxy~]]+\\
&xy [[~\xy 
 (-8,8);(-2,8) **@{-}, 
 (2,8);(8,8) **@{-}, 
 (-8,8);(-8,2) **@{-}, 
 (8,8);(8,2) **@{-}, 
 (-8,-8);(-2,-8) **@{-}, 
 (2,-8);(8,-8) **@{-}, 
 (-8,-8);(-8,-2) **@{-}, 
 (8,-8);(8,-2) **@{-},%
(-2,8);(2,4) **@{-}, (-2,4);(-0.4,5.4) **@{-},
(0.4,6.4);(2,8) **@{-}, (2,0);(4,-2) **@{-}, (-2,0);(-4,-2) **@{-},
(-2,-4);(1.4,-0.4) **@{-}, (2.4,0.4);(4,2) **@{-},
(2,-4);(0.6,-2.4) **@{-}, (-0.4,-1.6);(-1.7,-0.3) **@{-},  (-2.4,0.4);(-4,2) **@{-},
(-8,2);(-6.4,0.4) **@{-}, (-5.6,-0.4);(-4,-2) **@{-}, (-8,-2);(-4,2) **@{-},
(5.6,-0.4);(4,-2) **@{-}, (8,2);(6.6,0.4) **@{-},
(8,-2);(4,2) **@{-},
(-2,4);(2,4) **\crv{(0,1)}, (-2,0);(2,0) **\crv{(0,3)},
(-2,-4);(-2,-8) **\crv{(1,-6)}, (2,-8);(2,-4) **\crv{(-1,-6)},
 \endxy~]] +yx [[~\xy  
(-8,8);(-2,8) **@{-}, 
 (2,8);(8,8) **@{-}, 
 (-8,8);(-8,2) **@{-}, 
 (8,8);(8,2) **@{-}, 
 (-8,-8);(-2,-8) **@{-}, 
 (2,-8);(8,-8) **@{-}, 
 (-8,-8);(-8,-2) **@{-}, 
 (8,-8);(8,-2) **@{-},%
(-2,8);(2,4) **@{-}, (-2,4);(-0.4,5.4) **@{-},
(0.4,6.4);(2,8) **@{-}, (2,0);(4,-2) **@{-}, (-2,0);(-4,-2) **@{-},
(-2,-4);(1.4,-0.4) **@{-}, (2.4,0.4);(4,2) **@{-},
(2,-4);(0.6,-2.4) **@{-}, (-0.4,-1.6);(-1.7,-0.3) **@{-},  (-2.4,0.4);(-4,2) **@{-},
(-8,2);(-6.4,0.4) **@{-}, (-5.6,-0.4);(-4,-2) **@{-}, (-8,-2);(-4,2) **@{-},
(5.6,-0.4);(4,-2) **@{-}, (8,2);(6.6,0.4) **@{-},
(8,-2);(4,2) **@{-},
(-2,4);(-2,0) **\crv{(1,2)}, (2,0);(2,4) **\crv{(-1,2)},
(-2,-4);(2,-4) **\crv{(0,-7)}, (-2,-8);(2,-8) **\crv{(0,-5)},
 \endxy~]]
 ~+y^2 [[~\xy  
 (-8,8);(-2,8) **@{-}, 
 (2,8);(8,8) **@{-}, 
 (-8,8);(-8,2) **@{-}, 
 (8,8);(8,2) **@{-}, 
 (-8,-8);(-2,-8) **@{-}, 
 (2,-8);(8,-8) **@{-}, 
 (-8,-8);(-8,-2) **@{-}, 
 (8,-8);(8,-2) **@{-},%
(-2,8);(2,4) **@{-}, (-2,4);(-0.4,5.4) **@{-},
(0.4,6.4);(2,8) **@{-}, (2,0);(4,-2) **@{-}, (-2,0);(-4,-2) **@{-},
(-2,-4);(1.4,-0.4) **@{-}, (2.4,0.4);(4,2) **@{-},
(2,-4);(0.6,-2.4) **@{-}, (-0.4,-1.6);(-1.7,-0.3) **@{-},  (-2.4,0.4);(-4,2) **@{-},
(-8,2);(-6.4,0.4) **@{-}, (-5.6,-0.4);(-4,-2) **@{-}, (-8,-2);(-4,2) **@{-},
(5.6,-0.4);(4,-2) **@{-}, (8,2);(6.6,0.4) **@{-},
(8,-2);(4,2) **@{-},
(-2,4);(-2,0) **\crv{(1,2)}, (2,0);(2,4) **\crv{(-1,2)},
(-2,-4);(-2,-8) **\crv{(1,-6)}, (2,-8);(2,-4) **\crv{(-1,-6)},
 \endxy~]]\\
&= x^2 \langle~\xy (0,0) *\xycircle(3,3){-}, \endxy~
\xy (0,0) *\xycircle(3,3){-}, \endxy ~\rangle
+ xy \Big(A^{-1}(-A^{-3})\langle~\xy (0,0) *\xycircle(3,3){-}, 
\endxy~\rangle + A(-A^{3})^2\langle~\widetilde{3}_1~\rangle\Big)\\
&+ xy\Big(A(-A^{-3})\langle~\xy (0,0) *\xycircle(3,3){-}, \endxy~\xy (0,0) *\xycircle(3,3){-}, \endxy~\rangle
+A^{-1}(-A^{-3}) \langle~
\xy
(12,2);(8,-2) **@{-}, 
(5.6,0.4);(4,2) **@{-},
(9.6,0.4);(8,2) **@{-},
(10.2,-0.4);(12,-2) **@{-}, 
(8,-2);(6.2,-0.4) **@{-},
(8,2);(4,-2) **@{-},
(4,2);(12,2) **\crv{(8,6)},
(4,-2);(12,-2) **\crv{(8,-6)},
\endxy~\sharp~\xy
(12,2);(10.4,0.4) **@{-}, 
(9.6,-0.4);(8,-2) **@{-},
(12,-2);(8,2) **@{-},
(5.6,-0.4);(4,-2) **@{-}, 
(8,2);(6.6,0.4) **@{-},
(8,-2);(4,2) **@{-},
(4,2);(12,2) **\crv{(8,6)},
(4,-2);(12,-2) **\crv{(8,-6)},
\endxy~
\rangle\Big)\\
& + y^2 \Big(A(-A^{-3}) \langle~\xy (0,0) *\xycircle(3,3){-},\endxy~\rangle
+ A^{-1}(-A^3)\langle~\xy (0,0) *\xycircle(3,3){-},\endxy~\rangle\Big)\\
&=x^2(-A^{-2}-A^2) +xy \Big(-A^{-4} + 
A^7(-A^5-A^{-3}+A^{-7}) \Big) \\
& + xy \Big(-A^{-2}(-A^{-2}-A^{2})   
+(-A^{-4})(-A^4-A^{-4})^2\Big)+ y^2(-A^{-2}-A^{2})\\
&= (-A^{-2}-A^2)(x^2+y^2)+(-A^{-12}-2A^{-4}+2-2A^4-A^{12})xy. 
\end{align*}
This gives ${\mathbf K}(8^{-1,-1}_1;t)=\ll 8^{-1,-1}_1\gg(z(t), \overline{z}(t), t, t)
=(6t^5+10t^4-4t^3-9t^2+6t-1)t^{-4}.$
Hence for each integer $n \geq 2$,
\begin{align*}
{\mathbf K}_{2n-1}(8^{-1,-1}_1) &=\overline{{\mathbf K}(8^{-1,-1}_1;n)}=
[6n^5+10n^4-4n^3-9n^2+6n-1]n^{-4}.
\end{align*}
\end{example}

\begin{example}\label{examp-8^{1,1}_1} 
Consider the two torus component surface-link $8^{1,1}_1$ in Yoshikawa's table with the orientation indicated below. By Theorem \ref{thm-skein-rel-kb-1} together with (\ref{nkbp-exmps}) and (\ref{skein-rel-exmp}), we obtain 
\begin{align*}
&{\mathbf K}\bigg(~\xy 
(-8,8);(-2,8) **@{-}, 
(2,8);(8,8) **@{-}, 
(-8,8);(-8,2) **@{-}, 
(8,8);(8,2) **@{-}, 
(-8,-8);(-2,-8) **@{-}, 
(2,-8);(8,-8) **@{-}, 
(-8,-8);(-8,-2) **@{-}, 
(8,-8);(8,-2) **@{-},
(-2,8);(2,4) **@{-}, 
(-2,4);(2,8) **@{-},%
(-2,4);(4,-2) **@{-}, 
(2,4);(-4,-2) **@{-}, 
(-2,-8);(2,-4) **@{-}, 
(-2,-4);(2,-8) **@{-},
(-2,-4);(1.4,-0.4) **@{-}, 
(2.4,0.4);(4,2) **@{-},
(2,-4);(-1.7,-0.3) **@{-}, %
(-2.4,0.4);(-4,2) **@{-},
(-8,2);(-6.4,0.4) **@{-}, 
(-5.6,-0.4);(-4,-2) **@{-}, 
(-8,-2);(-4,2) **@{-},
(5.6,-0.4);(4,-2) **@{-}, 
(8,2);(6.6,0.4) **@{-},
(8,-2);(4,2) **@{-},
(-0.1,4.5);(-0.1,7.5) **@{-},
(0,4.5);(0,7.5) **@{-},
(0.1,4.5);(0.1,7.5) **@{-},
(-0.1,-3.5);(-0.1,-0.5) **@{-},
(0,-3.5);(0,-0.5) **@{-},
(0.1,-3.5);(0.1,-0.5) **@{-},
(-1.5,1.9);(1.5,1.9) **@{-},
(-1.5,2);(1.5,2) **@{-},
(-1.5,2.1);(1.5,2.1) **@{-},
(-1.5,-5.9);(1.5,-5.9) **@{-},
(-1.5,-6);(1.5,-6) **@{-},
(-1.5,-6.1);(1.5,-6.1) **@{-},
(-1,6.4)*{\ulcorner}, (-1.4,-8)*{\urcorner}, (-3,-1.2)*{\llcorner}, 
(1,-3)*{\lrcorner},
(0.5,-12.2) *{8^{1,1}_1},
 \endxy~\bigg)_N 
 = t^2 {\mathbf K}\bigg(~\xy 
 (-8,8);(-2,8) **@{-}, 
(2,8);(8,8) **@{-}, 
(-8,8);(-8,2) **@{-}, 
(8,8);(8,2) **@{-}, 
(-8,-8);(-2,-8) **@{-}, 
(2,-8);(8,-8) **@{-}, 
(-8,-8);(-8,-2) **@{-}, 
(8,-8);(8,-2) **@{-},
(-2,8);(2,4) **@{-}, 
(-2,4);(2,8) **@{-},%
(-2,8);(2,4) **@{-}, (-2,4);(2,8) **@{-},
(0.4,6.4);(2,8) **@{-}, (2,0);(4,-2) **@{-}, (-2,0);(-4,-2) **@{-},
(-2,-4);(1.4,-0.4) **@{-}, (2.4,0.4);(4,2) **@{-},
(2,-4);(-1.7,-0.3) **@{-},  
(-0.4,-1.6);(-1.7,-0.3) **@{-},  (-2.4,0.4);(-4,2) **@{-},
(-8,2);(-6.4,0.4) **@{-}, (-5.6,-0.4);(-4,-2) **@{-}, (-8,-2);(-4,2) **@{-},
(5.6,-0.4);(4,-2) **@{-}, (8,2);(6.6,0.4) **@{-},
(8,-2);(4,2) **@{-},
(-0.1,4.5);(-0.1,7.5) **@{-},
(0,4.5);(0,7.5) **@{-},
(0.1,4.5);(0.1,7.5) **@{-},
(-0.1,-3.5);(-0.1,-0.5) **@{-},
(0,-3.5);(0,-0.5) **@{-},
(0.1,-3.5);(0.1,-0.5) **@{-},
(-1,6.4)*{\ulcorner}, (-1.4,-8)*{\urcorner}, (-3,-1.2)*{\llcorner}, 
 (1,-3)*{\lrcorner},
(-2,4);(2,4) **\crv{(0,1)}, (-2,0);(2,0) **\crv{(0,3)},
(-2,-4);(2,-4) **\crv{(0,-7)}, (-2,-8);(2,-8) **\crv{(0,-5)},
 \endxy~\bigg)_N
 + t^2 {\mathbf K}\bigg(~\xy 
 (-8,8);(-2,8) **@{-}, 
(2,8);(8,8) **@{-}, 
(-8,8);(-8,2) **@{-}, 
(8,8);(8,2) **@{-}, 
(-8,-8);(-2,-8) **@{-}, 
(2,-8);(8,-8) **@{-}, 
(-8,-8);(-8,-2) **@{-}, 
(8,-8);(8,-2) **@{-},
(-2,8);(2,4) **@{-}, 
(-2,4);(2,8) **@{-},%
(-2,8);(2,4) **@{-}, (-2,4);(2,8) **@{-},
(0.4,6.4);(2,8) **@{-}, (2,0);(4,-2) **@{-}, (-2,0);(-4,-2) **@{-},
(-2,-4);(1.4,-0.4) **@{-}, (2.4,0.4);(4,2) **@{-},
(2,-4);(-1.7,-0.3) **@{-},  
(-0.4,-1.6);(-1.7,-0.3) **@{-},  (-2.4,0.4);(-4,2) **@{-},
(-8,2);(-6.4,0.4) **@{-}, (-5.6,-0.4);(-4,-2) **@{-}, (-8,-2);(-4,2) **@{-},
(5.6,-0.4);(4,-2) **@{-}, (8,2);(6.6,0.4) **@{-},
(8,-2);(4,2) **@{-},
(-0.1,4.5);(-0.1,7.5) **@{-},
(0,4.5);(0,7.5) **@{-},
(0.1,4.5);(0.1,7.5) **@{-},
(-0.1,-3.5);(-0.1,-0.5) **@{-},
(0,-3.5);(0,-0.5) **@{-},
(0.1,-3.5);(0.1,-0.5) **@{-},
(-1,6.4)*{\ulcorner}, (-1.4,-8)*{\urcorner}, (-3,-1.2)*{\llcorner}, 
(1,-3)*{\lrcorner},
(-2,4);(2,4) **\crv{(0,1)}, (-2,0);(2,0) **\crv{(0,3)},
(-2,-4);(-2,-8) **\crv{(1,-6)}, (2,-8);(2,-4) **\crv{(-1,-6)},
 \endxy~\bigg)_N+\\
&t^2 {\mathbf K}\bigg(~\xy 
(-8,8);(-2,8) **@{-}, 
(2,8);(8,8) **@{-}, 
(-8,8);(-8,2) **@{-}, 
(8,8);(8,2) **@{-}, 
(-8,-8);(-2,-8) **@{-}, 
(2,-8);(8,-8) **@{-}, 
(-8,-8);(-8,-2) **@{-}, 
(8,-8);(8,-2) **@{-},
(-2,8);(2,4) **@{-}, 
(-2,4);(2,8) **@{-},%
(4,-2);(5.6,-0.4) **@{-},
(6.6,0.4);(8,2) **@{-},
(-2,8);(2,4) **@{-}, (-2,4);(2,8) **@{-},
(0.4,6.4);(2,8) **@{-}, (2,0);(4,-2) **@{-}, (-2,0);(-4,-2) **@{-},
(-2,-4);(1.4,-0.4) **@{-}, (2.4,0.4);(4,2) **@{-},
(2,-4);(-1.7,-0.3) **@{-},  
(-0.4,-1.6);(-1.7,-0.3) **@{-},  (-2.4,0.4);(-4,2) **@{-},
(-8,2);(-6.4,0.4) **@{-}, (-5.6,-0.4);(-4,-2) **@{-}, (-8,-2);(-4,2) **@{-},
(8,-2);(4,2) **@{-},
(-0.1,4.5);(-0.1,7.5) **@{-},
(0,4.5);(0,7.5) **@{-},
(0.1,4.5);(0.1,7.5) **@{-},
(-0.1,-3.5);(-0.1,-0.5) **@{-},
(0,-3.5);(0,-0.5) **@{-},
(0.1,-3.5);(0.1,-0.5) **@{-},
(-1,6.4)*{\ulcorner}, (-1.4,-8)*{\urcorner}, (-3,-1.2)*{\llcorner}, 
(1,-3)*{\lrcorner},
(-2,4);(-2,0) **\crv{(1,2)}, (2,0);(2,4) **\crv{(-1,2)},
(-2,-4);(2,-4) **\crv{(0,-7)}, (-2,-8);(2,-8) **\crv{(0,-5)},
 ~\endxy\bigg)_N
 ~+t^2 {\mathbf K}\bigg(~\xy 
 (-8,8);(-2,8) **@{-}, 
(2,8);(8,8) **@{-}, 
(-8,8);(-8,2) **@{-}, 
(8,8);(8,2) **@{-}, 
(-8,-8);(-2,-8) **@{-}, 
(2,-8);(8,-8) **@{-}, 
(-8,-8);(-8,-2) **@{-}, 
(8,-8);(8,-2) **@{-},
(-2,8);(2,4) **@{-}, 
(-2,4);(2,8) **@{-},%
(-2,8);(2,4) **@{-}, (-2,4);(2,8) **@{-},
(0.4,6.4);(2,8) **@{-}, (2,0);(4,-2) **@{-}, (-2,0);(-4,-2) **@{-},
(-2,-4);(1.4,-0.4) **@{-}, (2.4,0.4);(4,2) **@{-},
(2,-4);(-1.7,-0.3) **@{-}, 
(-0.4,-1.6);(-1.7,-0.3) **@{-},  (-2.4,0.4);(-4,2) **@{-},
(-8,2);(-6.4,0.4) **@{-}, (-5.6,-0.4);(-4,-2) **@{-}, (-8,-2);(-4,2) **@{-},
(5.6,-0.4);(4,-2) **@{-}, (8,2);(6.6,0.4) **@{-},
(8,-2);(4,2) **@{-},
(-0.1,4.5);(-0.1,7.5) **@{-},
(0,4.5);(0,7.5) **@{-},
(0.1,4.5);(0.1,7.5) **@{-},
(-0.1,-3.5);(-0.1,-0.5) **@{-},
(0,-3.5);(0,-0.5) **@{-},
(0.1,-3.5);(0.1,-0.5) **@{-},
(-1,6.4)*{\ulcorner}, (-1.4,-8)*{\urcorner}, (-3,-1.2)*{\llcorner}, 
(1,-3)*{\lrcorner},
(-2,4);(-2,0) **\crv{(1,2)}, (2,0);(2,4) **\crv{(-1,2)},
(-2,-4);(-2,-8) **\crv{(1,-6)}, (2,-8);(2,-4) **\crv{(-1,-6)},
 \endxy~\bigg)_N
 =t^2
 {\mathbf K}\bigg(~\xy 
 (-8,8);(-2,8) **@{-}, 
(2,8);(8,8) **@{-}, 
(-8,8);(-8,2) **@{-}, 
(8,8);(8,2) **@{-}, 
(-8,-8);(-2,-8) **@{-}, 
(2,-8);(8,-8) **@{-}, 
(-8,-8);(-8,-2) **@{-}, 
(8,-8);(8,-2) **@{-},
(-2,8);(2,8) **@{-}, 
(2,0);(4,-2) **@{-}, 
(-2,0);(-4,-2) **@{-},
(2.4,0.4);(4,2) **@{-},%
(-2.4,0.4);(-4,2) **@{-},
(-8,2);(-6.4,0.4) **@{-}, (-5.6,-0.4);(-4,-2) **@{-}, (-8,-2);(-4,2) **@{-},
(5.6,-0.4);(4,-2) **@{-}, (8,2);(6.6,0.4) **@{-},
(8,-2);(4,2) **@{-},
(-1.4,-8)*{\urcorner}, (-3,-1.2)*{\llcorner}, 
(-2,0);(2,0) **\crv{(0,3)},
(-1.5,-1);(1.5,-1) **\crv{(0,-3)}, 
(-2,-8);(2,-8) **\crv{(0,-5)},
 \endxy~\bigg)_N+\\
&t^2 {\mathbf K}\bigg(~\xy  
 (-8,8);(-2,8) **@{-}, 
(2,8);(8,8) **@{-}, 
(-8,8);(-8,2) **@{-}, 
(8,8);(8,2) **@{-}, 
(-8,-8);(-2,-8) **@{-}, 
(2,-8);(8,-8) **@{-}, 
(-8,-8);(-8,-2) **@{-}, 
(8,-8);(8,-2) **@{-},
(-2,8);(2,8) **@{-}, 
(2,0);(4,-2) **@{-}, 
(-2,0);(-4,-2) **@{-},
(-2,-4);(1.4,-0.4) **@{-}, 
(2.4,0.4);(4,2) **@{-},
(2,-4);(-1.7,-0.3) **@{-},  
(-2.4,0.4);(-4,2) **@{-},
(-8,2);(-6.4,0.4) **@{-}, (-5.6,-0.4);(-4,-2) **@{-}, (-8,-2);(-4,2) **@{-},
(5.6,-0.4);(4,-2) **@{-}, (8,2);(6.6,0.4) **@{-},
(8,-2);(4,2) **@{-},
(-0.1,-3.5);(-0.1,-0.5) **@{-},
(0,-3.5);(0,-0.5) **@{-},
(0.1,-3.5);(0.1,-0.5) **@{-},
(-1.4,-8)*{\urcorner}, (-3,-1.2)*{\llcorner}, 
(1,-3)*{\lrcorner}, 
(-2,0);(2,0) **\crv{(0,3)},
(-2,-4);(-2,-8) **\crv{(1,-6)}, 
(2,-8);(2,-4) **\crv{(-1,-6)},
 \endxy~\bigg)_N
 +t^2 {\mathbf K}\bigg(~\xy  
(-8,8);(-2,8) **@{-}, 
(2,8);(8,8) **@{-}, 
(-8,8);(-8,2) **@{-}, 
(8,8);(8,2) **@{-}, 
(-8,-8);(-2,-8) **@{-}, 
(2,-8);(8,-8) **@{-}, 
(-8,-8);(-8,-2) **@{-}, 
(8,-8);(8,-2) **@{-},
(-2,8);(2,4) **@{-}, (-2,4);(2,8) **@{-},
(0.4,6.4);(2,8) **@{-}, (2,0);(4,-2) **@{-}, (-2,0);(-4,-2) **@{-},
(2.4,0.4);(4,2) **@{-},%
(-2.4,0.4);(-4,2) **@{-},
(-8,2);(-6.4,0.4) **@{-}, (-5.6,-0.4);(-4,-2) **@{-}, (-8,-2);(-4,2) **@{-},
(5.6,-0.4);(4,-2) **@{-}, (8,2);(6.6,0.4) **@{-},
(8,-2);(4,2) **@{-},
(-0.1,4.5);(-0.1,7.5) **@{-},
(0,4.5);(0,7.5) **@{-},
(0.1,4.5);(0.1,7.5) **@{-},
(-1,6.4)*{\ulcorner}, 
(-1.4,-8)*{\urcorner}, (-3,-1.2)*{\llcorner}, 
(-2,4);(-2,0) **\crv{(1,2)}, 
(2,0);(2,4) **\crv{(-1,2)}, 
(-2,-8);(2,-8) **\crv{(0,-5)},
(-1.5,-1);(1.5,-1) **\crv{(0,-3)},
 ~\endxy\bigg)_N
 +t^2 {\mathbf K}\bigg(~\xy  
 (-8,8);(-2,8) **@{-}, 
(2,8);(8,8) **@{-}, 
(-8,8);(-8,2) **@{-}, 
(8,8);(8,2) **@{-}, 
(-8,-8);(-2,-8) **@{-}, 
(2,-8);(8,-8) **@{-}, 
(-8,-8);(-8,-2) **@{-}, 
(8,-8);(8,-2) **@{-},
(-2,8);(2,4) **@{-}, (-2,4);(2,8) **@{-},
(0.4,6.4);(2,8) **@{-}, (2,0);(4,-2) **@{-}, (-2,0);(-4,-2) **@{-},
(-2,-4);(1.4,-0.4) **@{-}, (2.4,0.4);(4,2) **@{-},
(2,-4);(-1.7,-0.3) **@{-},  
(-0.4,-1.6);(-1.7,-0.3) **@{-},  (-2.4,0.4);(-4,2) **@{-},
(-8,2);(-6.4,0.4) **@{-}, (-5.6,-0.4);(-4,-2) **@{-}, (-8,-2);(-4,2) **@{-},
(5.6,-0.4);(4,-2) **@{-}, (8,2);(6.6,0.4) **@{-},
(8,-2);(4,2) **@{-},
(-0.1,4.5);(-0.1,7.5) **@{-},
(0,4.5);(0,7.5) **@{-},
(0.1,4.5);(0.1,7.5) **@{-},
(-0.1,-3.5);(-0.1,-0.5) **@{-},
(0,-3.5);(0,-0.5) **@{-},
(0.1,-3.5);(0.1,-0.5) **@{-},
(-1,6.4)*{\ulcorner}, (-1.4,-8)*{\urcorner}, (-3,-1.2)*{\llcorner}, 
(1,-3)*{\lrcorner},
(-2,4);(-2,0) **\crv{(1,2)}, (2,0);(2,4) **\crv{(-1,2)},
(-2,-4);(-2,-8) **\crv{(1,-6)}, (2,-8);(2,-4) **\crv{(-1,-6)},
 \endxy~\bigg)_N\\
 =&t^2{\mathbf K}(O^2)_N + t^3{\mathbf K}(2^2_1\sharp 2^{2*}_1)_N + t^3{\mathbf K}(O^2)_N+t^3{\mathbf K}(2^2_1\sharp 2^{2*}_1)_N + t^3{\mathbf K}(O^2)_N\\
&+t^4{\mathbf K}(2^2_1\sqcup 2^{2*}_1)_N + t^4{\mathbf K}(2^2_1\sharp 2^{2*}_1)_N + t^4{\mathbf K}(2^2_1\sharp 2^{2*}_1)_N+t^4{\mathbf K}(O^2)_N\\
 =&(t^2+2t^3+t^4)(t^{-1}-1) + \Big(2(t^4+t^3) + t^4(t^{-1}-1)\Big)(z(t)^{10}+z(t)^{2})(\overline{z}(t)^{10}+\overline{z}(t)^{2})\\
=&(6t^4+15t^3+3t^2-11t+3)t^{-1}.
 \end{align*}
Hence for each integer $n \geq 2$,
\begin{align*}
{\mathbf K}_{2n-1}(8^{1,1}_1)_N &=\overline{{\mathbf K}(8^{1,1}_2;n)_N}=[6n^4+15n^3+3n^2-11n+3]n^{-1}.
\end{align*}
Further, forgetting the orientation on $8^{1,1}_1$, we see that $\ll 8^{1,1}_1\gg_N$ $=\ll \widetilde{8^{1,1}_1}\gg$ and therefore
${\mathbf K}_{2n-1}(\widetilde{8^{1,1}_1})={\mathbf K}_{2n-1}(8^{1,1}_1)_N$ for all $n \geq 2.$
\end{example}

\begin{example}\label{examp-10^{1}_1} 
Let $10^{1}_1$ be the spun torus of the trefoil in Yoshikawa's table with the orientation indicated below. By Theorem \ref{thm-skein-rel-kb-1} together with (\ref{nkbp-exmps}) and (\ref{skein-rel-exmp}), we obtain 
\begin{align*}
&{\mathbf K}\bigg(~\xy 
(-12,8);(-2,8) **@{-}, 
(2,8);(12,8) **@{-}, 
(-12,8);(-12,2) **@{-}, 
(12,8);(12,2) **@{-}, 
(-12,-8);(-2,-8) **@{-}, 
(2,-8);(12,-8) **@{-}, 
(-12,-8);(-12,-2) **@{-}, 
(12,-8);(12,-2) **@{-},
(-2,8);(2,4) **@{-}, 
(-2,4);(2,8) **@{-},%
(-2,4);(4,-2) **@{-}, 
(2,4);(-4,-2) **@{-}, 
(-2,-8);(2,-4) **@{-}, 
(-2,-4);(2,-8) **@{-},
(-2,-4);(1.4,-0.4) **@{-}, 
(2.4,0.4);(4,2) **@{-},
(2,-4);(-1.7,-0.3) **@{-}, %
(-2.4,0.4);(-4,2) **@{-},
(-12,2);(-10.4,0.4) **@{-}, 
(-9.6,-0.4);(-8,-2) **@{-}, 
(-12,-2);(-8,2) **@{-},
(-8,2);(-6.4,0.4) **@{-}, 
(-5.6,-0.4);(-4,-2) **@{-}, 
(-8,-2);(-4,2) **@{-},
(12,2);(10.4,0.4) **@{-}, 
(9.6,-0.4);(8,-2) **@{-},
(12,-2);(8,2) **@{-},
(5.6,-0.4);(4,-2) **@{-}, 
(8,2);(6.6,0.4) **@{-},
(8,-2);(4,2) **@{-},
(-0.1,4.5);(-0.1,7.5) **@{-},
(0,4.5);(0,7.5) **@{-},
(0.1,4.5);(0.1,7.5) **@{-},
(-0.1,-3.5);(-0.1,-0.5) **@{-},
(0,-3.5);(0,-0.5) **@{-},
(0.1,-3.5);(0.1,-0.5) **@{-},
(-1.5,1.9);(1.5,1.9) **@{-},
(-1.5,2);(1.5,2) **@{-},
(-1.5,2.1);(1.5,2.1) **@{-},
(-1.5,-5.9);(1.5,-5.9) **@{-},
(-1.5,-6);(1.5,-6) **@{-},
(-1.5,-6.1);(1.5,-6.1) **@{-},
(-1,6.4)*{\ulcorner}, (-1.4,-8)*{\urcorner}, (-3,-1.2)*{\llcorner}, (-11,-1.2)*{\llcorner}, (-8.9,-1.2)*{\lrcorner}, (11,0.5)*{\urcorner}, (8.9,0.5)*{\ulcorner}, (1,-3)*{\lrcorner},
(0.5,-12.2) *{10^{1}_1},
 \endxy~\bigg)_N 
 = t^2 {\mathbf K}\bigg(~\xy (-12,8);(-2,8) **@{-}, (2,8);(12,8) **@{-}, (-12,8);(-12,2) **@{-}, (12,8);(12,2) **@{-}, (-12,-8);(-2,-8) **@{-}, (2,-8);(12,-8) **@{-}, (-12,-8);(-12,-2) **@{-}, (12,-8);(12,-2) **@{-},
(-2,8);(2,4) **@{-}, (-2,4);(2,8) **@{-},
(0.4,6.4);(2,8) **@{-}, (2,0);(4,-2) **@{-}, (-2,0);(-4,-2) **@{-},
(-2,-4);(1.4,-0.4) **@{-}, (2.4,0.4);(4,2) **@{-},
(2,-4);(-1.7,-0.3) **@{-},  
(-0.4,-1.6);(-1.7,-0.3) **@{-},  (-2.4,0.4);(-4,2) **@{-},
(-12,2);(-10.4,0.4) **@{-}, (-9.6,-0.4);(-8,-2) **@{-}, (-12,-2);(-8,2) **@{-},
(-8,2);(-6.4,0.4) **@{-}, (-5.6,-0.4);(-4,-2) **@{-}, (-8,-2);(-4,2) **@{-},
(12,2);(10.4,0.4) **@{-}, (9.6,-0.4);(8,-2) **@{-},
(12,-2);(8,2) **@{-},
(5.6,-0.4);(4,-2) **@{-}, (8,2);(6.6,0.4) **@{-},
(8,-2);(4,2) **@{-},
(-0.1,4.5);(-0.1,7.5) **@{-},
(0,4.5);(0,7.5) **@{-},
(0.1,4.5);(0.1,7.5) **@{-},
(-0.1,-3.5);(-0.1,-0.5) **@{-},
(0,-3.5);(0,-0.5) **@{-},
(0.1,-3.5);(0.1,-0.5) **@{-},
(-1,6.4)*{\ulcorner}, (-1.4,-8)*{\urcorner}, (-3,-1.2)*{\llcorner}, (-11,-1.2)*{\llcorner}, (-8.9,-1.2)*{\lrcorner}, (11,0.5)*{\urcorner}, (8.9,0.5)*{\ulcorner}, (1,-3)*{\lrcorner},
(-2,4);(2,4) **\crv{(0,1)}, (-2,0);(2,0) **\crv{(0,3)},
(-2,-4);(2,-4) **\crv{(0,-7)}, (-2,-8);(2,-8) **\crv{(0,-5)},
 \endxy~\bigg)_N
 + t^2 {\mathbf K}\bigg(~\xy (-12,8);(-2,8) **@{-}, (2,8);(12,8) **@{-}, (-12,8);(-12,2) **@{-}, (12,8);(12,2) **@{-}, (-12,-8);(-2,-8) **@{-}, (2,-8);(12,-8) **@{-}, (-12,-8);(-12,-2) **@{-}, (12,-8);(12,-2) **@{-},
(-2,8);(2,4) **@{-}, (-2,4);(2,8) **@{-},
(0.4,6.4);(2,8) **@{-}, (2,0);(4,-2) **@{-}, (-2,0);(-4,-2) **@{-},
(-2,-4);(1.4,-0.4) **@{-}, (2.4,0.4);(4,2) **@{-},
(2,-4);(-1.7,-0.3) **@{-},  
(-0.4,-1.6);(-1.7,-0.3) **@{-},  (-2.4,0.4);(-4,2) **@{-},
(-12,2);(-10.4,0.4) **@{-}, (-9.6,-0.4);(-8,-2) **@{-}, (-12,-2);(-8,2) **@{-},
(-8,2);(-6.4,0.4) **@{-}, (-5.6,-0.4);(-4,-2) **@{-}, (-8,-2);(-4,2) **@{-},
(12,2);(10.4,0.4) **@{-}, (9.6,-0.4);(8,-2) **@{-},
(12,-2);(8,2) **@{-},
(5.6,-0.4);(4,-2) **@{-}, (8,2);(6.6,0.4) **@{-},
(8,-2);(4,2) **@{-},
(-0.1,4.5);(-0.1,7.5) **@{-},
(0,4.5);(0,7.5) **@{-},
(0.1,4.5);(0.1,7.5) **@{-},
(-0.1,-3.5);(-0.1,-0.5) **@{-},
(0,-3.5);(0,-0.5) **@{-},
(0.1,-3.5);(0.1,-0.5) **@{-},
(-1,6.4)*{\ulcorner}, (-1.4,-8)*{\urcorner}, (-3,-1.2)*{\llcorner}, (-11,-1.2)*{\llcorner}, (-8.9,-1.2)*{\lrcorner}, (11,0.5)*{\urcorner}, (8.9,0.5)*{\ulcorner}, (1,-3)*{\lrcorner},
(-2,4);(2,4) **\crv{(0,1)}, (-2,0);(2,0) **\crv{(0,3)},
(-2,-4);(-2,-8) **\crv{(1,-6)}, (2,-8);(2,-4) **\crv{(-1,-6)},
 \endxy~\bigg)_N+\\
&t^2 {\mathbf K}\bigg(~\xy (-12,8);(-2,8) **@{-}, (2,8);(12,8) **@{-}, (-12,8);(-12,2) **@{-}, (12,8);(12,2) **@{-}, (-12,-8);(-2,-8) **@{-}, (2,-8);(12,-8) **@{-}, (-12,-8);(-12,-2) **@{-}, (12,-8);(12,-2) **@{-},
(-2,8);(2,4) **@{-}, (-2,4);(2,8) **@{-},
(0.4,6.4);(2,8) **@{-}, (2,0);(4,-2) **@{-}, (-2,0);(-4,-2) **@{-},
(-2,-4);(1.4,-0.4) **@{-}, (2.4,0.4);(4,2) **@{-},
(2,-4);(-1.7,-0.3) **@{-}, 
(-0.4,-1.6);(-1.7,-0.3) **@{-},  (-2.4,0.4);(-4,2) **@{-},
(-12,2);(-10.4,0.4) **@{-}, (-9.6,-0.4);(-8,-2) **@{-}, (-12,-2);(-8,2) **@{-},
(-8,2);(-6.4,0.4) **@{-}, (-5.6,-0.4);(-4,-2) **@{-}, (-8,-2);(-4,2) **@{-},
(12,2);(10.4,0.4) **@{-}, (9.6,-0.4);(8,-2) **@{-},
(12,-2);(8,2) **@{-},
(5.6,-0.4);(4,-2) **@{-}, (8,2);(6.6,0.4) **@{-},
(8,-2);(4,2) **@{-},
(-0.1,4.5);(-0.1,7.5) **@{-},
(0,4.5);(0,7.5) **@{-},
(0.1,4.5);(0.1,7.5) **@{-},
(-0.1,-3.5);(-0.1,-0.5) **@{-},
(0,-3.5);(0,-0.5) **@{-},
(0.1,-3.5);(0.1,-0.5) **@{-},
(-1,6.4)*{\ulcorner}, (-1.4,-8)*{\urcorner}, (-3,-1.2)*{\llcorner}, (-11,-1.2)*{\llcorner}, (-8.9,-1.2)*{\lrcorner}, (11,0.5)*{\urcorner}, (8.9,0.5)*{\ulcorner}, (1,-3)*{\lrcorner},
(-2,4);(-2,0) **\crv{(1,2)}, (2,0);(2,4) **\crv{(-1,2)},
(-2,-4);(2,-4) **\crv{(0,-7)}, (-2,-8);(2,-8) **\crv{(0,-5)},
 \endxy~\bigg)_N
 ~+t^2 {\mathbf K}\bigg(~\xy (-12,8);(-2,8) **@{-}, (2,8);(12,8) **@{-}, (-12,8);(-12,2) **@{-}, (12,8);(12,2) **@{-}, (-12,-8);(-2,-8) **@{-}, (2,-8);(12,-8) **@{-}, (-12,-8);(-12,-2) **@{-}, (12,-8);(12,-2) **@{-},
(-2,8);(2,4) **@{-}, (-2,4);(2,8) **@{-},
(0.4,6.4);(2,8) **@{-}, (2,0);(4,-2) **@{-}, (-2,0);(-4,-2) **@{-},
(-2,-4);(1.4,-0.4) **@{-}, (2.4,0.4);(4,2) **@{-},
(2,-4);(-1.7,-0.3) **@{-},  
(-0.4,-1.6);(-1.7,-0.3) **@{-},  (-2.4,0.4);(-4,2) **@{-},
(-12,2);(-10.4,0.4) **@{-}, (-9.6,-0.4);(-8,-2) **@{-}, (-12,-2);(-8,2) **@{-},
(-8,2);(-6.4,0.4) **@{-}, (-5.6,-0.4);(-4,-2) **@{-}, (-8,-2);(-4,2) **@{-},
(12,2);(10.4,0.4) **@{-}, (9.6,-0.4);(8,-2) **@{-},
(12,-2);(8,2) **@{-},
(5.6,-0.4);(4,-2) **@{-}, (8,2);(6.6,0.4) **@{-},
(8,-2);(4,2) **@{-},
(-0.1,4.5);(-0.1,7.5) **@{-},
(0,4.5);(0,7.5) **@{-},
(0.1,4.5);(0.1,7.5) **@{-},
(-0.1,-3.5);(-0.1,-0.5) **@{-},
(0,-3.5);(0,-0.5) **@{-},
(0.1,-3.5);(0.1,-0.5) **@{-},
(-1,6.4)*{\ulcorner}, (-1.4,-8)*{\urcorner}, (-3,-1.2)*{\llcorner}, (-11,-1.2)*{\llcorner}, (-8.9,-1.2)*{\lrcorner}, (11,0.5)*{\urcorner}, (8.9,0.5)*{\ulcorner}, (1,-3)*{\lrcorner},
(-2,4);(-2,0) **\crv{(1,2)}, (2,0);(2,4) **\crv{(-1,2)},
(-2,-4);(-2,-8) **\crv{(1,-6)}, (2,-8);(2,-4) **\crv{(-1,-6)},
 \endxy~\bigg)_N=t^2
 {\mathbf K}\bigg(~\xy (-12,8);(-2,8) **@{-}, (2,8);(12,8) **@{-}, (-12,8);(-12,2) **@{-}, (12,8);(12,2) **@{-}, (-12,-8);(-2,-8) **@{-}, (2,-8);(12,-8) **@{-}, (-12,-8);(-12,-2) **@{-}, (12,-8);(12,-2) **@{-},
(-2,8);(2,8) **@{-}, 
(2,0);(4,-2) **@{-}, 
(-2,0);(-4,-2) **@{-},
(2.4,0.4);(4,2) **@{-},
(-2.4,0.4);(-4,2) **@{-},
(-12,2);(-10.4,0.4) **@{-}, (-9.6,-0.4);(-8,-2) **@{-}, (-12,-2);(-8,2) **@{-},
(-8,2);(-6.4,0.4) **@{-}, (-5.6,-0.4);(-4,-2) **@{-}, (-8,-2);(-4,2) **@{-},
(12,2);(10.4,0.4) **@{-}, (9.6,-0.4);(8,-2) **@{-},
(12,-2);(8,2) **@{-},
(5.6,-0.4);(4,-2) **@{-}, (8,2);(6.6,0.4) **@{-},
(8,-2);(4,2) **@{-},
(-1.4,-8)*{\urcorner}, (-3,-1.2)*{\llcorner}, (-11,-1.2)*{\llcorner}, 
(11,0.5)*{\urcorner}, (8.9,0.5)*{\ulcorner}, 
(-2,0);(2,0) **\crv{(0,3)},
(-1.5,-1);(1.5,-1) **\crv{(0,-3)}, 
(-2,-8);(2,-8) **\crv{(0,-5)},
\endxy~\bigg)_N+\\
&t^2 {\mathbf K}\bigg(~\xy (-12,8);(-2,8) **@{-}, (2,8);(12,8) **@{-}, (-12,8);(-12,2) **@{-}, (12,8);(12,2) **@{-}, (-12,-8);(-2,-8) **@{-}, (2,-8);(12,-8) **@{-}, (-12,-8);(-12,-2) **@{-}, (12,-8);(12,-2) **@{-},
(-2,8);(2,8) **@{-}, 
(2,0);(4,-2) **@{-}, 
(-2,0);(-4,-2) **@{-},
(-2,-4);(1.4,-0.4) **@{-}, 
(2.4,0.4);(4,2) **@{-},
(2,-4);(-1.7,-0.3) **@{-},   
(-2.4,0.4);(-4,2) **@{-},
(-12,2);(-10.4,0.4) **@{-}, (-9.6,-0.4);(-8,-2) **@{-}, (-12,-2);(-8,2) **@{-},
(-8,2);(-6.4,0.4) **@{-}, (-5.6,-0.4);(-4,-2) **@{-}, (-8,-2);(-4,2) **@{-},
(12,2);(10.4,0.4) **@{-}, (9.6,-0.4);(8,-2) **@{-},
(12,-2);(8,2) **@{-},
(5.6,-0.4);(4,-2) **@{-}, (8,2);(6.6,0.4) **@{-},
(8,-2);(4,2) **@{-},
(-0.1,-3.5);(-0.1,-0.5) **@{-},
(0,-3.5);(0,-0.5) **@{-},
(0.1,-3.5);(0.1,-0.5) **@{-},
(-1.4,-8)*{\urcorner}, (-3,-1.2)*{\llcorner}, (-11,-1.2)*{\llcorner},  
(11,0.5)*{\urcorner}, (8.9,0.5)*{\ulcorner}, (1,-3)*{\lrcorner}, 
(-2,0);(2,0) **\crv{(0,3)},
(-2,-4);(-2,-8) **\crv{(1,-6)}, 
(2,-8);(2,-4) **\crv{(-1,-6)},
 \endxy~\bigg)_N+
 t^2 {\mathbf K}\bigg(~\xy (-12,8);(-2,8) **@{-}, (2,8);(12,8) **@{-}, (-12,8);(-12,2) **@{-}, (12,8);(12,2) **@{-}, (-12,-8);(-2,-8) **@{-}, (2,-8);(12,-8) **@{-}, (-12,-8);(-12,-2) **@{-}, (12,-8);(12,-2) **@{-},
(-2,8);(2,4) **@{-}, (-2,4);(2,8) **@{-},
(0.4,6.4);(2,8) **@{-}, (2,0);(4,-2) **@{-}, (-2,0);(-4,-2) **@{-},
(2.4,0.4);(4,2) **@{-}, 
(-2.4,0.4);(-4,2) **@{-},
(-12,2);(-10.4,0.4) **@{-}, (-9.6,-0.4);(-8,-2) **@{-}, (-12,-2);(-8,2) **@{-},
(-8,2);(-6.4,0.4) **@{-}, (-5.6,-0.4);(-4,-2) **@{-}, (-8,-2);(-4,2) **@{-},
(12,2);(10.4,0.4) **@{-}, (9.6,-0.4);(8,-2) **@{-},
(12,-2);(8,2) **@{-},
(5.6,-0.4);(4,-2) **@{-}, (8,2);(6.6,0.4) **@{-},
(8,-2);(4,2) **@{-},
(-0.1,4.5);(-0.1,7.5) **@{-},
(0,4.5);(0,7.5) **@{-},
(0.1,4.5);(0.1,7.5) **@{-},
(-1,6.4)*{\ulcorner}, 
(-1.4,-8)*{\urcorner}, (-3,-1.2)*{\llcorner}, (-11,-1.2)*{\llcorner}, (-8.9,-1.2)*{\lrcorner}, 
(11,0.5)*{\urcorner}, (8.9,0.5)*{\ulcorner},
(-2,4);(-2,0) **\crv{(1,2)}, 
(2,0);(2,4) **\crv{(-1,2)}, 
(-2,-8);(2,-8) **\crv{(0,-5)},
(-1.5,-1);(1.5,-1) **\crv{(0,-3)},
 \endxy~\bigg)_N
 +t^2 {\mathbf K}\bigg(~\xy (-12,8);(-2,8) **@{-}, (2,8);(12,8) **@{-}, (-12,8);(-12,2) **@{-}, (12,8);(12,2) **@{-}, (-12,-8);(-2,-8) **@{-}, (2,-8);(12,-8) **@{-}, (-12,-8);(-12,-2) **@{-}, (12,-8);(12,-2) **@{-},
(-2,8);(2,4) **@{-}, (-2,4);(2,8) **@{-},
(0.4,6.4);(2,8) **@{-}, (2,0);(4,-2) **@{-}, (-2,0);(-4,-2) **@{-},
(-2,-4);(1.4,-0.4) **@{-}, (2.4,0.4);(4,2) **@{-},
(2,-4);(-1.7,-0.3) **@{-},  
(-0.4,-1.6);(-1.7,-0.3) **@{-},  (-2.4,0.4);(-4,2) **@{-},
(-12,2);(-10.4,0.4) **@{-}, (-9.6,-0.4);(-8,-2) **@{-}, (-12,-2);(-8,2) **@{-},
(-8,2);(-6.4,0.4) **@{-}, (-5.6,-0.4);(-4,-2) **@{-}, (-8,-2);(-4,2) **@{-},
(12,2);(10.4,0.4) **@{-}, (9.6,-0.4);(8,-2) **@{-},
(12,-2);(8,2) **@{-},
(5.6,-0.4);(4,-2) **@{-}, (8,2);(6.6,0.4) **@{-},
(8,-2);(4,2) **@{-},
(-0.1,4.5);(-0.1,7.5) **@{-},
(0,4.5);(0,7.5) **@{-},
(0.1,4.5);(0.1,7.5) **@{-},
(-0.1,-3.5);(-0.1,-0.5) **@{-},
(0,-3.5);(0,-0.5) **@{-},
(0.1,-3.5);(0.1,-0.5) **@{-},
(-1,6.4)*{\ulcorner}, (-1.4,-8)*{\urcorner}, (-3,-1.2)*{\llcorner}, (-11,-1.2)*{\llcorner}, (-8.9,-1.2)*{\lrcorner}, (11,0.5)*{\urcorner}, (8.9,0.5)*{\ulcorner}, (1,-3)*{\lrcorner},
(-2,4);(-2,0) **\crv{(1,2)}, (2,0);(2,4) **\crv{(-1,2)},
(-2,-4);(-2,-8) **\crv{(1,-6)}, (2,-8);(2,-4) **\crv{(-1,-6)},
 \endxy~\bigg)_N
 \end{align*}
 \begin{align*}
&=t^2{\mathbf K}(O^2)_N + t^3{\mathbf K}(3_1\sharp 3_1^*)_N + t^3{\mathbf K}(O^2)_N+t^3{\mathbf K}(3_1\sharp 3_1^*)_N + t^3{\mathbf K}(O^2)_N\\
&~~~~+t^4{\mathbf K}(3_1\sqcup 3_1^*)_N + t^4{\mathbf K}(3_1\sharp 3_1^*)_N + t^4{\mathbf K}(3_1\sharp 3_1^*)_N+t^4{\mathbf K}(O^2)_N\\
&=(t^2+2t^3+t^4)(t^{-1}-1) + \Big(2(t^4+t^3) + t^4(t^{-1}-1)\Big)\\
 &\hskip 2.20cm (-z(t)^{16}+z(t)^{12}+z(t)^4)(-\overline{z}(t)^{16}+\overline{z}(t)^{12}+\overline{z}(t)^{4}) \\
&=(8t^6+32t^5+32t^4-32t^3-18t^2+17t-3)t^{-3}.
 \end{align*}
Hence for each integer $n \geq 2$,
\begin{align*}
{\mathbf K}_{2n-1}(10^{1}_1)_N &=\overline{{\mathbf K}(10^{1}_1;n)_N}=[8n^6+32n^5+32n^4-32n^3-18n^2+17n-3]n^{-3}.
\end{align*}
Further, forgetting the orientation on $10^{1}_1$, we see that $\ll 10^{1}_1\gg_N$ $=\ll \widetilde{10^{1}_1}\gg$ and therefore
${\mathbf K}_{2n-1}(\widetilde{10^{1}_1})={\mathbf K}_{2n-1}(10^{1}_1)_N$ for all $n \geq 2.$
\end{example}

\bigskip

As a summary of the above discussion of examples, we give the following Table I of the first four invariants ${\mathbf K}_{2n-1}(\mathcal L)_N (n=2,3,4,5)$ of all unoriented surface-links $\mathcal L$ in Yoshikawa's table with ch-index $\leq 8$ and three more $2$-knots $9_1, 10_2$ and $10^1_1$.
\[
\begin{array}{|l|l|l|l|l|}
\hline
\text{Surface-link $\mathcal L$}&{\mathbf K}_3(\mathcal L)&{\mathbf K}_5(\mathcal L)&{\mathbf K}_7(\mathcal L)&{\mathbf K}_9(\mathcal L)\\
\hline
0_1&[1]&[1]&[1]&[1]\\
2^1_1&[1]&[1]&[1]&[1]\\
2^{-1}_1&[1]2^{-\frac{1}{2}}\sqrt{5}&[1]3^{-\frac{1}{2}}\sqrt{8}&[1]4^{-\frac{1}{2}}\sqrt{11}&[1]5^{-\frac{1}{2}}\sqrt{14}\\
6^{0,1}_1&[1]2^{-2}&[4]3^{-2}&[2]4^{-2}&[7]5^{-2}\\
7^{0,-2}_1&[2]2^{-5}&[1]2^{-1}3^{-4}&[1]2^{-1}4^{-4}&[4]2^{-1}5^{-4}\\
8_1&[1]2^{-4}&[1]3^{-4}&[4]4^{-4}&[4]5^{-4}\\
8^{1,1}_1&[2]2^{-1}&[1]&[4]4^{-1}&[5]5^{-1}\\
8^{-1,-1}_1&[1]2^{-4}&[1]3^{-4}&[4]4^{-4}&[4]5^{-4}\\
9_1&[2]2^{-7}&[4]3^{-7}&[2]4^{-7}&[1]5^{-7}\\
10_2&[1]2^{-6}&[4]3^{-6}&[1]4^{-6}&[1]5^{-6}\\
10^{1}_1&[2]2^{-3}&[4]3^{-2}&[1]4^{-3}&[8]5^{-3}\\
\hline
\end{array}
\]
\centerline{{\bf Table I.} The invariants ${\mathbf K}_{2n-1}(\mathcal L) (n=2,3,4,5)$ of surface-links $\mathcal L$}

\bigskip

We close this subsection with the following remarks come from Table I.

\begin{remark} 
\begin{itemize}
\item[(1)] The invariants ${\mathbf K}_{2n-1}(\mathcal L) (n=2,3,4,5)$ distinguish the spun $2$-knot $8_1$, the $2$-twist spun $2$-knot $10_2$ and the spun torus $10^1_1$ of the trefoil knot $3_1$.

\item[(2)] The invariants ${\mathbf K}_{2n-1}(\mathcal L) (n=2,3,4,5)$ distinguish two nonorientable surface-links $7^{0,-2}_1$ and $8^{-1,-1}_1$ of the same nonorientable total genus.

\item[(3)] The invariants ${\mathbf K}_{2n-1}(\mathcal L) (n=2,3,4,5)$ distinguish two component orientable surface-links $6^{0,1}_1$ and $8^{1,1}_1$, which have the same surface-link group $\mathbb Z \oplus \mathbb Z$ and so have the same Alexander ideal.

\item[(4)] The invariants ${\mathbf K}_{2n-1}(\mathcal L) (n=2,3,4,5)$ distinguish the ribbon $2$-knot $9_1$ associated with the knot $6_1$ from the spun $2$-knots $8_1, 10_2$ and $10_2$.

\item[(5)]) The $2$-knot $8_1$ and the nonorientable surface-link $8^{-1,-1}_1$ have the same invariant ${\mathbf K}_{2n-1}(\mathcal L) (n=2,3,4,5)$.

\item[(6)] ${\mathbf K}_5(6^{0,1}_1)_N =[4]3^{-2}={\mathbf K}_5(10^{1}_1)_N =[4]3^{-2}$. But, the other three invariants ${\mathbf K}_{3}(\mathcal L), {\mathbf K}_{7}(\mathcal L)$ and ${\mathbf K}_{9}(\mathcal L)$ distinguish two surface-links $6^{0,1}_1$ and $10^1_1$.
\end{itemize}
\end{remark}

\section*{Acknowlegements}
This work was supported by Basic Science Research Program through the National Research Foundation of Korea(NRF) funded by the Ministry of Education, Science and Technology (2013R1A1A2012446) and NRF-2016R1A2B4016029.


\end{document}